\documentclass[a4paper, 11pt]{amsart}
\usepackage{amsmath, amsthm, amscd, amssymb, amsfonts, amsxtra, amssymb, latexsym}
\usepackage{verbatim}
\usepackage{graphicx} 
\usepackage{enumerate}
\usepackage{mathrsfs}
\usepackage[mathscr]{euscript}
\usepackage{tikz}
\usepackage{caption}
\usepackage{subcaption}
\usepackage{float}

\usetikzlibrary{arrows.meta, positioning}

\usepackage{hyperref}
\hypersetup{colorlinks = true,	allcolors  = blue}

\usepackage{xcolor}

\newcommand{\sk}{\smallskip}
\newcommand{\msk}{\medskip}

\newcommand{\N}{\mathbb{N}}
\newcommand{\Z}{\mathbb{Z}}
\newcommand{\Q}{\mathbb{Q}}
\newcommand{\R}{\mathbb{R}}
\newcommand{\C}{\mathbb{C}}

\newcommand{\Tr}{\mathrm{Tr}}
\newcommand{\G}{\Gamma}

\numberwithin{equation}{section}

\newtheorem{thm}{Theorem}[section]
\newtheorem{prop}[thm]{Proposition}
\newtheorem{lem}[thm]{Lemma}
\newtheorem{coro}[thm]{Corollary}

\theoremstyle{definition}
\newtheorem{rem}[thm]{Remark}
\newtheorem{exam}[thm]{Example}
\newtheorem{defi}[thm]{Definition}

\begin{document} \sloppy
\title{On di-Cayley graphs and their spectrum}

\author{Paula M.\@ Chiapparoli, Ricardo A.\@ Podest\'a}

\dedicatory{\today}

\thanks{Partially supported by CONICET, FonCyT and SECyT-UNC}

\address{Ricardo A.\@ Podest\'a. 
	\newline 
	FaMAF--CIEM (CONICET), Universidad Nacional de C\'ordoba. 
	\newline 
	Av.\@ Medina Allende 2144, Ciudad Universitaria, (5000), C\'ordoba, Argentina. 
	\newline
	{\it E-mail: ricardo.podesta@unc.edu.ar}}

\address{Paula M.\@ Chiapparoli. 
	\newline 
	FaMAF--CIEM (CONICET), Universidad Nacional de C\'ordoba. 
	\newline 
	Av.\@ Medina Allende 2144, Ciudad Universitaria, (5000), C\'ordoba, Argentina.
	\newline 
	{\it E-mail: paula.chiapparoli@mi.unc.edu.ar}}

\maketitle
	
\begin{abstract}
Given a group $G$ and three subsets $S_\ell, S_r, S_m \subset G$, we consider di-Cayley graphs $DX(G;S_\ell,S_r,S_m)$ and di-Cayley sum graphs $DX^+(G;S_\ell,S_r,S_m)$, directed generalizations of the bi-Cayley (sum) graphs $BX(G;S_\ell,S_r,S_m)$ and $BX^+(G;S_\ell,S_r,S_m)$. We refer to these four kinds of graphs collectively as $X^*(G;S_\ell,S_r,S_m)$. 
First, we give the basic properties of these graphs and compute their adjacency matrices. Then, we obtain the eigenvalues of $X^*(G;S_\ell,S_r,S_m)$ in terms of the spectra of the associated Cayley graphs $X(G,S)$ with $S\in \{S_\ell,S_r,S_m\}$ in two ways, using adjacency matrices and using irreducible characters of $G$. 
\end{abstract}

\section{Introduction} \label{sec: intro}
This work deals with structural and spectral properties of certain bi-Cayley and di-Cayley graphs. In this Introduction,
we first recall the family of bi-Cayley graphs and introduce a directed generalization, the di-Cayley graphs, then we recall the basic spectral definitions for arbitrary graphs; and finally we summarize the main results.

\subsection*{Cayley and Cayley sum graphs}
We denote by $\G=(V,E)$ an arbitrary graph with vertex set $V$ and edge set $E$.
Let $G$ be a group with identity $e$ and $S$ a subset of $G$. A \textit{Cayley graph} 
	$$ X(G,S) $$ 
is the directed graph whose vertex set is $G$ and where $u,v \in G$ form a directed edge (or arc) from $u$ to $v$, 
if $vu^{-1}$ ($v-u$ if $G$ is abelian) is in $S$, that is $v=us$ for some $s\in S$. In symbols
\begin{equation} \label{eq: edges}
	u \sim v \qquad \Leftrightarrow \qquad (u,v) \in E \qquad \Leftrightarrow \qquad  
	vu^{-1} \in S.	
\end{equation}

Analogously, the \textit{Cayley sum graph} 
	$X^+(G,S)$ 
has the same vertex set $G$ but now $u,v \in G$ form a directed edge from $u$ to $v$ if $vu \in S$ ($v+u$ if $G$ is abelian). 
In both graphs, if $(u,v)$ and $(v,u)$ are directed edges, they are usually considered as a single undirected edge $\{u,v\}$.
In both cases, $S$ is usually called the \textit{connection set}. 
The graphs $X(G,S)$ and $X^+(G,S)$ are $|S|$-regular.
We will write 
	$$ X^*(G,S) $$ 
to refer to $X(G,S)$ and $X^+(G,S)$ simultaneously.

We will say that a graph is \textit{looped} if it has a single loop at every vertex and that it is \textit{looped at $T$} if it has single loops at the vertices in $T \subset G$. 
If $e \in S$ then $X(G,S)$ has a loop at each vertex $x$, since $xx^{-1}=e \in S$, while if $e \notin S$ then $X(G,S)$ has no loops at all.
However, $X^+(G,S)$ may contain loops either if $e \in S$ or not. There is a loop at $x$ in $X^+(G,S)$ if and only if
$x^2 \in S$ ($2x \in S$ if $G$ is abelian). 
For instance, there is a loop at $x$ if $x^2=e$ and $e\in S$ ($2x=0$ and $0\in S$). If $S$ is a subgroup, $X^+(G,S)$ is looped (at least) at $S$.

Notice that $X(G,S)$ is undirected if and only if $S$ is \textit{symmetric}, that is $S$ is closed under inversion 
\begin{equation} \label{eq: S=S^-1}
	S=S^{-1}
\end{equation} 
(this holds in particular if $S$ is a subgroup). 
Similarly, $X(G,S)$ is directed if and only if $S$ is \textit{antisymmetric}, that is 
\begin{equation} \label{eq: antisymmetric}	
	S \cap S^{-1} = \varnothing. 
\end{equation}
On the other hand, $X^+(G,S)$ is undirected if and only if for every $gh\in S$ we have that $hg \in S$. 
This happens if and only if 	
\begin{equation} \label{eq: normal}	
	gSg^{-1} =S 
\end{equation}
for every $g\in G$, i.e.\@ $S$ is closed under conjugation (hence, union of conjugacy classes). Such $S$ is also called a \textit{normal} subset. 
That is $X^+(G,S)$ is undirected if and only if $S$ is normal (this holds in particular if $S$ is a normal subgroup). 
Thus, if $S \lhd G$ is a normal subgroup, then both $X(G,S)$ and $X^+(G,S)$ are undirected.
Furthermore, $X^+(G,S)$ is directed if and only if $S$ is what we call \textit{antinormal}, i.e.
\begin{equation} \label{eq: antinormal}	
	S \cap N_G(S) =\varnothing,
\end{equation}
where 
$ N_G(S) = \{g\in G : gSg^{-1} =S \} $ 
is the normalizer of $S$ in $G$.
Of course, if $G$ is abelian then any $S$ is closed under conjugation.
A Cayley graph $X(G,S)$ is called \textit{abelian} if $G$ is abelian, \textit{symmetric} if $S$ is closed by inversion and \textit{normal} if $S$ is closed under conjugation, hence abelian implies normal. 
We will adopt the same definitions for Cayley sum graphs.

\subsection*{On bi-Cayley and di-Cayley (sum) graphs}
More generally, given a group $G $ and $S_\ell$, $S_r$ and $S_m$ subsets of $G$, the \textit{bi-Cayley graph} 				
	$$ BX(G;S_\ell,S_r,S_m) $$ 
is the graph whose vertex set is $G \times \{0,1\}$ and where the vertices $(u,i)$ and $(v,i)$ with $i \in \{0,1\}$ form an arc from $(u,i)$ to $(v,i)$ if and only if $i=0$ and
$vu^{-1}\in S_{\ell}$ or else if
$i=1$ and $vu^{-1} \in S_r$; while the vertices $(u,0)$ and $(v,1)$ form an undirected edge if and only if $vu^{-1}\in S_m$.
As usual, if there is an arc from $(u,i)$ to $(v,i)$ and an arc from $(v,i)$ to $(u,i)$, for $i\in \{0,1\}$, we consider the arcs $((u,i),(v,i))$ and $((v,i),(u,i))$ as a single undirected edge $\{(u,i),(v,i)\}$.
Similarly, the \textit{bi-Cayley sum graph} 
	$BX^+(G;S_{\ell},S_r,S_m)$ 
is defined as above, changing $vu^{-1}$ by $vu$ in every appearance.

\subsubsection*{Some history} 
The study of graphs admitting a semiregular automorphism group with exactly two orbits naturally extends the rich theory of vertex-transitive Cayley graphs. Originally termed \textit{semi-Cayley graphs}, these structures were pioneered by Marušič \cite{Marusic, Marusic2} in the late 1980s, who focused on vertex-transitive graphs of prime power order and strongly regular bicirculants. Shortly after, de Resmini and Jungnickel \cite{RJ} formalized the algebraic framework of semi-Cayley graphs, demonstrating their utility in constructing strongly regular graphs and drawing deep connections with finite geometries. Today, these structures are widely known as \textit{bi-Cayley graphs} and serve as a fundamental tool for generating highly symmetric networks that relax the strict vertex-transitivity requirement of traditional Cayley graphs.
	
In subsequent years, the literature expanded significantly to address the full characterization of their automorphism groups \cite{AT, ZF}, the isomorphism problem, and their spectral properties. Gao and Luo \cite{GL1} formally posed the isomorphism problem for bi-Cayley graphs, which was subsequently resolved for abelian groups by Liu and Feng \cite{LF} and Kovács and Kuzman \cite{KK}, and for cyclic groups by Kovács, Muzychuk, and Somlai \cite{KMS} using the theory of Schur rings. From a spectral perspective, the eigenvalues of semi-Cayley graphs over abelian groups were explicitly computed in \cite{GL}, revealing how the block structure of their adjacency matrices governs the overall spectrum. While these works focus exclusively on undirected connections between the two Cayley covers, in this paper we relax this condition to allow directed edges.

\subsubsection*{Di-Cayley (sum) graphs} 
Now, we introduce a generalization of these graphs allowing directed graphs between vertices of the form $(u,0)$ and $(v,1)$.
The \textit{di-Cayley graph} 
	$$ DX(G;S_\ell,S_r,S_m) $$ 
is the graph whose vertex set is $G \times \{0,1\}$ and 
where $(u,i)$ and $(v,j)$ with $i,j \in \{0,1\}$ form a directed edge from $(u,i)$ to $(v,j)$ if and only if any of the following four cases holds: ($a$) $i=j=0$ and $vu^{-1}\in S_{\ell}$, ($b$) $i=j=1$ and $vu^{-1} \in S_r$, ($c$) $i=0$, $j=1$ and $vu^{-1}\in S_m$, or ($d$) $i=1$, $j=0$ and $vu^{-1}\in S_m$. We adopt the same convention with double arcs as for bi-Cayley graphs. Namely, one arc and its reverse arc, if exist, are considered a single undirected edge between the vertices.
Similarly, the \textit{di-Cayley sum graph} 
	$DX^+(G;S_{\ell},S_r,S_m)$ 
is defined as above, by changing each $vu^{-1}$ by $vu$.

We will say that $S_\ell$ and $S_r$ are the \textit{connection sets} and that $S_m$ is the \textit{bi-connection} or \textit{di-connection} set in the bi-Cayley or the di-Cayley graph, respectively.
We write 
	$$ BX^*(G;S_\ell,S_r,S_m) \qquad \text{and} \qquad DX^*(G;S_\ell,S_r,S_m) $$ 
to simultaneously refer to $BX(G;S_\ell,S_r,S_m)$ or $BX^+(G;S_\ell,S_r,S_m)$ 
and, 
to $DX(G;S_\ell,S_r,S_m)$ or $DX^+(G;S_\ell,S_r,S_m)$, respectively. Sometimes it will be useful to simply write 
	$$X^*(G;S_\ell,S_r,S_m)$$ 
to refer to both $BX^*(G;S_\ell,S_r,S_m)$ or $DX^*(G;S_\ell,S_r,S_m)$.

If $S_\ell=S_r$ we say that $BX^*(G;S_\ell,S_r,S_m)$ is a \textit{mirror} bi-Cayley (sum) graph and $DX^*(G;S_\ell,S_r,S_m)$ is a \textit{mirror} di-Cayley (sum) graph.
When $S_m$ is a symmetric subset, the di-Cayley graph $DX^*(G;S_\ell,S_r,S_m)$ coincides with the bi-Cayley graph $BX^*(G;S_\ell,S_r,S_m)$, but in general they are different. Hence we will try to study both kind of graphs in general, focusing on the similarities and the differences. A particular family of mirror di-Cayley (sum) graphs was recently studied by the authors in \cite{ChP}.

\subsection*{Graph spectral definitions}
Let $\G=(V,E)$ be a graph with $n$ vertices, where $V$ denotes the vertex set and $E$ the edge set. The eigenvalues of $\Gamma$ are the eigenvalues $\{\lambda_i\}_{i=1}^{n}$ of its adjacency matrix $A$. The \textit{spectrum} of $\Gamma$, denoted
	$$ Spec(\Gamma)=\{[\lambda_{i_1}]^{m_{i_1}},\cdots,[\lambda_{i_s}]^{m_{i_s}}\} $$
is the (multi)set of all the different eigenvalues $\{\lambda_{i_j}\}$ of $\Gamma$, counted with their multiplicities $\{m_{i_j}\}$, that is $m(\lambda_{i_j})=m_{i_j}$. The spectrum of $\Gamma$ is called \textit{symmetric} if $Spec(\G)=-Spec(\G)$, that is the multiplicities of $\lambda$ and $-\lambda$ coincide for any $\lambda$, in symbols 
	$$m(\lambda) = m(-\lambda)$$ 
for every $\lambda\in Spec(\Gamma)$. The spectrum of $\G$ is said to be \textit{real} or \textit{integral} if $Spec(\Gamma)\subset\mathbb{R}$ or $Spec(\Gamma)\subset\mathbb{Z}$, respectively. 
Two graphs with the same number of vertices $\Gamma_1$ and $\Gamma_2$ are said to be \textit{isospectral} if 
$$ Spec(\Gamma_1) = Spec(\Gamma_2).$$ 

We recall that if $\Gamma$ is a $k$-regular graph, then $\lambda_1 = k$. Moreover, the multiplicity of $\lambda_1$ equals the number of connected components of $\G$ and hence $\Gamma$ is connected if and only if $m(\lambda_1) = 1$. Spectrally, a $k$-regular graph $G$ is bipartite if and only if the spectrum of $G$ is symmetric, which happens if and only if $-k$ is an eigenvalue of $\Gamma$.

\subsection*{Outline and results}
Briefly, in Section \ref{sec: diCayley graphs} we introduce di-Cayley graphs and study basic properties, in Section \ref{sec: adj mat} we study their adjacency matrices, in Section~\ref{sec: spec} we study the spectrum of Cayley (sum) graphs, in Section \ref{sec: spectrum of dicayleys} we study the spectrum of di-Cayley (sum) graphs and in Section \ref{sec: explicit computations} we do some explicit computations.

\sk 

More precisely, in Section \ref{sec: diCayley graphs} we introduce the family of di-Cayley (sum) graphs as a directed generalization of bi-Cayley graphs (see Definitions \ref{defi: biCayleys} and \ref{defi: biCayleys+}). Then we study basic structural properties of the graphs $\G^*=DX^*(G;S_\ell,S_r,S_m)$. In Proposition \ref{prop: DX undirected} we give necessary and sufficient conditions for $\G$ and $\G^+$ to be directed or undirected in terms of (anti)symmetry and (anti)normality of the (di)connection sets, respectively.
In Proposition \ref{prop: loops} we give necessary and sufficient conditions for the graphs $\G^*=DX^*(G;S_\ell,S_r,S_m)$ to have loops (globally or in a cover). Finally, in Proposition \ref{prop: biCay bireg} we show that these graphs are biregular (two different valencies). 

\sk 

The next section is devoted to the adjacency matrices of di-Cayley (sum) graphs.
In Proposition \ref{prop: AdjMat} we give the adjacency matrices $A_{BX^*}$ and $A_{DX^*}$ of $BX^*(G;S_\ell,S_r,S_m)$ and $DX^*(G;S_\ell,S_r,S_m)$, respectively, in terms of the adjacency matrices of the associated Cayley graphs $X(G;S_\ell)$, $X(G;S_r)$, $X(G;S_m)$ and to the Cayley sum graphs $X^+(G;S_\ell)$, $X^+(G;S_r)$, $X^+(G;S_m)$, respectively.
In Theorem~\ref{thm: ASAT=ATAS} we show that if $G$ is a finite group and $S,T$ are normal subgroups, the adjacency matrices of $X^*(G,S)$ and $X^*(G,T)$ commute with each other. This will be crucial to obtain the spectral results in Section \ref{sec: spectrum of dicayleys}.
The result for Cayley (sum) graphs generalizes to mirror bi/di-Cayley (sum) graphs having normal and symmetric connection sets (see Corollary \ref{coro: mirror Adj mats}).
Finally, in Corollary \ref{coro: commuting mirror adj mat} we show that if $\G=X^*(G;S,S,T)$ and $\G'=X^*(G;S',S',T')$ are two normal mirror bi/di-Cayley (sum) graphs with $T,T'$ symmetric, then their adjacency matrices $A_\G$ and $A_{\G'}$ commute, that is 
$$	A_\G A_{\G'} = A_{\G'} A_{\G}. $$

In Section \ref{sec: spec} we recall the spectrum of Cayley (sum) graphs $X^*(G,S)$ in terms of irreducible characters of the group $G$ (see Propositions \ref{prop: spec XGS G abelian}--\ref{prop: spec XGS+ G abelian} for $G$ abelian and Propositions \ref{prop: spec XGS}--\ref{prop: spec XGS+} for $G$ arbitrary). 
Then, using this, in Section \ref{sec: spectrum of dicayleys} we study the spectrum of di-Cayley (sum) graphs which is our main interest. 
In Theorems \ref{thm: Spec DX} and \ref{thm: Spec DX+}, which are the main results in the article, we give the eigenvalues of bi/di-Cayley (sum) graphs. More precisely, if $G$ is a finite group and $S_\ell$, $S_r$, $S_m$ are symmetric and closed under conjugation subsets of $G$, then the eigenvalues bi/di-Cayley (sum) graph $BX^+(G;S_\ell,S_r,S_m)$ and $DX^+(G;S_\ell,S_r,S_m)$ are 
respectively given by 
\begin{align*} 
	\lambda_{BX^*}^{\pm} (\chi) = \tfrac{ \lambda_{S_\ell,\chi}^* + \lambda_{S_r,\chi}^*}{2} \pm 
	\sqrt{ \Big(\tfrac{\lambda_{S_\ell,\chi}^* - \lambda_{S_r,\chi}^*}{2} \Big)^2 + |\lambda_{S_m,\chi}^*|^2 }, \\[2mm]
	\lambda_{DX^*}^{\pm} (\chi) = \tfrac{\lambda_{S_\ell,\chi}^* + \lambda_{S_r,\chi}^*}{2} \pm 
	\sqrt{ \Big( \tfrac{\lambda_{S_\ell, \chi}^* - \lambda_{S_r,\chi}^*}{2} \Big)^2 + (\lambda_{S_m,\chi}^*)^2} 
\end{align*}	
where $\lambda_{S_\ell,\chi}^*$, $\lambda_{S_r,\chi}^*$ and $\lambda_{S_m,\chi}^*$ are the eigenvalues of $X^*(G,S_\ell)$, $X^*(G,S_r)$ and $X^*(G,S_m)$, and $\chi \in \hat G$ runs over the set of all irreducible characters of $G$.
In Corollary \ref{coro: spec with chi(G)} we give these four spectra in terms of irreducible characters of $G$.

The above expressions for the eigenvalues simplify dramatically in the case of mirror di-Cayley graphs. Namely, in Corollary \ref{coro: spec mirror} we show that   
if $G$ is a finite group and $S, T \subset G$ are closed under conjugation, then the eigenvalues of 
the mirror bi/di-Cayley (sum) graphs 
$BX^*(G;S,S,T)$ and $DX^*(G;S,S,T)$ are given by
\begin{equation*} 
	\lambda_{BX^*}^{\pm} (\chi) = \lambda_{S,\chi}^* \pm  |\lambda_{T,\chi}^*|,
	\qquad \text{and} \qquad
	\lambda_{DX^*}^{\pm} (\chi) = \lambda_{S,\chi}^* \pm  \lambda_{T,\chi}^*,
\end{equation*}
respectively, where $\lambda_{S,\chi}^*$ and $\lambda_{T,\chi}^*$ are the eigenvalues of $X^*(G,S)$ and $X^*(G,T)$ corresponding to $\chi$, with $\chi \in \hat G$. Finally, in Proposition  \ref{prop: integral graph} we give conditions for integrality of bi/di-Cayley (sum) graphs $X^*(G;S,S,T)$ in terms of the integrality of the covers, that is of the Cayley (sum) graphs $X^*(G,S)$ and $X^*(G,T)$.  

Finally, in Section \ref{sec: explicit computations} we do some explicit computations. In Examples \ref{exam: 6.1} to \ref{exam: 6.6} we consider bi/di-Cayley (sum) graphs over small groups, we give their graphic representations and compute their spectra in detail using the results in the previous sections. We consider mirror and non-mirror bi/di-Cayley (sum) graphs, with abelian and non-abelian groups, and having (Gaussian) integral or non (Gaussian) integral spectrum.

\section{Di-Cayley (sum) graphs} \label{sec: diCayley graphs}
Here we introduce the family of di-Cayley (sum) graphs as a directed generalization of the bi-Cayley (sum) graphs. 
We give some basic structural properties for these graphs.

\subsection*{Basic definitions}
We now define di-Cayley (sum) graphs in a slightly more general way than the standard definition of bi-Cayley (sum) graphs, allowing directed edges not only between vertices of the form $(u,i)$ and $(v,i)$ with $i\in \{0,1\}$, but also between 
vertices of the form $(u,i)$ and $(v,j)$ with $i,j \in \{0,1\}$ and $i\ne j$.
\begin{defi} \label{defi: biCayleys}
Let $G$ be a group 
and let $S_\ell$, $S_r$, $S_m$ be subsets of $G$. 
The \textit{di-Cayley graph} over $G$, denoted 
	$$ DX(G; S_\ell,S_r,S_m),$$ 
is the graph having the vertex set $V=G\times\{0,1\}$ and where the vertices $(h,i)$ and $(g,j)$ form a directed edge if and only if one of the following four possibilities occurs: 
\goodbreak
\begin{enumerate}
	\item[($a$)] \: $i=j=0$ \quad and \quad $gh^{-1}\in S_\ell$; \sk  
	\item[($b$)] \: $i=j=1$ \quad and \quad $gh^{-1}\in S_r$; \sk 
	\item[($c$)] \: $i=0$, $j=1$ \quad and \quad $gh^{-1}\in  S_m$; \sk 
	\item[($d$)] \: $i=1$, $j=0$ \quad and \quad $gh^{-1}\in  S_m$. 
\end{enumerate}
\nopagebreak
In cases ($a$) and ($b$) the edge is directed from $(h,i)$ to $(g,i)$ with $i=0,1$, in case ($c$) from $(h,0)$ to $(g,1)$ and in case ($d$) from $(h,1)$ to $(g,0)$.
If $S_\ell=S_m=S$ we will say that $DX(G; S, S, S_m)$ is a \textit{mirror di-Cayley graph}. 
\end{defi}

Analogously, we can define di-Cayley sum graphs, just by changing $gh^{-1}$ by $gh$ in the above definition of di-Cayley graph. In the case that $G$ is abelian, $g-h$ changes by $g+h$, from where the name comes. More precisely, we have the following.

\begin{defi} \label{defi: biCayleys+}
Let $G$ be a group and let $S_\ell$, $S_r$, $S_m$ be subsets of $G$. 
The \textit{di-Cayley sum graph} over $G$, denoted 
	$$ DX^{+}(G;S_\ell,S_r,S_m)$$ 
has the vertex group $G\times\{0,1\}$ where $(h,i)$ and $(g,j)$ form a directed edge if and only if one of the following four possibilities occurs:
\begin{enumerate}[($a$)]
	\item \: $i=j=0$ and $gh\in S_\ell$; \sk 
	\item \: $i=j=1$ and $gh\in S_r$; \sk 
	\item \: $i=0$, $j=1$ and $gh\in  S_m$. \sk 
	\item \: $i=1$, $j=0$ and $gh\in  S_m$.
\end{enumerate}
In cases ($a$) and ($b$) the edge is directed from $(h,i)$ to $(g,i)$ with $i=0,1$, in case ($c$) from $(h,0)$ to $(g,1)$ and in case ($d$) from $(h,1)$ to $(g,0)$.
If $S_\ell=S_r=S$ we will say that $DX^+(G; S, S, S_m)$ is a \textit{mirror di-Cayley sum graph}. 
\end{defi}

By Definitions \ref{defi: biCayleys} and \ref{defi: biCayleys+},
we will write $DX^*(G; S_\ell, S_r, S_m)$ to refer to the graphs $DX(G; S_\ell, S_r, S_m)$ and $DX^+(G; S_\ell, S_r, S_m)$ simultaneously.
For mirror di-Cayley (sum) graphs we will write $S$ instead of $S_\ell=S_r$ and $T$ in place of $S_m$. Hence, in these cases we will use the notation $DX^*(G;S,S,T)$. 

To study bi and di-Cayley graphs in an unified manner, we will sometimes use the notation 
\begin{equation}  \label{eq: X* not}
	X^*(G;S_\ell,S_r,S_m)	
\end{equation}
to denote both the bi-Cayley (sum) graphs $BX^*(G;S_\ell,S_r,S_m)$ or the di-Cayley (sum) graphs $DX^*(G;S_\ell,S_r,S_m)$, and we refer to them as \textit{bi/di-Cayley (sum) graphs}. 
Hence, we will also write 
	$ X^*(G;S,S,T) $ 
when referring to the mirror bi/di-Cayley (sum) graphs together.

We now make some comments about the previous definitions.

\begin{rem} \label{rem: cosas1}
($i$) The original definition of bi-Cayley graphs consider items ($a$), $(b)$ and ($c$) as in Definition \ref{defi: biCayleys}, excluding ($d$), but the edges given by condition ($c$) are undirected by definition.
That is, it allows the graphs $X(G,S_\ell)$ and $X(G,S_r)$ to be directed or not, depending whether the connections sets $S_\ell$ and $S_r$ are symmetric or not, but the edges between these two graphs are undirected. We prefer to consider this more general situation, allowing directed edges between these Cayley graphs. In both cases, the more general bi/di-Cayley graph is of mixed type (there can be directed and undirected edges).

\noindent ($ii$) 
Similarly as in ($i$), if in Definition \ref{defi: biCayleys+} we consider items ($a$), $(b)$ and ($c$), excluding ($d$), and the edges given by condition ($c$) are undirected by definition, we get the \textit{bi-Cayley sum graphs}. 

\noindent ($iii$) It is clear by the definitions that, when $S_m$ is symmetric, the di-Cayley graph $DX(G; S_\ell, S_r, S_m)$ is a bi-Cayley graph and coincides with $BX(G; S_\ell, S_r, S_m)$. Similarly, when $S_m$ is normal, the di-Cayley sum graph $DX^+(G; S_\ell, S_r, S_m)$ coincides with $BX^+(G; S_\ell, S_r, S_m)$.

\noindent ($iv$) 
If $S_m$ consists only of the identity element $e$ of $G$, we will say that $DX(G; S_\ell, S_r, \{e\})$ is a \textit{one-matching di-Cayley graph} over $G$, 
following the definition in the bi-Cayley case.
\end{rem}

Notice that conditions ($a$) and ($b$) in Definitions \ref{defi: biCayleys} and \ref{defi: biCayleys+} give the Cayley (sum) graphs $X^*(G,S_\ell)$ and $X^*(G,S_r)$ respectively, and that conditions ($c$) and ($d$) glue these graphs together depending on the set $S_m$. Thus, we can think of the bi/di-Cayley graph 
$X^*(G; S_\ell,S_r,S_m)$ as the union of the two Cayley (sum) graphs $X^*(G, S_\ell)$ and $X^*(G, S_r)$ under the action of $S_m$. 
Hence, we will refer to the graphs $X^*(G, S_\ell)$ and $X^*(G, S_r)$ as the \textit{left and right covers} of $X^*=X^*(G;S_\ell,S_r,S_m)$, respectively. Also, we will say that $S_\ell$ and $S_r$ are the (left and right) connections sets and $S_m$ is the \textit{di-connection} set (or a \textit{bi-connection} set if $X^*$ is bi-Cayley).
Sometimes it will be useful to refer to the edges between different covers, that is those given by ($c$) and ($d$), as the \textit{di-edges} (or \textit{bi-edges} if $X^*$ is bi-Cayley).

\begin{rem}
We have defined the family of mirror di-Cayley (sum) graphs (MDCGs for short) as those di-Cayley (sum) graphs with $S_\ell=S_r$. This is a fancy family since it can be proved that mirror di-Cayley (sum) graphs are in fact Cayley graphs (this is not true in general). 
Indeed, in Proposition 2.3 of \cite{ChP} we have proved that 
	$$ DX^*(G;S,S,T) = X^*(G\times \Z_2, S\times \{0\} \cup T\times \{1\}). $$ 
This relation does not hold for bi-Cayley (sum) graphs in general, and it is one our main motivations for studying these graphs.
\end{rem}

\subsection*{Arcs, loops and regularity}
Here we study the most basic properties of a graph, such as whether if it is directed, undirected or mixed, whether if it has loops or not, and regularity.

\subsubsection*{Directedness}
We begin by making clear when the di-Cayley graphs, which are generically mixed, are directed or undirected.	

\begin{prop} \label{prop: DX undirected}
Consider the di-Cayley graph $\G = DX(G;S_\ell,S_r,S_m)$ and the di-Cayley sum graph 
$\G^+ = DX^+(G;S_\ell,S_r,S_m)$. 
\begin{enumerate}[$(a)$]
	\item $\G$ is undirected if and only if the (di)connection sets are symmetric. \sk 
	
	\item $\G^+$ is undirected if and only if the (di)connection sets are normal. \sk 
	
	\item $\G$ is directed if and only if the (di)connection sets are antisymmetric. \sk 

	\item $\G^+$ is undirected if and only if the (di)connection sets are antinormal. 
\end{enumerate}
\end{prop}

\begin{proof}
($a$) $\G$ is undirected if and only if $X(G,S_\ell)$ and $X(G,S_r)$ are undirected and the edges between the covers $X(G,S_\ell) \times \{0\}$ and 
$X(G,S_r) \times \{1\}$ are undirected, and this happens if and only if $S_\ell$, $S_r$ and $T$ are symmetric subsets by \eqref{eq: S=S^-1}.
	
\noindent 
($b$) $\G^+$ is undirected if and only if $X^+(G,S_\ell)$ and $X^+(G,S_r)$ are undirected and the edges between the covers $X^+(G,S_\ell) \times \{0\}$ and 
$X^+(G,S_r) \times \{1\}$ are undirected, and this happens if and only if $S_\ell$, $S_r$ and $T$ are normal subsets by \eqref{eq: normal}.
	
\noindent 
($c$) $\G$ is directed if and only if $X(G,S_\ell)$ and $X(G,S_r)$ are directed and the edges between the covers $X(G,S_ \ell) \times \{0\}$ and 
$X(G,S_r) \times \{1\}$ are directed, and this happens if and only if $S_\ell$, $S_r$ and $T$ are antisymmetric subsets by \eqref{eq: antisymmetric}.
	
\noindent 
($d$) $\G^+$ is directed if and only if $X^+(G,S_\ell)$ and $X^+(G,S_r)$ are directed and the edges between the covers $X^+(G,S_\ell) \times \{0\}$ and 
$X^+(G,S_r) \times \{1\}$ are directed, and this happens if and only if $S_\ell$, $S_r$ and $T$ are antinormal subsets by \eqref{eq: antinormal}.
\end{proof}

Notice that if $\G^* = BX^*(G;S_\ell,S_r,S_m)$ is bi-Cayley then $\G^*$ is never directed by definition and $\G$ (resp.\@ $\G^+$) is undirected if and only if $S_\ell$ and $S_r$ are symmetric (resp.\@ normal).

In the general case, we have the following cases of interest (among others) of mixed di-Cayley graphs:

\noindent		
($a$) $\G$ (resp.\@ $\G^+$) is mixed with undirected covers joined by arcs if and only if $S_\ell$, $S_r$ are symmetric (resp.\@ normal) and $S_m$ is non-symmetric (resp.\@ non-normal). 	

\noindent		
($b$) $\G$ (resp.\@ $\G^+$) is mixed with directed covers joined by undirected edges if and only if $S_\ell$, $S_r$ are non-symmetric (resp.\@ non-normal) and $S_m$ is symmetric (resp.\@ normal).

\subsubsection*{Loops}
We now study loops. 
\begin{itemize}
	\item The graph $X(G,S)$ is looped if $e \in S$ or loopless if $e \notin S$. \sk 
	
	\item The graph $X^+(G,S)$ has a loop at vertex $x$ if and only if $x^2 \in S$ ($2x \in S$ if $G$ is abelian). 
	If $S$ is subgroup, then $X^+(G,S)$ is looped at least at $S$. 
\end{itemize}

We now study loops in bi/di-Cayley (sum) graphs. 

\goodbreak 

\begin{prop} \label{prop: loops}
Let $G$ be a group and $S_\ell, S_r, S_m \subset G$. Let $\G = X(G;S_\ell, S_r, S_m)$, $\G^+ = X^+(G;S_\ell, S_r, S_m)$ and $S \in \{S_\ell, S_r\}$. We have the following:

\noindent $(a)$ 
$\G$ is looped (resp.\@ loopless) if and only if $e \in S_\ell \cap S_r$ (resp.\@ $e \notin S_\ell \cap S_r$). 
In particular, this happens if $S_\ell$ and $S_r$ are subgroups of $G$. 

\noindent $(b)$ 
$\G$ is looped only at the cover $X(G,S_\ell) \times \{0\}$ if and only if $e\in S_\ell$ and $e \notin S_r$ 
while $\G$ is looped only at the cover $X(G,S_r) \times \{1\}$ if and only if $e\in S_r$ and $e \notin S_\ell$.

\noindent $(c)$ 
$\G^+$ has loops at $(x,0)$ if and only if $x^2\in S_\ell$ and has loops at $(x,1)$ if and only if $x^2\in S_r$. This happens for instance for the elements $x$ of order 2 in $S$ if $e\in S$ and it is automatic for any element of $S$ if it is a subgroup. 
\end{prop}

\begin{proof}
This is clear from the definitions and the previous comments. 
\end{proof}

As a consequence, if $\G^*=X^*(G;S_\ell,S_r,S_m)$ and $S_\ell, S_r$ are subgroups of $G$ then $\G$ has loops at every vertex and $\G^+$ has loops at least at every vertex of the form $(s,0)$, $(t,1)$, with $s\in S_\ell$ and $t\in S_r$.

\begin{exam}
Let $\Gamma^+ = X^+(\Z_n;2\Z_n,2\Z_n, T)$ for any $n \in \N$ with $T\subset \Z_n$. 
Since $2\Z_n$ is a subgroup of $\Z_n$, the graph $X^+(\Z_n,2\Z_n)$ is looped. If $n$ is odd we have that $X^+(\Z_n,2\Z_n)=\mathring{K_n}$ (the complete graph of order $n$ with a loop at every vertex), while if $n=2m$ then 
$X^+(\Z_n,2\Z_n)$ is the union of two copies of $\mathring{C_m}$ (the $m$-cycle graph with a loop at every vertex). 
Hence, the graph $\Gamma^+ = X^+(\Z_n;2\Z_n,2\Z_n, T)$ is looped for any $n$ and $T$.
\hfill $\diamond$
\end{exam}

\subsubsection*{Bi-regularity}
If $G$ is finite of order $n$, then $X^*(G, S_\ell,S_r,S_m)$ has order $2n$.
We will use the notations 
	$$|S_\ell| = k_\ell, \qquad |S_r| = k_r \qquad \text{and} \qquad |S_m| = k_m. $$
Note that the graph $X^*(G, S_\ell)$ is $k_\ell$-regular and the graph $X^*(G, S_r)$ is $k_r$-regular (whether they are directed or not). We have that if $k_\ell = k_r$, then $X^*(G; S_\ell, S_r, S_m)$
is $(k_\ell+k_m)$-regular. 

Now, we introduce the concept of biregularity of graphs. 
A graph is biregular if it has vertices of two different degrees. More precisely, we have the following.

\begin{defi} \label{def: biregular}
A graph $G = (V,E)$ is said to be \textit{biregular} with regularity degrees $n$ and $m$, or simply \textit{$(n,m)$-regular}, if there is a bipartition of the set of vertices, i.e.\@ 
	$$V = V_1\cup V_2 \qquad \text{with} \qquad V_1 \cap V_2 = \varnothing,$$ 
such that 
	$$ \delta(v) = n \quad \forall \, v\in V_1 \qquad \text{and} \qquad \delta(w) = m  \quad \forall w\in V_2.$$ 
If $n = m$, then the graph is $n$-regular.
\end{defi} 
A trivial example of biregular graphs is given by the complete bipartite graphs $K_{m,n}$ which are $(n,m)$-regular.

\begin{prop} \label{prop: biCay bireg}
	The graphs $X^*(G; S_\ell, S_r, S_m)$ are $(k_\ell + k_m, k_r + k_m)$-regular.
\end{prop}

\begin{proof} 
Each vertex of the form $(g,0)$ has degree $k_\ell + k_m$, the $\ell$ coming from edges on the cover $X^*(G,S_\ell)$ and the $m$ coming from edges between the cover $X^*(G,S_r)$. 
Similarly, vertices of the form $(g,1)$ have degree $k_r + k_m$, taking into account the neighbors on $X^*(G,S_r)$ and on $X^*(G,S_\ell)$, respectively. 
Therefore, the graph $X^*(G; S_\ell, S_r, S_m)$ is biregular with degrees $(k_\ell + k_m, k_r + k_m)$.
\end{proof}

\section{Adjacency matrices} \label{sec: adj mat}
In this section, we give the adjacency matrices of bi/di-Cayley (sum) graphs and study their properties. 
We will need them to obtain our main results in the next section; that is, for the computation of the spectrum of the graphs.

The adjacency matrix of a graph $G$ with $n$ vertices is an $n\times n$ matrix $A=A_G$ where each entry $a_{ij}$
represents the presence or absence of an edge from $i$ to $j$. Specifically, 
is $G$ is undirected, $a_{ij}=1$ if there is an edge from $i$ to $j$, and $0$ otherwise (hence $A$ is symmetric and the eigenvalues are real); while if $G$ is directed, $a_{ij}=1$ if there is a directed edge from $i$ to $j$ and $0$ otherwise.
In other words, \textit{adjacency matrix} $A$ of $G$ is defined by 
\begin{equation} \label{eq: adj matrix}
	A_{ij} = \begin{cases}
		1 & \quad \text{if $ij \in E$}, \\[1mm]
		0 & \quad \text{if $ij \not\in E$}, \end{cases} 
	\qquad \text{or} \qquad 
	A_{ij} = \begin{cases}
		1 & \quad \text{if $\vec{ij} \in E$}, \\[1mm]
		0 & \quad \text{if $\vec{ij} \not\in E$}, \end{cases}	
\end{equation}
depending whether $G$ is undirected or not.
We have that $G$ is undirected if and only if $A$ is symmetric ($A=A^t$) and thus with real spectrum.

A Cayley graph $X(G,S)$ is undirected if and only if $S$ is symmetric while a Cayley sum graph $X^+(G,S)$ is undirected if and only if $S$ is normal. 

The following basic result, relating the transpose of the adjacency matrix of $X(G,S)$ with the adjacency matrix of $X(G,S^{-1})$, is automatic from the definitions.

\begin{lem} \label{lem: A y At}
If $A_S$ is the adjacency matrix of $X(G,S)$, then $A_S^t$ is the adjacency matrix of $X(G,S^{-1})$, that is $A_S^t=A_{S^{-1}}$. 
\end{lem} 

\begin{proof}
Let $A_S=(a_{ij})$ and $A_{S^{-1}}=(a'_{ij})$ be the adjacency matrices of $X(G,S)$ and $X(G,S^{-1})$, respectively. Then, $a_{ij}=1$ if and only if there is a directed edge from $x_i$ to $x_j$, i.e.\@ $x_jx_i^{-1} \in S$, which happens if and only if 
	$$ x_ix_j^{-1} = (x_jx_i^{-1})^{-1} \in S^{-1}. $$ 
Thus, $a'_{ij}=1$ if and only if $a_{ji}=1$, from which the result follows.
\end{proof}

Therefore, $S$ is symmetric if and only if $A_S$ is symmetric.
There is no such result as Lemma~\ref{lem: A y At} for the Cayley sum graph $X^+(G,S)$. However, if $S$ is normal, the adjacency matrix $A_S^+$ of $X^+(G,S)$ is symmetric.

It is easy to express the adjacency matrices of the bi-Cayley (sum) graph $BX^*(G;S_\ell,S_r,S_m)$ and the di-Cayley (sum) graph $DX^*(G;S_\ell,S_r,S_m)$ in terms of the adjacency matrices of the Cayley (sum) graphs $X^*(G,S_\ell)$, $X^*(G,S_r)$ and $X^*(G,S_m)$ as follows.

\begin{prop} \label{prop: AdjMat} 
The adjacency matrices of the bi and di-Cayley (sum) graphs $BX^*(G;S_\ell,S_r,S_m)$ and 
$DX^*(G;S_\ell,S_r,S_m)$ are given by 
\begin{equation} \label{eq: Adj mat}
	A_{BX^*} = \begin{pmatrix}	A_{X_\ell^*} & A_{X_m^*} \\ A_{X_m^*}^t & A_{X_r^*} \end{pmatrix}
	 \qquad \text{and} \qquad 
 	A_{DX^*} = \begin{pmatrix}	A_{X_\ell^*} & A_{X_m^*} \\ A_{X_m^*} & A_{X_r^*}  \end{pmatrix},
\end{equation}
where $A_{X_\ell^*}$, $A_{X_r^*}$ and $A_{X_m^*}$ are the adjacency matrices of the Cayley (sum) graphs $X^*(G,S_\ell)$, $X^*(G,S_r)$ and $X^*(G,S_m)$, respectively.  
In particular, we have that: 
\begin{enumerate}[$(a)$]
	\item If $A_{X_m^*}$ is symmetric then $A_{BX^*} = A_{DX^*}$. \msk 
	
	\item $A_{BX^*}$ and $A_{DX^*}$ are symmetric if and only if $A_{X_\ell^*}, A_{X_r^*}$ and $A_{X_m^*}$ are symmetric.
\end{enumerate}
\end{prop}

\begin{proof} 
The expressions in \eqref{eq: Adj mat} follows directly from the definition of the adjacency matrix and the structure of the bi/di-Cayley graphs.
The remaining assertions are clear. 
\end{proof}

In the case of the bi-Cayley graph $BX(G;S_\ell,S_r,S_m)$ with $G$ abelian the matrix $A_{BX}$ in \eqref{eq: Adj mat} was previously obtained in Lemma 3.1 in \cite{GL}.

For mirror bi/di-Cayley, the following is automatic from Proposition \ref{prop: AdjMat}. 
\begin{coro} \label{coro: mirror Adj mats}
The adjacency matrices of the mirror bi/di-Cayley (sum) graphs $BX^*(G;S,S,T)$ and $DX^*(G;S,S,T)$ are respectively given by
\begin{equation} \label{eq: mirror Adj mats}
	A_{BX^*} = \begin{pmatrix} A_{X_S^*} & A_{X_T^*} \\ A_{X_T^*}^t & A_{X_S^*} \end{pmatrix}  
		\qquad \text{and} \qquad 
	A_{DX^*} = \begin{pmatrix} A_{X_S^*} & A_{X_T^*} \\ A_{X_T^*} & A_{X_S^*} \end{pmatrix}
\end{equation}
where $A_{X_S^*}$ and $A_{X_T^*}$ are the adjacency matrices of $X^*(G,S)$ and $X^*(G,T)$ respectively.
\end{coro}

Next we show that the adjacency matrices of Cayley (sum) graphs defined by normal subsets commute with each other. 
We will need the following basic result.	

\begin{lem} \label{lem: classfunction}
Let $S$ be a normal subset of a group $G$. Then, the characteristic function $\chi_S$ of $S$ is a class function of $G$. 
\end{lem}	

\begin{proof}
By definition, $\chi_S(x)=1$ if $x\in S$ and $\chi_S(x)=0$ if $x\notin S$.
To see that $\chi_S$ is a class function we must see that given $x\in G$ then 
	$$ \chi_S(gxg^{-1}) = \chi(x) \qquad \forall \,g\in G. $$ 
But $\chi(gxg^{-1})=1$ if and only if $gxg^{-1}=s \in S$, that is if $x=g^{-1}sg \in S$. Thus, $x \in S$, since $S$ is normal, and hence $\chi(x)=1$.   
\end{proof}

The next result will be crucial for the computation of the spectrum of bi/di-Cayley (sum) graphs, but it is interesting on its own.
It says that the adjacency matrices of two normal Cayley (sum) graphs always commute with each other. 
\begin{thm} \label{thm: ASAT=ATAS}
Let $G$ be a finite group and $S,T \subset G$. If $S,T$ are normal, then 
	$$A_{S^*} A_{T^*} = A_{T^*} A_{S^*},$$ 
where $A_{S^*}$ and $A_{T^*}$ are the adjacency matrices of $X^*(G,S)$ and $X^*(G,T)$. 
\end{thm}

\begin{proof}
We will prove the commutativity of the adjacency matrices of the Cayley graphs $X(G,S)$ and $X(G,T)$, 
that is, we will show that for any $i,j$ it holds
	$$ (A_S A_T)_{ij} = (A_T A_S)_{ij}. $$ 
The corresponding proof for the sum graphs, i.e.\@ that $(A_S^+ A_T^+)_{ij} = (A_T^+ A_S^+)_{ij}$ for any $i,j$ 
is analogous.

Assume that $G=\{g_1,\ldots,g_n\}$. Hence $G$ is the vertex set of the Cayley (sum) graphs $\G_S^*=X^*(G,S)$ and $\G_T^*=X^*(G,T)$, that is $V_{\G_S^*}=V_{\G_T^*}=G$.

On the one hand, using \eqref{eq: adj matrix} we have that 
\begin{equation} \label{eq: ASATij}
	(A_S A_T)_{ij}  = \sum_{k=1}^n (A_S)_{ik} (A_T)_{kj} = \# \{ 1\le k \le n : (A_S)_{ik} = (A_T)_{kj} = 1 \}.	
\end{equation} 
Now, recalling that in a Cayley graph $X(G,U)$ we have that $uv\in E$ if and only if $vu^{-1}  \in U$ (see \eqref{eq: edges}), we have that
\begin{equation} \label{eq: lados en X}
	(A_S)_{ik}=1 \qquad \Leftrightarrow \qquad g_i g_k \in E_{\G_S} \qquad \Leftrightarrow \qquad g_kg_i^{-1} \in S;	
\end{equation}
and hence we have obtained that 
\begin{align*}
	(A_S A_T)_{ij} = \# \{ 1\le k \le n \: : \: g_kg_i^{-1} \in S \: \text{ and } \: g_jg_k^{-1} \in T \} . 
	\end{align*}
That is, $g_k= sg_i = t^{-1}g_j$ for some $s \in S$ and some $t \in  T$. In this way, we arrive at 
$(A_S A_T)_{ij} = \# \{ (s,t) \in S\times T : s g_i = t^{-1}g_j  \}$, from which we get 
	$$ (A_S A_T)_{ij}  = \# \{ (s,t) \in S\times T :  g_j g_i^{-1} = ts \}. $$ 
Similarly, we obtain that 
	$$ (A_T A_S)_{ij}  = \# \{ (s,t) \in S\times T : g_j g_i^{-1} = st \}. $$

On the other hand, recall that the group algebra $\C[G]$ of $G$ consists of the elements of the form 
	$$ \alpha = \sum_{g \in G} a_g g $$ 
with $a_g \in \C$. One can think of the coefficients as functions $a: G \rightarrow \C$ where $a(g)=a_g$. Clearly, $\alpha$ is in the center $\mathcal{Z}(\C[G])$ of $\C[G]$ if and only if 
	$$ a(g)=a(xgx^{-1}) $$ 
for every $g,x \in G$, that is if $a$ is a class function of $G$.

Now, for any $R \subset G$ we define the element 
	$$ \alpha_R = \sum_{r \in R} r \in \C[G]. $$
Then, we have 
\begin{align*}
	\alpha_S \alpha_T = \big( \sum_{s \in S} s \big) \big( \sum_{t \in T} t \big) = \sum_{s\in S} \sum_{t \in T} st = \sum_{g \in G} a_g g, \\
	\alpha_T \alpha_S = \big( \sum_{t \in T} t \big) \big( \sum_{s \in S} s \big) = \sum_{t\in T} \sum_{s \in S} ts = \sum_{g \in G} b_g g,
\end{align*}
where $a_g = \# \{ (s,t) \in S \times T : st=g \}$ and $b_g = \# \{ (s,t) \in S \times T : ts=g \}$.

Notice that for any normal subset $S \subset G$, the element $\alpha_S \in \C[G]$ defined by 
$$ \alpha_S = \sum_{s\in S} s = \sum_{g \in G} \chi_S(g) g$$ 
is in fact in the center of $\C[G]$, by Lemma \ref{lem: classfunction}. 
Thus $\alpha_S$ and $\alpha_T$ commute with each other, i.e.
	$$\alpha_S\alpha_T = \alpha_T \alpha_S,$$ 
and thus $a_g=b_g$ for every $g \in G$, which implies that $(A_S A_T)_{ij} = (A_T A_S)_{ij}$ for any $i,j$, that is $A_SA_T=A_TA_S$ as we wanted to show.  

Finally, to prove that $A_S^+ A_T^+ = A_T^+ A_S^+$, we define the elements 
	$$ \alpha_S^+ = \sum_{s\in S} s^{-1} \qquad \text{and} \qquad \alpha_T^+ = \sum_{t\in T} t^{-1} $$
which are in the center of $ \C[G]$ and proceed similarly as before, using the definition of Cayley sum graphs 
in \eqref{eq: lados en X}.
\end{proof}

Notice that the same proof is valid for $G=\{g_k\}_{k\in \N}$ denumerable and $S,T$ finite subsets, since \eqref{eq: ASATij} still makes sense in this case summing over $k\in \N$.
However, we do not need this in this work.

\begin{rem} \label{rem: BBt=BtB}
In general $BB^t \ne B^tB$. However, if $B$ is the adjacency matrix of a normal Cayley graph $X(G,S)$ then 
	$$ BB^t = B^tB, $$
that is $B$ is normal.
This follows from the fact that $S$ is normal if and only if $S^{-1}$ is normal, by Lemma~\ref{lem: A y At} and Theorem  \ref{thm: ASAT=ATAS}.
\end{rem}

In analogy for the Cayley graphs, we say that a bi/di-Cayley (sum) graph $X^*(G;S,R,T)$ is \textit{normal} (resp.\@ \textit{symmetric}) if $R,S,T$ are normal (resp.\@ symmetric) subsets of $G$.

Putting together the results in the section we obtain that the adjacency matrices of two normal mirror bi/di-Cayley (sum) graphs with symmetric di-connection sets commute, thus generalizing Theorem \ref{thm: ASAT=ATAS} for Cayley (sum) graphs.

\begin{coro} \label{coro: commuting mirror adj mat}
Let $G$ be a finite group and $\G=X^*(G;S,S,T)$ and $\G'=X^*(G;S',S',T')$ are two normal mirror bi/di-Cayley (sum) graphs with adjacency matrices $A_\G$ and $A_{\G'}$ respectively. If $T,T'$ are symmetric, then 
	\begin{equation} \label{eq: commuting matrices}
		A_\G A_{\G'} = A_{\G'} A_{\G}.
	\end{equation}	
\end{coro}

\begin{proof}
By Corollary \ref{coro: mirror Adj mats}, the adjacency matrices of the mirror bi/di-Cayley (sum) graphs have a $2 \times 2$ block structure. 
In the bi-Cayley case, since $T$ and $T'$ are symmetric subsets, their adjacency matrices are symmetric, i.e., $A_{X_T^*}^t = A_{X_T^*}$ and $A_{X_{T'}^*}^t = A_{X_{T'}^*}$. Thus, for both bi-Cayley and di-Cayley graphs, the adjacency matrices of $\G$ and $\G'$ share the same block-symmetric structure:
	$$	A_\G = \begin{pmatrix} A_{X_S^*} & A_{X_T^*} \\ A_{X_T^*} & A_{X_S^*} \end{pmatrix} 
	\qquad \text{and} \qquad 
	A_{\G'} = \begin{pmatrix} A_{X_{S'}^*} & A_{X_{T'}^*} \\ A_{X_{T'}^*} & A_{X_{S'}^*} \end{pmatrix}. $$
Computing the product $A_\G A_{\G'}$ yields
	$$ A_\G A_{\G'} = \begin{pmatrix} 
		A_{X_S^*} A_{X_{S'}^*} + A_{X_T^*} A_{X_{T'}^*} & A_{X_S^*} A_{X_{T'}^*} + A_{X_T^*} A_{X_{S'}^*} \\[1mm] 
		A_{X_T^*} A_{X_{S'}^*} + A_{X_S^*} A_{X_{T'}^*} &  A_{X_T^*} A_{X_{T'}^*} + A_{X_S^*} A_{X_{S'}^*} \end{pmatrix}. $$
Since $S,S',T,T'$ are normal subsets, Theorem \ref{thm: ASAT=ATAS} guarantees that any pair of these individual blocks commutes (i.e., $A_{X_U^*} A_{X_V^*} = A_{X_V^*} A_{X_U^*}$ for all $U,V \in \{S,T,S',T'\}$). Applying this commutativity to every term in the block matrix directly yields $A_{\G'} A_{\G}$, completing the proof. 
\end{proof}

\begin{rem}
The preceding commutativity result does not hold in general for non-mirror bi/di-Cayley (sum) graphs. For arbitrary normal bi/di-Cayley (sum) graphs $\G=X^*(G; S_\ell, S_r, S_m)$ and $\G'=X^*(G; S_\ell', S_r', S_m')$, we can only guarantee that the main diagonal blocks of their adjacency matrix products coincide; that is, the block diagonals of $A_\G A_{\G'}$ and $A_{\G'} A_\G$ are equal as $2 \times 2$ block matrices, i.e.\@ $diag(A_\G A_{\G'}) = diag(A_{\G'} A_\G)$.
\end{rem}

\section{The spectrum of Cayley (sum) graphs} \label{sec: spec}
Here we focus on the spectrum of Cayley (sum) graphs $X^*(G,S)$ for $G$ finite. 
We will give the eigenvalues of these graphs in terms of characters of $G$. 
We recall the known results, because we will use them in the next section to give the eigenvalues of the bi/di-Cayley (sum) graphs $X^*(G;S_\ell,S_r,S_m)$.

It is well-known that the spectra of a Cayley (sum) graph $X^*(G,S)$ can be computed by using the irreducible characters $\hat{G}$ of $G$. We now recall these results in four propositions distinguishing, for ease, the case $G$ abelian or not, for the Cayley graphs $X(G,S)$ and Cayley sum graphs $X^+(G,S)$. 

Given an irreducible character $\chi\in\hat{G}$, we will need the notations
\begin{equation} \label{eq: eigenvalue and eigenvector}
	\lambda_{\chi}=\chi(S)=\sum_{s\in S}\chi(s) \qquad\text{and}\qquad v_{\chi}=(\chi(g))_{g\in G}.
\end{equation}

\subsection*{The case $G$ abelian} 
In this case, the following result on the eigenvalues of Cayley graphs is classic. 

\begin{prop} \label{prop: spec XGS G abelian}
Let $G$ be a finite 
abelian group and $S \subset G$. 
Then, the set of eigenvalues of $\G=X(G,S)$ is 
	$$ \mathrm{Eig}(\G)=\{ \lambda_{\chi}=\chi(S)\}_{\chi \in \hat G}. $$
The eigenvalue $\lambda_{\chi}$ has associated eigenvector $v_{\chi}$. 
\end{prop}

\begin{proof}
See Corollary 3.2 in \cite{Babai}.
\end{proof}

So, the eigenvalues depend on $\chi$ and $S$, but the eigenvectors depend only on $\chi$.

For Cayley sum graphs $X^+(G,S)$ with $G$ abelian we have the following analogous result. 
We need to recall that given a character $\chi:G \rightarrow \C$ of $G$, $\chi$ is said to be a real character if it takes real values, that is $\chi(g) \in \R$ for every $g\in G$. Equivalently, $\chi$ is a real character if and only if $\chi=\chi^{-1}$ in the group $\hat G$. We denote the subgroup of real characters of $G$ by $\hat{G}_\R$.

\begin{prop} \label{prop: spec XGS+ G abelian}
Let $G$ be a finite abelian group and $S \subset G$. 
Then, the set of eigenvalues of $\G^+=X^+(G,S)$ is 
\begin{equation} \label{eq: eig sum}
	{\rm Eig}(\G^+) = \{ \lambda_\chi = \chi(S) \}_{\chi \in \hat{G}_\R} \cup \{ \lambda^\pm_\chi = \pm |\chi(S)| \}_{\chi \notin \hat{G}_\R} \subset \R.
\end{equation}	
\end{prop}

\begin{proof}
	See Theorem 2.1 in \cite{DVGM}. 
\end{proof}

\subsection*{The case of arbitrary $G$}
In the general case, when $G$ is any finite group, similar results as the previous ones hold for $X(G,S)$ and $X^+(G,S)$. 
In this case we need that $S$ is a normal subset of $G$ so that one can use the character theory of $G$. 

For Cayley graphs we have the following. 

\begin{prop} \label{prop: spec XGS} 
Let $G$ be a finite group and $S$ a normal subset of $G$. 
Then, the set of eigenvalues of the Cayley graph $\G=X(G,S)$ is
	$$ {\rm Eig}(\G) = \{ \lambda_{\chi} = \chi(1)^{-1} \cdot \chi(S) \}_{\chi \in \hat{G}}, $$ 
where the multiplicity of $\lambda_\chi$ is given by 
	$$ m(\lambda_\chi) = \sum_{\eta \in \hat G, \, \lambda_\eta = \lambda_\chi} \eta(1)^2.$$
\end{prop}

\begin{proof}
See for instance Theorem 1 in \cite{Z}.
\end{proof}

For the sum graph case, we have obtained the following similar result, thus completing the picture.

\begin{prop} \label{prop: spec XGS+} 
Let $G$ be a finite group with $S$ a normal subset of $G$.
The eigenvalues of the Cayley sum graph $\G^+=X^+(G,S)$ are given by
	$$ {\rm Eig}(\G^+) = \{ \lambda_\chi = \tfrac{1}{\chi(1)}\chi(S)\}_{\chi \in \hat{G}_\R} \cup \{ \lambda_\chi^\pm = \pm \tfrac{1}{\chi(1)}\chi(S)\}_{\chi \notin \hat{G}_\R} \subset \R, $$ 	
where $\hat G$ is the set of all irreducible characters of $G$.
\end{prop}

\begin{proof}
For clarity, we divide the proof into four parts. 	

$\bullet$ \textit{The element $z$ and its spectral decomposition}:
Let $\C[G]$ be the complex group algebra of $G$ and consider the element
	$$ z=\sum_{h \in S} h \in \C[G].$$
Since $S$ is normal, we have that
	$$ gzg^{-1} = g \Big( \sum_{h\in S} h \Big) g^{-1} = \sum_{h\in S} ghg^{-1} = \sum_{s\in S} s  = z $$ 
for all $g\in G$, 
which implies that $gz=zg$ for all $g\in G$. Since $G$ form a basis of $\C[G]$, then $z$ is in the center of the group algebra, i.e.\@ $z\in \mathcal{Z}(\C[G])$.

Now, consider the decomposition of $\C[G]$ into its minimal ideals, say 
\begin{equation} \label{eq: C[G] decomp}
	\C[G] = I_1 \oplus \cdots \oplus I_t,
\end{equation}
and let $e_i$ be the central idempotent of $I_i$, for all 
$i \in \{1,\dots,t\}$. 
Since $z \in \mathcal{Z}(\C[G])$, there exist $\lambda_1,\dots,\lambda_t \in \C$ such that
\begin{equation}\label{eq: z linear comb}
	z = \sum_{i=1}^t \lambda_i e_i.
\end{equation}

$\bullet$ \textit{Three endomorphisms of $\C[G]$}:
Define the endomorphism $\mu_{z} : \C[G] \rightarrow \C[G]$ on $G$ by 	
	$$ \mu_{z}(g) = gz $$
and extend it by linearity to $\C[G]$.	 
The restriction of $\mu_{z}$ to each ideal $I_j$ is $\lambda_j {\rm Id}_{\nu_j}$ where $\nu_j=\dim I_j$. 
In particular, we have that $\{\lambda_1,\dots,\lambda_t\}$ is the set of all different eigenvalues of $\mu_{z}$, that is 
	$$Spec(\mu_z)=\{ [\lambda_i]^{\nu_i}\}_{i=1}^t.$$

Let $\iota :\C[G]\rightarrow \C[G]$ be the linear extension of the inversion map of $G$, that is
	$$ \iota \big(\sum_{g\in G} a_g g\big) = \sum_{g\in G} a_g g^{-1}. $$  
Note that $\iota$ is a $\C$-algebra automorphism. 
Hence, it permutes the simple ideals in the decomposition \eqref{eq: C[G] decomp} of $\C[G]$. 
In particular, for each $j \in \{1,\dots,t\}$ there exists some $k \in \{1,\dots,t\}$ such that 
	$$ \iota(I_j) = I_k. $$ 
It may happen that $k=j$ (so that $I_j$ is fixed), or else that $\iota$ maps $I_j$ onto another block $I_k$.

Now, for each $a \in \C[G]$ we consider the vector space endomorphism $\rho_a : \C[G] \rightarrow \C[G]$ 
defined on $G$ by					
	$$\rho_a(g) := g^{-1}a, $$ 
and extended by linearity to $\C[G]$. 
Notice that 
$$ \rho_{z} = \mu_{z} \circ \iota.$$
Since $z$ acts as the scalar $\lambda_j$ on each simple block $I_j$, 
and the map $\iota$ corresponds to matrix transposition, we have that
	$$ \rho_{z} (E_{k\ell}) = \mu_z(\iota(E_{k \ell})) = \mu_z(E_{\ell k}) = \lambda_j \, E_{\ell k}, $$
where $\{E_{k \ell}\}$ are the elemental matrices with a $1$ in positions $k \ell$ and $\ell k$ and $0$'s in the remaining positions.
Hence, for $k \ne \ell$, the operator $\rho_z$ is represented 
in the basis $\{E_{k \ell}, E_{\ell k}\}$ by the matrix 
	$$ 
	\begin{pmatrix} 0 & \lambda_j \\ \lambda_j & 0 \end{pmatrix}, $$
with eigenvalues $\pm \lambda_j$, while for $k=\ell$ it acts as multiplication by $\lambda_j$. 
Therefore, reordering if necessary, the spectrum of $\rho_z$ is given by 
\begin{equation} \label{eq: spec}
	\big\{ [\lambda_i]^{\nu_i} \big\}_{i=1}^r \cup \big\{ [\lambda_j]^{\frac{\nu_j}2}, [-\lambda_j]^{\frac{\nu_j}2} \big\}_{j=r+1}^t,	
\end{equation}
where $1\le r \le t$.

$\bullet$ \textit{The adjacency operator}:
Now, let $\alpha$ be adjacency endomorphism $\alpha : \C[G] \rightarrow \C[G]$
defined on $G$ by
$$ \alpha(g) = \sum_{hg \in S} h$$
and extended to $\C[G]$ by linearity. 
Notice that the matrix $[\alpha]$ of $\alpha$ is exactly the adjacency matrix $A$ of the graph $X^+(G,S)$. In fact, if $G=\{g_1,\ldots,g_n\}$, we have that
	$$ \alpha(g_i) = \sum_{g_j g_i \in S} g_j$$ 
for $i=1,\ldots,n$ and thus, 
$$ [\alpha]_{ij} = A_{ij} = \begin{cases}
	1 & \quad \text{if } g_j g_i \in S, \\[1mm]
	0 & \quad \text{if } g_j g_i \notin S,
\end{cases}$$
for any $1\le i,j \le n$.

On the other hand, since $\alpha(1)=z$, for each $g\in G$ we have that 
	$$  \alpha(g) = \sum_{hg \in S} h = \sum_{h \in S} hg^{-1} = \sum_{h \in S} g^{-1}h = g^{-1} \, \alpha(1) = g^{-1} {z} = \rho_{z}(g). $$ 
Thus, we obtain that $\alpha=\rho_z$ and hence $Spec(\alpha)=Spec(\rho_z)$, i.e. as in \eqref{eq: spec}.

$\bullet$ \textit{Expression for the eigenvalues}:
Now we show that the eigenvalues $\lambda_1, \ldots, \lambda_t$ has the expression as in the statement. For each $i\in\{1,\dots,t\}$, let 
$$ \phi_i = {\rho_z}_{|I_i}, $$ 
which is an irreducible representation of $G$ (by the minimality of $I_i$) and let $\chi_i$ be the character of $\phi_i$, that is 
	$$\chi_i(g)=\Tr(\phi_i(g)).$$
For each fixed $j \in \{ 1,\dots,t \}$, using the definition of $z$, we have
$$  \sum_{s\in S} \phi_j(s) = \phi_j \Big( \sum_{s\in S} s \Big) = \phi_j(z)$$
and, by \eqref{eq: z linear comb}, we have
$$ \phi_j(z) = \phi_j \Big( \sum_{i=1}^t \lambda_i e_i \Big) = \sum_{i=1}^t \lambda_i\phi_j(e_i).$$	
Putting together the above expressions and taking traces we obtain that
$$\sum_{s\in S}\chi_j(s) = \lambda_j \chi_j(1),$$
from where we finally get 
$\lambda_j = \frac{1}{\chi_j(1)} \chi_j(S)$.

Finally, recall that ${\rho_z}_{|I_i} = \lambda_i Id_{\nu_i}$ for $i=1,\ldots,r$ and 
	$$ {\rho_z}_{|I_j} = diag \left( \Big( \begin{smallmatrix} 0 & \lambda_j \\ \lambda_j & 0 \end{smallmatrix} \Big), \ldots, \Big( \begin{smallmatrix} 0 & \lambda_j \\ \lambda_j & 0 \end{smallmatrix} \Big) \right), \qquad \tfrac{\nu_j}2\text{-times}.$$
Therefore, $\chi_1, \ldots, \chi_r$ are the real characters and $\chi_{r+1}, \ldots, \chi_t$ are the non-real characters of $G$. This finish the proof. 	
\end{proof}

\begin{rem}
The previous result generalizes Theorem 3.1 in \cite{Bis} (which assumes $S$ symmetric), to the case when $S$ is not necessarily symmetric. 
The given proof is different than the one in \cite{Bis}, and uses the same ideas as in the proof of Theorem \ref{thm: ASAT=ATAS}.
\end{rem}

\begin{rem}
When $G$ is abelian and $S$ is symmetric, the expressions in Propositions~\ref{prop: spec XGS G abelian} and \ref{prop: spec XGS} coincide, as well as Propositions \ref{prop: spec XGS+ G abelian} and \ref{prop: spec XGS+}, as it should be. In fact, if $G$ is abelian $\chi(1)=1$ and if $S$ is symmetric $\chi(S) \in \R$.
\end{rem}

\section{The spectrum of di-Cayley (sum) graphs} \label{sec: spectrum of dicayleys} 
We now extend the approach used in the previous section to compute the eigenvalues of bi/di-Cayley (sum) graphs, using irreducible characters of the group $G$.

\subsubsection*{Spectrum of di-Cayley graphs in terms of the spectrum of the Cayley covers}
It happens that the spectrum of the bi/di-Cayely (sum) graphs can be given in terms of the corresponding one of the Cayley (sum) covers.

The following result was previously proved in \cite{GL} (see Theorem 3.2) for bi-Cayley graphs in the particular case of $G$ abelian and $S_\ell$, $S_r$ symmetric subsets of $G$. 
We now generalize this result to arbitrary bi/di-Cayley graphs for any $G$ in the case when the sets $S_\ell$, $S_r$, $S_m$ are normal (not requiring that they are symmetric sets).

\begin{thm} \label{thm: Spec DX}
Let $G$ be a finite group and $S_\ell$, $S_r$, $S_m$ normal (automatic if $G$ is abelian) subsets of $G$. 
Then, the eigenvalues of the bi-Cayley graph $BX(G;S_\ell,S_r,S_m)$ are 
\begin{equation} \label{eq: autovalores BX}
	\lambda_{BX}^{\pm} (\chi) = \tfrac{ \lambda_{S_\ell,\chi} + \lambda_{S_r,\chi}}{2} \pm 
	\sqrt{ \Big(\tfrac{\lambda_{S_\ell,\chi} - \lambda_{S_r,\chi}}{2} \Big)^2 + |\lambda_{S_m,\chi}|^2 },
\end{equation}	
and the eigenvalues of the di-Cayley graph $DX(G;S_\ell,S_r,S_m)$ are given by 
\begin{equation} \label{eq: autovalores DX}
	\lambda_{DX}^{\pm} (\chi) = \tfrac{\lambda_{S_\ell,\chi} + \lambda_{S_r,\chi}}{2} \pm 
	\sqrt{ \Big( \tfrac{\lambda_{S_\ell, \chi} - \lambda_{S_r,\chi}}{2} \Big)^2 + (\lambda_{S_m,\chi})^2},
\end{equation}	
where $\lambda_{S_\ell,\chi}$, $\lambda_{S_r,\chi}$ and $\lambda_{S_m,\chi}$ are the eigenvalues of $X(G,S_\ell)$, $X(G,S_r)$ and $X(G,S_m)$ respectively and $\chi \in \hat G$ runs over the set of all irreducible characters of $G$.
\end{thm}

\begin{proof} 
By Proposition \ref{prop: AdjMat}, we have that the adjacency matrix of the bi-Cayley graph $BX(G;S_\ell,S_r,S_m)$ is given by
	$$	
	A_{BX}= \begin{pmatrix}
		A_{X_\ell} & A_{X_m} \\[1.5mm]
		(A_{X_m})^t & A_{X_r} \end{pmatrix}.
	$$
Since $S_\ell, S_r, S_m$ are normal subsets, by Theorem \ref{thm: ASAT=ATAS} and Remark \ref{rem: BBt=BtB}, the matrices $A_{X_\ell}$, $A_{X_r}$, $A_{X_m}$ and $A_{X_m}^t$ are mutually commuting normal matrices. Hence, they diagonalize simultaneously, and $A_{BX}$ is equivalent to the block matrix 
	$$ 
	\Delta_{BX} = \begin{pmatrix}
		\Delta_{X_\ell} & \Delta_{X_m} \\[1.5mm]
		\overline{\Delta_{X_m}} & \Delta_{X_r}\end{pmatrix},
	$$
where each diagonal matrix $\Delta_{X_i}$ contains the eigenvalues of the respective Cayley graph $X(G,S_i)$. By Propositions \ref{prop: spec XGS G abelian} and \ref{prop: spec XGS}, these eigenvalues are denoted by $\lambda_{S_i,\chi_k}$ for $k=1,\dots,n$ (accounting for multiplicities). Thus,
	$$ 
	\Delta_{X_i} = \mathrm{diag}(\lambda_{S_i,\chi_1}, \lambda_{S_i,\chi_2}, \dots, \lambda_{S_i,\chi_n})
	$$
for $i \in \{\ell,r,m\}$. Furthermore, by applying a suitable permutation of rows and columns, the matrix $\Delta_{BX}$ is equivalent to the block diagonal matrix
	$$ \mathrm{diag}(B_1, B_2, \dots, B_n) $$
where each $2 \times 2$ block is given by
	$$ B_k = \begin{pmatrix}
		\lambda_{S_\ell, \chi_k} & \lambda_{S_m, \chi_k} \\[1.5mm]
		\overline{\lambda_{S_m, \chi_k}} & \lambda_{S_r, \chi_k}
	\end{pmatrix}. $$
Then, the eigenvalues of the bi-Cayley graph $BX(G;S_\ell,S_r,S_m)$ are precisely the eigenvalues of these $2 \times 2$ blocks, which are easily computed as:
	\begin{equation} \label{eq: autovalores BX*}
		\lambda_{BX}^{\pm} (\chi_i) = \tfrac{ \lambda_{S_\ell,\chi_i} + \lambda_{S_r,\chi_i}}{2} \pm 
		\sqrt{ \Big(\tfrac{\lambda_{S_\ell,\chi_i} - \lambda_{S_r,\chi_i}}{2} \Big)^2 + |\lambda_{S_m,\chi_i}|^2 },
	\end{equation}
for $i=1,\dots,n$.
	
The proof of \eqref{eq: autovalores DX} is completely analogous, using the adjacency matrix 
	$$ A_{DX} = \begin{pmatrix} A_{X_\ell} & A_{X_m} \\	A_{X_m} & A_{X_r} \end{pmatrix} $$ 
of the di-Cayley graph $DX(G;S_\ell,S_r,S_m)$. In this case, the permutation of rows and columns yields $2 \times 2$ blocks of the form
	$$ 
	B_k = \begin{pmatrix}
		\lambda_{S_\ell, \chi_k} & \lambda_{S_m, \chi_k} \\[1mm] \lambda_{S_m, \chi_k} & \lambda_{S_r, \chi_k} \end{pmatrix},
	$$
which directly leads to expression \eqref{eq: autovalores DX}, as desired.
\end{proof}

Now, we  give the eigenvalues of bi/di-Cayley sum graphs $X^+(G;S_\ell, S_r, S_m)$ for any $G$ in the case when the sets $S_\ell$, $S_r$ and $S_m$ are normal. 

\begin{thm} \label{thm: Spec DX+}
Let $G$ be a finite group and $S_\ell$, $S_r$, $S_m$ normal subsets of $G$. 
Then, the eigenvalues of the bi-Cayley sum graph $BX^+(G;S_\ell,S_r,S_m)$ are given by 
\begin{equation} \label{eq: autovalores BX+}
	\lambda_{BX^+}^{\pm} (\chi) = \tfrac{ \lambda_{S_\ell,\chi}^+ + \lambda_{S_r,\chi}^+}{2} \pm 
	\sqrt{ \Big(\tfrac{\lambda_{S_\ell,\chi}^+ - \lambda_{S_r,\chi}^+}{2} \Big)^2 + |\lambda_{S_m,\chi}^+|^2 },
\end{equation}	
and the eigenvalues of the di-Cayley graph $DX^+(G;S_\ell,S_r,S_m)$ are given by 
\begin{equation} \label{eq: autovalores DX+}
	\lambda_{DX^+}^{\pm} (\chi) = \tfrac{\lambda_{S_\ell,\chi}^+ + \lambda_{S_r,\chi}^+}{2} \pm 
	\sqrt{ \Big( \tfrac{\lambda_{S_\ell, \chi}^+ - \lambda_{S_r,\chi}^+}{2} \Big)^2 + (\lambda_{S_m,\chi}^+)^2}, 
\end{equation}	
where $\lambda_{S_\ell,\chi}^+$, $\lambda_{S_r,\chi}^+$ and $\lambda_{S_m,\chi}^+$ are the eigenvalues of $X^+(G,S_\ell)$, $X^+(G,S_r)$ and $X^+(G,S_m)$ respectively and $\chi \in \hat G$ runs over the set of all irreducible characters of $G$.
\end{thm}

\begin{proof}
The proof is completely analogous to the one of Theorem \ref{thm: Spec DX}.
\end{proof}

We now make some observations on the eigenvalues and on isospectrality.

\begin{rem}
Observe that if $S_m$ is symmetric, by ($iii$) in Remark \ref{rem: cosas1} we have the di-Cayley graph $DX^*(G;S_\ell,S_r,S_m)$ is the bi-Cayley graph $BX^*(G;S_\ell,S_r,S_m)$, and also $\lambda_{S_m,\chi}^* \in \R$ for any $\chi \in \hat G$, and hence expressions \eqref{eq: autovalores BX} and \eqref{eq: autovalores DX} and also \eqref{eq: autovalores BX+} and \eqref{eq: autovalores DX+} coincide, that is 
	$$ \lambda_{BX^*}^{\pm} (\chi) = \lambda_{DX^*}^{\pm} (\chi), $$ 
as it should be. 
\end{rem}

\begin{rem} We now make some comments on isospectrality.
	
\noindent ($a$) 
Notice that if $X^*(G,S)$ is isospectral to $X^*(G,S')$ for all the pairs $(S,S') \in \{(S_\ell, S_\ell'), (S_r,S_r'), (S_m,S_m')\}$, then $X^*(G;S_\ell, S_r, S_m)$ is isospectral to $X^*(G;S_\ell', S_r', S_m')$, by expressions \eqref{eq: autovalores BX}--\eqref{eq: autovalores DX} and \eqref{eq: autovalores BX+}--\eqref{eq: autovalores DX+} in Theorems \ref{thm: Spec DX}--\ref{thm: Spec DX+}.
	
\noindent ($b$) 
If $X(G,S)$ is isospectral to $X^+(G,S)$ for $S \in \{S_\ell, S_r, S_m\}$, then $X(G;S_\ell, S_r, S_m)$ is isospectral to $X^+(G;S_\ell, S_r, S_m)$, by expressions 
\eqref{eq: autovalores BX}--\eqref{eq: autovalores DX} and \eqref{eq: autovalores BX+}--\eqref{eq: autovalores DX+} in \linebreak Theorems \ref{thm: Spec DX} and \ref{thm: Spec DX+}.
By Proposition 2.5 and Corollary 2.9 in \cite{PV}, $X(G,S)$ is isospectral to $X^+(G,S)$ if $X(G,S)$, and hence $X^+(G,S)$, has symmetric spectrum. 
\end{rem}

\subsubsection*{Spectrum of bi/di-Cayley (sum) graphs in terms of the characters of $G$}
Putting together the previous results in the section, 
we get expressions for the eigenvalues of the bi/di-Cayley (sum) graphs only in terms of the characters of $G$ evaluated in the connection sets, i.e.\@ in terms of $\chi(S_\ell)$, $\chi(S_r)$ and $\chi(S_m)$ for any $\chi \in \hat G$.

For the bi/di-Cayely graphs we have the following.

\begin{coro} \label{coro: spec with chi(G)}
Let $G$ be a finite group and $S_\ell$, $S_r$, $S_m$ normal subsets of $G$. 
Then, the eigenvalues of the bi-Cayley graph $BX(G;S_\ell,S_r,S_m)$ are given by 
\begin{equation} \label{eq: autovalores BX chars}
	\lambda_{BX}^{\pm} (\chi) = \tfrac{1}{\chi(1)} \Bigg \{ \tfrac{ \chi(S_\ell) + \chi(S_r)}{2} \pm 
	\sqrt{ \Big(\tfrac{\chi(S_\ell) - \chi(S_r)}{2} \Big)^2 + |\chi(S_m)|^2 } \Bigg \},
\end{equation}	
and the eigenvalues of the di-Cayley graph $DX(G;S_\ell,S_r,S_m)$ are given by 
\begin{equation} \label{eq: autovalores DX chars}
	\lambda_{DX}^{\pm} (\chi) = \tfrac{1}{\chi(1)} \Bigg \{ \tfrac{ \chi(S_\ell) + \chi(S_r)}{2} \pm 
	\sqrt{ \Big(\tfrac{\chi(S_\ell) - \chi(S_r)}{2} \Big)^2 + \chi(S_m)^2 } \Bigg \},
\end{equation}	
where $\chi \in \hat G$ runs over the set of all irreducible characters of $G$.
\end{coro}

\begin{proof}
The result is obtained by introducing the expressions of the eigenvalues of Cayley graphs of Proposition \ref{prop: spec XGS}  inside the expressions for the eigenvalues of the bi/di-Cayley graphs in Theorem \ref{thm: Spec DX}.
\end{proof}

And for the bi/di-Cayley sum graphs we get the  following. 
\begin{coro} \label{coro: spec with chi(G)+}
Let $G$ be a finite group and $S_\ell$, $S_r$, $S_m$ normal subsets of $G$. 

\noindent $(a)$ 
The eigenvalues of the bi-Cayley sum graph $BX^+(G;S_\ell,S_r,S_m)$ are given as follows: 
\begin{itemize}
	\item If $\chi \in \hat{G}_{\R}$, then 

\begin{equation} \label{eq: autovalores BX+ chars real}
	\lambda_{BX^+}^{\pm} (\chi) = 
			\tfrac{1}{\chi(1)} \left\{ \tfrac{ \chi(S_\ell) + \chi(S_r)}{2} \pm 
			\sqrt{ \Big(\tfrac{ \chi(S_\ell) - \chi(S_r)}{2} \Big)^2 + |\chi(S_m)|^2 } \right\},  
\end{equation}	

\item If $\chi \notin \hat{G}_{\R}$, then
\begin{equation} \label{eq: autovalores BX+ chars}
	\lambda_{BX^+}^{\pm, \sigma} (\chi) = 
		\tfrac{\sigma}{\chi(1)} \left\{ \tfrac{ |\chi(S_\ell)| + |\chi(S_r)|}{2} \pm 
		\sqrt{ \Big(\tfrac{ |\chi(S_\ell)| - |\chi(S_r)|}{2} \Big)^2 + |\chi(S_m)|^2 } \right\}, 
\end{equation}	
where $\sigma \in \{\pm 1\}$ and $\chi \in \hat G$ runs over the set of all irreducible characters of $G$.
\end{itemize}

\noindent $(b)$ 
The eigenvalues of the di-Cayley sum graph $DX^+(G;S_\ell,S_r,S_m)$ are given as follows:

\begin{itemize}
	\item If $\chi \in \hat{G}_{\R}$, then 
	\begin{equation} \label{eq: autovalores DX+ chars real}
		\lambda_{DX^+}^{\pm} (\chi) = 
			\tfrac{1}{\chi(1)} \left\{  \tfrac{ \chi(S_\ell) + \chi(S_r)}{2} \pm 
			\sqrt{ \Big(\tfrac{ \chi(S_\ell) - \chi(S_r)}{2} \Big)^2 + \chi(S_m)^2 } \right\}.
	\end{equation}

	\item If $\chi \notin \hat{G}_{\R}$, then 
	\begin{equation} \label{eq: autovalores DX+ chars}
		\lambda_{DX^+}^{\pm, \sigma} (\chi) = 
			\tfrac{\sigma}{\chi(1)} \left\{  \tfrac{ |\chi(S_\ell)| + |\chi(S_r)|}{2} \pm 
			\sqrt{ \Big(\tfrac{ |\chi(S_\ell)| - |\chi(S_r)|}{2} \Big)^2 + |\chi(S_m)|^2 } \right\}, 
	\end{equation}
\end{itemize}  
where $\sigma \in \{\pm 1\}$ and $\chi \in \hat G$ runs over the set of all irreducible characters of $G$.
\end{coro}

Note that for $\chi$ a non-real character of $G$, the expressions for the eigenvalues for the di and bi-Cayley sum graphs coincide, that is $\lambda_{BX^+}^{\pm, \sigma} (\chi) = \lambda_{DX^+}^{\pm, \sigma} (\chi)$.

\begin{proof}
The result is obtained by introducing the expressions of the eigenvalues of Cayley sum graphs of Proposition \ref{prop: spec XGS+} inside the expressions for the eigenvalues of the bi/di-Cayley sum graphs in Theorem \ref{thm: Spec DX+}.
Indeed, in the cases where $\chi\notin\hat{G}_{\R}$, the eigenvalues of the individual blocks are of the form $\pm |\chi(S_i)|$. However, considering the expression under every possible combination of signs $\pm$ yields the exact same multiset of eigenvalues as the one generated using only the single parameter $\sigma \in \{\pm 1\}$. This allows us to write the formulas in the simplified form using the global sign $\sigma$.
\end{proof}

\begin{rem}
Using that $\chi(S \cup T)= \chi(S) +\chi(T) -\chi(S\cap T)$ for subsets $S,T$ of $G$, expressions \eqref{eq: autovalores BX chars} and \eqref{eq: autovalores DX chars} in the previous corollary can be put in terms of unions and intersections of $S_\ell$ and $S_r$; namely, 
	$$ \chi(S_\ell)+\chi(S_r)=\chi(S_\ell \cup S_r) + \chi(S_\ell \cap S_r). $$
In particular, if $S_\ell$ and $S_r$ are disjoint, we have $ \chi(S_\ell)+\chi(S_r)=\chi(S_\ell \cup S_r)$.

Moreover, if $G$ is abelian, then $\chi(-T)=-\chi(T)$ for any $T\subset G$ and hence 
	$$ \chi(S_\ell)-\chi(S_r) = \chi(S_\ell \cup -S_r) + \chi(S_\ell \cap -S_r) .$$
In this way, if $S_r$ is symmetric, i.e.\@ $S_r = -S_r$, the expressions for the eigenvalues depends on the number $\chi(S_\ell)+\chi(S_r)$. 

Summing up, if $G$ is abelian, $S_\ell, S_r$ are disjoint subsets of $G$ with $S_r$ symmetric, expression \eqref{eq: autovalores DX chars} for the eigenvalues in the no-sum case, takes the simple form 
 \begin{equation} \label{eq: autovalores DX chars simple}
 	\lambda_{DX}^{\pm} (\chi) = \tfrac{1}{\chi(1)} \big\{ \chi_{\ell,r} \pm 
 	\sqrt{ \chi_{\ell,r}^2 + \chi(S_m)^2} \big \},
 \end{equation}	
where $\chi_{\ell,r}= \frac{\chi(S_\ell) + \chi(S_r)}2$, and similarly for the expression \eqref{eq: autovalores BX chars} for $\lambda_{BX}^{\pm} (\chi)$.
\end{rem}

\begin{rem}
There are families of non-abelian groups $G$ having real characters and hence the Cayley (sum) graphs $X^*(G,S)$ have real spectrum. 
For instance: dihedral groups $\mathbb{D}_n$, quaternionic groups $\mathbb{Q}_{4n}$ for $n$ even and symmetric groups $\mathbb{S}_n$ have real characters.
Notice that, for the same groups $G$ as above, we also have that the bi/di-Cayley (sum) graphs $X^*(G;S_\ell,S_r,S_m)$ have real spectrum, as one can see from the expressions in Corollaries \ref{coro: spec with chi(G)} and \ref{coro: spec with chi(G)+}.
\end{rem}

\subsection*{Mirror di-Cayley (sum) graphs and integrality of the spectrum}
In the case of mirror bi/di-Cayley (sum) graphs the previous results for the eigenvalues get extremely simple. 
The following is automatic.

\begin{coro} \label{coro: spec mirror}
Let $G$ be a finite group and $S, T \subset G$ closed under conjugation.
Then, the eigenvalues of $BX^*(G;S,S,T)$ and $DX^*(G;S,S,T)$ are given by
\begin{equation} \label{eq: spec mirror}
	\lambda_{BX^*}^{\pm} (\chi) = \lambda_{S,\chi}^* \pm  |\lambda_{T,\chi}^*|,
		\qquad \text{and} \qquad
	\lambda_{DX^*}^{\pm} (\chi) = \lambda_{S,\chi}^* \pm  \lambda_{T,\chi}^*,
\end{equation}
where $\lambda_{S,\chi}^*$ and $\lambda_{T,\chi}^*$ are respectively the eigenvalues of $X^*(G,S)$ and $X^*(G,T)$ corresponding to $\chi$, where $\chi \in \hat G$ runs over the set of all irreducible characters of $G$.
\end{coro}

\begin{proof}
Just use that $S_\ell=S_r=S$ in expressions \eqref{eq: autovalores BX}--\eqref{eq: autovalores DX} in Theorem \ref{thm: Spec DX} and \eqref{eq: autovalores BX+}--\eqref{eq: autovalores DX+} in Theorem \ref{thm: Spec DX+}. 	
\end{proof}

\begin{rem} \label{rem: mirror graph products}
A particular family of mirror di-Cayley (sum) graphs was recently investigated by the authors in \cite{ChP}. Specifically, when the di-connection set is restricted to $T = S$, $T = \{e\}$, or $T = S \cup \{e\}$, the resulting mirror graphs $X^*(G; S, S, T)$ can be structurally described as standard graph products (such as the tensor, Cartesian, or strong product) of the underlying Cayley (sum) cover $X^*(G, S)$ with the graph $K_2$ (or $P_2$). In these specific cases, the spectrum of the mirror graph can be directly obtained from the spectrum of the cover by applying the well-known spectral formulas for graph products.
\end{rem}

In this case, we can say something about the integrality of the MDCGs in terms of the integrality of the Cayley (sum) covers.

\begin{prop} \label{prop: integral graph}
Let $G$ be a group and $S,T$ subsets of $G$.
		
	\begin{enumerate}[$(a)$]
		\item $X^*(G,S)$ and $X^*(G,T)$ are integral if and only if $DX^*(G;S,S,T)$ is integral. \sk    
			
		\item If $X^*(G,S)$ and $X^*(G,T)$ are integral, then $BX^*(G;S,S,T)$ is integral. 
		Conversely, if $BX^*(G;S,S,T)$ is integral then $X^*(G,S)$ is integral and $|Spec(X^*(G,T))| \subset \Z$.  \sk 
	\end{enumerate}
\end{prop}
	
\begin{proof}
Clearly, by \eqref{eq: spec mirror}, if $X^*(G,S)$ and $X^*(G,T)$ are integral then $DX^*(G;S,S,T)$ and $BX^*(G;S,S,T)$ are both integral. So, it remains to see the converses. \sk 
		
\noindent ($a$) If $\lambda_{DX^*}^{\pm} (\chi) = \lambda_{S,\chi}^* \pm  \lambda_{T,\chi}^* \in \Z$ we have that
$\lambda_{S,\chi}^* + \lambda_{T,\chi}^*=k$ and $\lambda_{S,\chi}^* - \lambda_{T,\chi}^*=l$ for some $k,l \in \Z$, from where we get that 
	$$ \lambda_{S,\chi}^* = \tfrac{k+l}2 \qquad \text{and} \qquad \lambda_{T,\chi}^* = \tfrac{k-l}2. $$  
Since there are no (non-integral) rational eigenvalues for any graph, we must have that $k$ and $l$ have the same parity and thus  $\lambda_{S,\chi}^*$, $\lambda_{T,\chi}^* \in \Z$.

\sk 
		
\noindent ($b$)
If $\lambda_{BX^*}^{\pm} (\chi) \in \Z$, since $|\lambda_{T,\chi}^*| \in \R$, then $\lambda_{S, \chi}^* \in \R$. Using an argument analogous to that in the proof of $(a)$, we obtain that $\lambda_{S, \chi}^*$ and $|\lambda_{T, \chi}^*|$ must be integers.
\end{proof}

We now show that the integrality of the spectrum of $BX^*(G;S,S,T)$ is not enough to ensure that $X^*(G,S)$ and $X^*(G,T)$ are integral.
 
\begin{exam}
Consider the mirror di-Cayley graph $\G=BX(\Z_3; \Z_3^*, \Z_3^*, \{2\})$, which has spectrum 
	$ Spec(\G)=\{[3]^1, [1]^1, [0]^2, [-2]^2 \} \subset \Z$. 
On the other hand, the spectrum of the Cayley covers are given by 
	$ Spec \big( X(\Z_3,\Z_3^*) \big) = \{[2]^1, [-1]^2 \} \subset\Z$ 
and 			
	$ Spec\big( X(\Z_3,\{2\}) \big) = \{[1]^1, [-\tfrac{1}{2} + \tfrac{\sqrt{3}}{2}i]^1, [-\tfrac{1}{2} - \tfrac{\sqrt{3}}{2}i]^1 \} \subset \Z[i] \smallsetminus \Z$.
However, the eigenvalue of $X(\Z_3,\{2\})$ has integral norm equal to 1.
\hfill $\diamond$
\end{exam}

\section{Explicit computations} 
\label{sec: explicit computations}
In this final section we give explicit computations of the spectrum of some bi/di-Cayley (sum) graphs. We will consider six examples of both mirror and non-mirror bi/di-Cayley (sum) graphs with both abelian and non-abelian groups. We will perform the calculations in all detail.

In these examples we will use the results in Section \ref{sec: spec} to compute the eigenvalues of the Cayley (sum) graphs $X^*(G,S_\ell)$, $X^*(G,S_r)$ and $X^*(G,T)$ using characters of $G$ and then then the results of Section \ref{sec: spectrum of dicayleys} to compute the eigenvalues of the bi/di-Cayley (sum) graphs, namely Corollaries~\ref{coro: spec with chi(G)} and \ref{coro: spec with chi(G)+} for the non-mirror case and Corollary \ref{coro: spec mirror} for MDCGs.

Also, the sum of the eigenvalues is the number of loops, so if the graph is loopless the sum equals zero.

We begin with three examples of MDCGs with abelian and non-abelian groups.

\begin{exam}[\textit{Mirror, abelian}] \label{exam: 6.1}
Let $G=\Z_6$ and its subsets 
	$$	S = \{1, 2, 4, 5\} \qquad \text{and} \qquad T = \{1, 2\} $$
and consider the associated 
mirror bi/di-Cayley (sum) graphs $\G_{bi}^*=BX^*(\Z_6;S,S,T)$ and $\G_{di}^*=DX^*(\Z_6;S,S,T)$.
Since the group is abelian, the sum graphs $\G_{bi}^+$ and $\G_{di}^+$ are undirected and coincide.
Notice that $S$ is symmetric while $T$ is not, hence the Cayley covers are undirected while the edges connecting them are directed in the non-sum case. 
The sum graphs have loops.
The geometric representations of these graphs and of their Cayley covers are given, separately for more clarity, as follows

	\tikzset{
		cayley node/.style={circle, fill=black, inner sep=0pt, minimum size=2mm},
		edge S/.style={thick, color=black},
		edge T/.style={thick, color=purple, ->, >=stealth},
		edge T bi/.style={thick, color=purple}, 
		font label/.style={font=\scriptsize}
	}
	
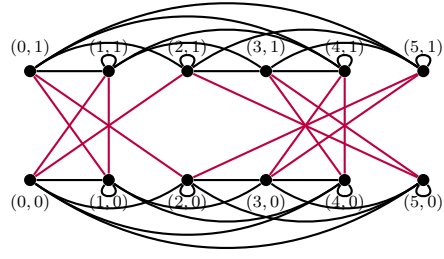
	\begin{figure}[H]
		\centering
		\begin{minipage}[c]{0.65\textwidth}
			\centering
			\begin{tikzpicture}[scale=0.8, every node/.append style={transform shape}]
				\foreach \i in {0,...,5} {
					\node[cayley node, label={[font label]above:{$(\i,1)$}}] (M\i) at (\i*1.3, 1.8) {};
					\node[cayley node, label={[font label]below:{$(\i,0)$}}] (P\i) at (\i*1.3, 0) {};
				}
				
				\foreach \x in {0,...,5} {
					\pgfmathtruncatemacro{\nextone}{int(mod(\x + 1, 6))}
					\pgfmathtruncatemacro{\nexttwo}{int(mod(\x + 2, 6))}
					
					\ifnum\nextone<\x
					\draw[edge S] (P\x) to[bend left=30] (P\nextone);
					\draw[edge S] (M\x) to[bend right=30] (M\nextone);
					\else
					\draw[edge S] (P\x) -- (P\nextone);
					\draw[edge S] (M\x) -- (M\nextone);
					\fi
					
					\ifnum\nexttwo<\x
					\draw[edge S] (P\x) to[bend left=40] (P\nexttwo);
					\draw[edge S] (M\x) to[bend right=40] (M\nexttwo);
					\else
					\draw[edge S] (P\x) to[bend right=20] (P\nexttwo);
					\draw[edge S] (M\x) to[bend left=20] (M\nexttwo);
					\fi
				}
				
				\foreach \x in {0,...,5} {
					\pgfmathtruncatemacro{\tone}{int(mod(\x + 1, 6))}
					\pgfmathtruncatemacro{\ttwo}{int(mod(\x + 2, 6))}
					
					\draw[edge T] (P\x) -- (M\tone);
					\draw[edge T] (P\x) -- (M\ttwo);
					\draw[edge T] (M\x) -- (P\tone);
					\draw[edge T] (M\x) -- (P\ttwo);
				}
			\end{tikzpicture}
			\caption*{$DX(\mathbb{Z}_6;S,S,T)$.}
			
	\vspace{0.3cm}
			
			\begin{tikzpicture}[scale=0.8, every node/.append style={transform shape}]
				\foreach \i in {0,...,5} {
					\node[cayley node, label={[font label]above:{$(\i,1)$}}] (M\i) at (\i*1.3, 1.8) {};
					\node[cayley node, label={[font label]below:{$(\i,0)$}}] (P\i) at (\i*1.3, 0) {};
				}
				
				\foreach \x in {0,...,5} {
					\pgfmathtruncatemacro{\nextone}{int(mod(\x + 1, 6))}
					\pgfmathtruncatemacro{\nexttwo}{int(mod(\x + 2, 6))}
					
					\ifnum\nextone<\x
					\draw[edge S] (P\x) to[bend left=30] (P\nextone);
					\draw[edge S] (M\x) to[bend right=30] (M\nextone);
					\else
					\draw[edge S] (P\x) -- (P\nextone);
					\draw[edge S] (M\x) -- (M\nextone);
					\fi
					
					\ifnum\nexttwo<\x
					\draw[edge S] (P\x) to[bend left=40] (P\nexttwo);
					\draw[edge S] (M\x) to[bend right=40] (M\nexttwo);
					\else
					\draw[edge S] (P\x) to[bend right=20] (P\nexttwo);
					\draw[edge S] (M\x) to[bend left=20] (M\nexttwo);
					\fi
				}
				
				\foreach \x in {0,...,5} {
					\pgfmathtruncatemacro{\tone}{int(mod(\x + 1, 6))}
					\pgfmathtruncatemacro{\ttwo}{int(mod(\x + 2, 6))}
					
					\draw[edge T bi] (P\x) -- (M\tone);
					\draw[edge T bi] (P\x) -- (M\ttwo);
				}
			\end{tikzpicture}
			\caption*{$BX(\mathbb{Z}_6;S,S,T)$.}
		\end{minipage}%
		\hfill
		\begin{minipage}[c]{0.35\textwidth}
			\centering
			\begin{tikzpicture}[scale=1.5]
				\foreach \i in {0,...,5} {
					\pgfmathtruncatemacro{\angle}{90 - \i * 60}
					\node[cayley node, label={[font label]{\angle}:{$\i$}}] (v\i) at ({\angle}:1.1) {};
				}
				
				\foreach \i in {0,...,5} {
					\pgfmathtruncatemacro{\next}{mod(\i + 1, 6)}
					\draw[edge S] (v\i) -- (v\next);
				}
				
				\foreach \i in {0,...,5} {
					\pgfmathtruncatemacro{\next}{mod(\i + 2, 6)}
					\ifnum\i<\next
					\draw[edge S] (v\i) -- (v\next);
					\fi
				}
				\draw[edge S] (v4) -- (v0);
				\draw[edge S] (v5) -- (v1);			
			\end{tikzpicture}
			\caption*{$X(\mathbb{Z}_6,S)$.}
		\end{minipage}
	\end{figure}
	\vspace{-1cm}
	\begin{figure}[H]
		\centering
		\begin{minipage}[c]{0.65\textwidth}
			\centering
			\begin{tikzpicture}[scale=0.8, every node/.append style={transform shape}]
				\foreach \u in {0,...,5} {
					\node[cayley node, label={[font label]above:{$(\u,1)$}}] (B\u) at (\u*1.3, 1.8) {};
					\node[cayley node, label={[font label]below:{$(\u,0)$}}] (A\u) at (\u*1.3, 0) {};
				}
				
				\foreach \u in {0,...,5} {
					\foreach \s in {1,2,4,5} {
						\pgfmathtruncatemacro{\v}{mod(\s - \u + 6, 6)}
						\ifnum\u=\v
						\draw[edge S] (A\u) to[out=230,in=310,looseness=6] (A\u); 
						\fi
						\ifnum\u<\v
						\pgfmathtruncatemacro{\diff}{\v - \u}
						\ifnum\diff=1 \draw[edge S] (A\u) -- (A\v); \fi
						\ifnum\diff>1 \draw[edge S] (A\u) to[bend right=35] (A\v); \fi
						\fi
					}
				}
				
				\foreach \u in {0,...,5} {
					\foreach \s in {1,2,4,5} {
						\pgfmathtruncatemacro{\v}{mod(\s - \u + 6, 6)}
						\ifnum\u=\v
						\draw[edge S] (B\u) to[out=50,in=130,looseness=6] (B\u); 
						\fi
						\ifnum\u<\v
						\pgfmathtruncatemacro{\diff}{\v - \u}
						\ifnum\diff=1 \draw[edge S] (B\u) -- (B\v); \fi
						\ifnum\diff>1 \draw[edge S] (B\u) to[bend left=35] (B\v); \fi
						\fi
					}
				}
				
				\foreach \u in {0,...,5} {
					\foreach \t in {1,2} {
						\pgfmathtruncatemacro{\v}{mod(\t - \u + 6, 6)}
						\draw[edge T bi] (A\u) -- (B\v);
					}
				}
			\end{tikzpicture}
			\caption*{$BX^+(\mathbb{Z}_6;S,S,T) = DX^+(\mathbb{Z}_6;S,S,T)$.}
		\end{minipage}%
		\hfill
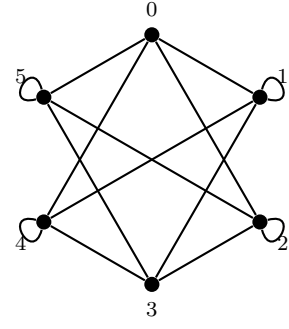
		\begin{minipage}[c]{0.35\textwidth}
			\centering
			\begin{tikzpicture}[scale=1.5]
				\foreach \i in {0,...,5} {
					\pgfmathtruncatemacro{\angle}{90 - \i * 60}
					\node[cayley node, label={[font label]{\angle}:{$\i$}}] (v\i) at ({\angle}:1.1) {};
				}
				
				\foreach \u in {0,...,5} {
					\foreach \s in {1,2,4,5} {
						\pgfmathtruncatemacro{\v}{mod(\s - \u + 6, 6)}
						\ifnum\u=\v
						\pgfmathtruncatemacro{\angle}{90 - \u * 60}
						\draw[edge S] (v\u) to[out=\angle-45,in=\angle+45,looseness=8] (v\u);
						\fi
						\ifnum\u<\v
						\draw[edge S] (v\u) -- (v\v);
						\fi
					}
				}
			\end{tikzpicture}
			\caption*{$X^+(\mathbb{Z}_6,S)$.}
		\end{minipage}
	\end{figure}
	
We now compute the spectrum of the graphs $\G_{di}$, $\G_{bi}$ and $\G_{di}^+$. The eigenvalues of these graphs are given in terms of the irreducible characters of $\Z_6$, which are of the form 
\begin{equation}
	\chi_k(g) = \omega^{k \cdot g} = e^{\frac{\pi i}{3} k \cdot g}, \qquad (0 \le k \le 5),
\end{equation}
where $\omega = e^{\frac{2\pi i}{6}}$. 
We give the character table of $\Z_6$; for $j,k \in \{0,\ldots,5\}$ we have
\begin{table}[H]
	\centering
	\label{tab:caracteres_z6}
	\begin{tabular}{ccccccc}
		\hline
		$j$	& $0$ & $1$ & $2$ & $3$ & $4$ & $5$ \\ \hline
		$\chi_0(j)$ & $1$ & $1$ & $1$ & $1$ & $1$ & $1$ \\
		$\chi_1(j)$ & $1$ & $\omega$ & $\omega^2$ & $-1$ & $\omega^4$ & $\omega^5$ \\
		$\chi_2(j)$ & $1$ & $\omega^2$ & $\omega^4$ & $1$ & $\omega^2$ & $\omega^4$ \\
		$\chi_3(j)$ & $1$ & $-1$ & $1$ & $-1$ & $1$ & $-1$ \\
		$\chi_4(j)$ & $1$ & $\omega^4$ & $\omega^2$ & $1$ & $\omega^4$ & $\omega^2$ \\
		$\chi_5(j)$ & $1$ & $\omega^5$ & $\omega^4$ & $-1$ & $\omega^2$ & $\omega$ \\ \hline
	\end{tabular}
	\caption{Character table of the group $\Z_6$.}
\end{table}

Using this table, we compute $\chi_k(S)$ and $\chi_k(T)$ for $k=0,1,2,3,4,5$:

For $k=0$, we have that $\lambda_{\G_{di}}^\pm(\chi_0) = \chi_0(S) \pm \chi_0(T) = 4 \pm 2$ since
	\begin{align*}
		 \chi_0(S) &= \chi_0(1)+\chi_0(2)+\chi_0(4)+\chi_0(5) = 1+1+1+1 = 4, \\
		 \chi_0(T) &= \chi_0(1)+\chi_0(2) = 1+1 = 2. 
	\end{align*}
For $k=1$, we have that $\lambda_{\G_{di}}^\pm(\chi_1) = \chi_1(S) \pm \chi_1(T) = \pm\sqrt{3}i$ since
	\begin{align*}
		\chi_1(S) &= \chi_1(1)+\chi_1(2)+\chi_1(4)+\chi_1(5) = \omega + \omega^2 + \omega^4 + \omega^5 = 0, \\
		\chi_1(T) &= \chi_1(1)+\chi_1(2) = \omega + \omega^2 = \sqrt{3}i.
	\end{align*}
For $k = 2$, we have that $\lambda_{\G_{di}}^\pm(\chi_2) = \chi_2(S) \pm \chi_2(T) = -2\pm(-1)$ since
	\begin{align*}
		\chi_2(S) &= \chi_2(1)+\chi_2(2)+\chi_2(4)+\chi_2(5) =  \omega^2 + \omega^4 + \omega^2 + \omega^4 = -2, \\
		\chi_2(T) &= \chi_2(1)+\chi_2(2) = \omega^2 + \omega^4 = -1. 
	\end{align*}
For $k = 3$, we have that $\lambda_{\G_{di}}^\pm(\chi_3)= \chi_3(S) \pm \chi_3(T) = 0$, with multiplicity $2$, since
	\begin{align*}
		\chi_3(S) &= \chi_3(1)+\chi_3(2)+\chi_3(4)+\chi_3(5) = -1+1+1+(-1) = 0, \\
		\chi_3(T) &= \chi_3(1)+\chi_3(2) = -1+1 = 0.
	\end{align*}
For $k=4$, we have that $\lambda_{\G_{di}}^\pm(\chi_4) = \chi_4(S) \pm \chi_4(T) = -2\pm(-1)$ since
	\begin{align*}
		\chi_4(S) & = \chi_4(1)+\chi_4(2)+\chi_4(4)+\chi_4(5) = \omega^4+\omega^2+\omega^4+\omega^2 = -2,\\
		\chi_4(T) & = \chi_4(1)+\chi_4(2) = \omega^4+\omega^2 = -1.
	\end{align*}
For $k=5$, we have that $\lambda_{\G_{di}}^\pm(\chi_5) = \chi_5(S) \pm \chi_5(T) = \pm\sqrt{3}i$ since
	\begin{align*}
		\chi_5(S) & = \chi_5(1)+\chi_5(2)+\chi_5(4)+\chi_5(5) = \omega^5+\omega^4+\omega^2+\omega = 0,\\
		\chi_5(T) & = \chi_5(1)+\chi_5(2) = \omega^5+\omega^4 = -\sqrt{3}i.
	\end{align*}
Putting together all this information, we have that the spectrum of $\G_{di}$ 
is given by
$$ Spec(\G_{di}) = \{[6]^1,[2]^1,[0]^2,[-1]^2,[-3]^2,[\sqrt{3}i]^2,[-\sqrt{3}i]^2\}.$$
That is $\G$ is an abelian mirror di-Cayley graph with spectrum $Spec(\G_{di}) \subset \Z[\sqrt 3 i]$.

In contrast, for the mirror bi-Cayley graph $\G_{bi}=BX(\mathbb{Z}_6; S, S, T)$ the eigenvalues associated with each irreducible character $\chi_k$ are given by 
\begin{equation}
	\lambda_{BX}^\pm(\chi_k) = \chi_k(S) \pm |\chi_k(T)|.
\end{equation}
Evaluating these expressions, the spectrum of the mirror bi-Cayley graph is given by
$$ Spec(\G_{bi}) = \{ [6]^1, [2]^1, [0]^2, [-1]^2, [-3]^2, [\sqrt{3}]^2, [-\sqrt{3}]^2 \} \subset \Z[\sqrt 3].$$
Relative to  the spectrum of the mirror di-Cayley sum graph $\G_{di}^+ = DX^+(\Z_6;S,S,T)$ we obtain
$$ Spec(\G_{di}^+) = \{ [6]^1, [3]^1, [2]^1, [1]^1, [0]^2, [-1]^1, [-3]^1, [\sqrt{3}]^2, [-\sqrt{3}]^2 \} \subset \Z[\sqrt 3].  $$
Notice that the sum of the eigenvalues equals 8, which is the number of loops. No graph in this example is integral.
\hfill $\diamond$
\end{exam}

\begin{exam}[\textit{Mirror, abelian}] \label{exam: 6.2}
Let $G$ be the cyclic group $\mathbb{Z}_{15}$ and its subsets
$S=\{5,10\}$ and $T=\{3,6,9,12\}$.  
Consider the associated mirror bi/di-Cayley (sum) graphs $\G_{bi}^*=BX^*(\Z_{15};S,S,T)$ and $\G_{di}^*=DX^*(\Z_{15};S,S,T)$.
In this case both $S$ and $T$ are symmetric and normal, hence the graphs are undirected and $\G_{bi}=\G_{di}$ and $\G_{bi}^+= \G_{di}^+$. 

For these graphs we only give the geometric representations of the Cayley covers, which are as follows:
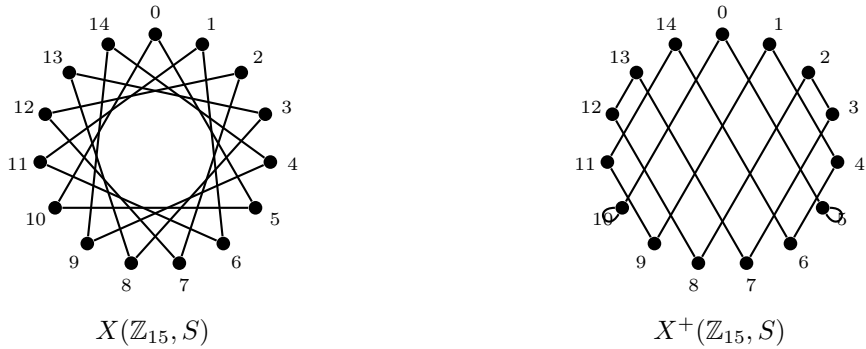
\begin{figure}[H]
	\centering
	\begin{subfigure}{0.5\textwidth}
		\centering
		\begin{tikzpicture}[
			scale=0.85,
			v/.style={circle, fill=black, inner sep=0pt, minimum size=1.8mm},
			e/.style={draw=black, thick}
			]
			\foreach \u in {0,1,...,14} {
				\node[v] (N\u) at ({90 - \u*24}: 1.8cm) {};
				\node[font=\tiny] at ({90 - \u*24}: 2.15cm) {\u};
			}
			
			\foreach \u in {0,1,...,14} {
				\pgfmathtruncatemacro{\v}{mod(\u + 5, 15)}
				\draw[e] (N\u) -- (N\v);
			}
		\end{tikzpicture}
		\caption*{$X(\mathbb{Z}_{15},S)$}
	\end{subfigure}%
	\hfil
	\begin{subfigure}{0.5 \textwidth}
		\centering
		\begin{tikzpicture}[
			scale=0.85,
			v/.style={circle, fill=black, inner sep=0pt, minimum size=1.8mm},
			e/.style={draw=black, thick}
			]
			\foreach \u in {0,1,...,14} {
				\node[v] (N\u) at ({90 - \u*24}: 1.8cm) {};
				\node[font=\tiny] at ({90 - \u*24}: 2.15cm) {\u};
			}
			
			\foreach \u in {0,1,...,14} {
				\foreach \s in {5,10} {
					\pgfmathtruncatemacro{\v}{mod(15 + \s - \u, 15)}
					\ifnum\u<\v
					\draw[e] (N\u) -- (N\v);
					\fi
					\ifnum\u=\v
					\pgfmathsetmacro{\ang}{90 - \u*24}
					\draw[e] (N\u) to[out=\ang-30, in=\ang+30, looseness=8] (N\u);
					\fi
				}
			}
		\end{tikzpicture}
		\caption*{$X^+(\mathbb{Z}_{15},S)$}
	\end{subfigure}
\end{figure}
Notice that they are disconneted graphs. Indeed, $X(\Z_{15},S)$ is a union of five $C_3$'s while $X^+(\Z_{15},S)$ is the union of two $C_6$'s and a 3-path with loops $\mathring P_3$. However, the bi/di-Cayley (sum) graphs are connected. All sum versions have loops.

The eigenvalues of $\G$ and $\G^+$ are given in terms of the irreducible characters of $\Z_{15}$, which are given by
\begin{equation}
	\chi_k(g) = \omega^{k \cdot g} = e^{\frac{2 \pi i}{15} k \cdot g}, \qquad (0 \le k \le 14),
\end{equation}
where $\omega = e^{\frac{2\pi i}{15}}$. 
We give the character table of $\Z_{15}$; for $j,k \in \{0,\ldots,14\}$ we have
	
	\begin{table}[H]
		\centering
		\smallskip
		\footnotesize 
		\setlength{\tabcolsep}{3pt} 
		\begin{tabular}{c|ccccccccccccccc}
			$j$ & $0$ & $1$ & $2$ & $3$ & $4$ & $5$ & $6$ & $7$ & $8$ & $9$ & $10$ & $11$ & $12$ & $13$ & $14$ \\
			\hline
			$\chi_0(j)$  & $1$ & $1$ & $1$ & $1$ & $1$ & $1$ & $1$ & $1$ & $1$ & $1$ & $1$ & $1$ & $1$ & $1$ & $1$ \\
			$\chi_1(j)$  & $1$ & $\omega$ & $\omega^2$ & $\omega^3$ & $\omega^4$ & $\omega^5$ & $\omega^6$ & $\omega^7$ & $\omega^8$ & $\omega^9$ & $\omega^{10}$ & $\omega^{11}$ & $\omega^{12}$ & $\omega^{13}$ & $\omega^{14}$ \\
			$\chi_2(j)$  & $1$ & $\omega^2$ & $\omega^4$ & $\omega^6$ & $\omega^8$ & $\omega^{10}$ & $\omega^{12}$ & $\omega^{14}$ & $\omega$ & $\omega^3$ & $\omega^5$ & $\omega^7$ & $\omega^9$ & $\omega^{11}$ & $\omega^{13}$ \\
			$\chi_3(j)$  & $1$ & $\omega^3$ & $\omega^6$ & $\omega^9$ & $\omega^{12}$ & $1$ & $\omega^3$ & $\omega^6$ & $\omega^9$ & $\omega^{12}$ & $1$ & $\omega^3$ & $\omega^6$ & $\omega^9$ & $\omega^{12}$ \\
			$\chi_4(j)$  & $1$ & $\omega^4$ & $\omega^8$ & $\omega^{12}$ & $\omega$ & $\omega^5$ & $\omega^9$ & $\omega^{13}$ & $\omega^2$ & $\omega^6$ & $\omega^{10}$ & $\omega^{14}$ & $\omega^3$ & $\omega^7$ & $\omega^{11}$ \\
			$\chi_5(j)$  & $1$ & $\omega^5$ & $\omega^{10}$ & $1$ & $\omega^5$ & $\omega^{10}$ & $1$ & $\omega^5$ & $\omega^{10}$ & $1$ & $\omega^5$ & $\omega^{10}$ & $1$ & $\omega^5$ & $\omega^{10}$ \\
			$\chi_6(j)$  & $1$ & $\omega^6$ & $\omega^{12}$ & $\omega^3$ & $\omega^9$ & $1$ & $\omega^6$ & $\omega^{12}$ & $\omega^3$ & $\omega^9$ & $1$ & $\omega^6$ & $\omega^{12}$ & $\omega^3$ & $\omega^9$ \\
			$\chi_7(j)$  & $1$ & $\omega^7$ & $\omega^{14}$ & $\omega^6$ & $\omega^{13}$ & $\omega^5$ & $\omega^{12}$ & $\omega^4$ & $\omega^{11}$ & $\omega^3$ & $\omega^{10}$ & $\omega^2$ & $\omega^9$ & $\omega$ & $\omega^8$ \\
			$\chi_8(j)$  & $1$ & $\omega^8$ & $\omega$ & $\omega^9$ & $\omega^2$ & $\omega^{10}$ & $\omega^3$ & $\omega^{11}$ & $\omega^4$ & $\omega^{12}$ & $\omega^5$ & $\omega^{13}$ & $\omega^6$ & $\omega^{14}$ & $\omega^7$ \\
			$\chi_9(j)$  & $1$ & $\omega^9$ & $\omega^3$ & $\omega^{12}$ & $\omega^6$ & $1$ & $\omega^9$ & $\omega^3$ & $\omega^{12}$ & $\omega^6$ & $1$ & $\omega^9$ & $\omega^3$ & $\omega^{12}$ & $\omega^6$ \\
			$\chi_{10}(j)$ & $1$ & $\omega^{10}$ & $\omega^5$ & $1$ & $\omega^{10}$ & $\omega^5$ & $1$ & $\omega^{10}$ & $\omega^5$ & $1$ & $\omega^{10}$ & $\omega^5$ & $1$ & $\omega^{10}$ & $\omega^5$ \\
			$\chi_{11}(j)$ & $1$ & $\omega^{11}$ & $\omega^7$ & $\omega^3$ & $\omega^{14}$ & $\omega^5$ & $\omega$ & $\omega^{12}$ & $\omega^8$ & $\omega^4$ & $\omega^{10}$ & $\omega^6$ & $\omega^2$ & $\omega^{13}$ & $\omega^9$ \\
			$\chi_{12}(j)$ & $1$ & $\omega^{12}$ & $\omega^9$ & $\omega^6$ & $\omega^3$ & $1$ & $\omega^{12}$ & $\omega^9$ & $\omega^6$ & $\omega^3$ & $1$ & $\omega^{12}$ & $\omega^9$ & $\omega^6$ & $\omega^3$ \\
			$\chi_{13}(j)$ & $1$ & $\omega^{13}$ & $\omega^{11}$ & $\omega^9$ & $\omega^7$ & $\omega^5$ & $\omega^3$ & $\omega$ & $\omega^{14}$ & $\omega^{12}$ & $\omega^{10}$ & $\omega^8$ & $\omega^6$ & $\omega^4$ & $\omega^2$ \\
			$\chi_{14}(j)$ & $1$ & $\omega^{14}$ & $\omega^{13}$ & $\omega^{12}$ & $\omega^{11}$ & $\omega^{10}$ & $\omega^9$ & $\omega^8$ & $\omega^7$ & $\omega^6$ & $\omega^5$ & $\omega^4$ & $\omega^3$ & $\omega^2$ & $\omega$ \\
			\hline
		\end{tabular}
		\caption{Character table of $\mathbb{Z}_{15}$.}
		\label{tab:caracteres_z15}
	\end{table}
	
Using this table, we compute $\chi_k(S)$ and $\chi_k(T)$ for $k=0,1,\dots,14$.

For $k=0$, we have that $\lambda_{\G}^\pm(\chi_0) = \chi_0(S) \pm \chi_0(T) = 2 \pm 4$ since
		\begin{align*}
			\chi_0(S) &= \chi_0(5)+\chi_0(10) = 1+1 = 2, \\
			\chi_0(T) &= \chi_0(3)+\chi_0(6)+\chi_0(9)+\chi_0(12) = 1+1+1+1 = 4. 
		\end{align*}

For $k=1$, we have that $\lambda_{\G}^\pm(\chi_1) = \chi_1(S) \pm \chi_1(T) = -1 \pm (-1)$ since
		\begin{align*}
			\chi_1(S) &= \chi_1(5)+\chi_1(10) = \omega^5 + \omega^{10} = -1, \\
			\chi_1(T) &= \chi_1(3)+\chi_1(6)+\chi_1(9)+\chi_1(12) = \omega^3 + \omega^6 + \omega^9 + \omega^{12} = -1.
		\end{align*}

For $k=2$, we have that $\lambda_{\G}^\pm(\chi_2) = \chi_2(S) \pm \chi_2(T) = -1 \pm (-1)$ since
		\begin{align*}
			\chi_2(S) &= \chi_2(5)+\chi_2(10) = \omega^{10} + \omega^5 = -1, \\
			\chi_2(T) &= \chi_2(3)+\chi_2(6)+\chi_2(9)+\chi_2(12) = \omega^6 + \omega^{12} + \omega^3 + \omega^9 = -1. 
		\end{align*}
 
For $k=3$, we have that $\lambda_{\G}^\pm(\chi_3) = \chi_3(S) \pm \chi_3(T) = 2 \pm (-1)$ since
		\begin{align*}
			\chi_3(S) &= \chi_3(5)+\chi_3(10) = 1+1 = 2, \\
			\chi_3(T) &= \chi_3(3)+\chi_3(6)+\chi_3(9)+\chi_3(12) = \omega^9 + \omega^3 + \omega^{12} + \omega^6 = -1.
		\end{align*}

For $k=4$, we have that $\lambda_{\G}^\pm(\chi_4) = \chi_4(S) \pm \chi_4(T) = -1 \pm (-1)$ since
		\begin{align*}
			\chi_4(S) &= \chi_4(5)+\chi_4(10) = \omega^5 + \omega^{10} = -1, \\
			\chi_4(T) &= \chi_4(3)+\chi_4(6)+\chi_4(9)+\chi_4(12) = \omega^{12} + \omega^9 + \omega^6 + \omega^3 = -1.
		\end{align*}
 
For $k=5$, we have that $\lambda_{\G}^\pm(\chi_5) = \chi_5(S) \pm \chi_5(T) = -1 \pm 4$ since
		\begin{align*}
			\chi_5(S) &= \chi_5(5)+\chi_5(10) = \omega^{10} + \omega^5 = -1, \\
			\chi_5(T) &= \chi_5(3)+\chi_5(6)+\chi_5(9)+\chi_5(12) = 1+1+1+1 = 4.
		\end{align*}
 
For $k=6$, we have that $\lambda_{\G}^\pm(\chi_6) = \chi_6(S) \pm \chi_6(T) = 2 \pm (-1)$ since
		\begin{align*}
			\chi_6(S) &= \chi_6(5)+\chi_6(10) = 1+1 = 2, \\
			\chi_6(T) &= \chi_6(3)+\chi_6(6)+\chi_6(9)+\chi_6(12) = \omega^3 + \omega^6 + \omega^9 + \omega^{12} = -1.
		\end{align*}
 
For $k=7$, we have that $\lambda_{\G}^\pm(\chi_7) = \chi_7(S) \pm \chi_7(T) = -1 \pm (-1)$ since
		\begin{align*}
			\chi_7(S) &= \chi_7(5)+\chi_7(10) = \omega^5 + \omega^{10} = -1, \\
			\chi_7(T) &= \chi_7(3)+\chi_7(6)+\chi_7(9)+\chi_7(12) = \omega^6 + \omega^{12} + \omega^3 + \omega^9 = -1.
		\end{align*}

For $k=8$, we have that $\lambda_{\G}^\pm(\chi_8) = \chi_8(S) \pm \chi_8(T) = -1 \pm (-1)$ since
		\begin{align*}
			\chi_8(S) &= \chi_8(5)+\chi_8(10) = \omega^{10} + \omega^5 = -1, \\
			\chi_8(T) &= \chi_8(3)+\chi_8(6)+\chi_8(9)+\chi_8(12) = \omega^9 + \omega^3 + \omega^{12} + \omega^6 = -1.
		\end{align*}

For $k=9$, we have that $\lambda_{\G}^\pm(\chi_9) = \chi_9(S) \pm \chi_9(T) = 2 \pm (-1)$ since
		\begin{align*}
			\chi_9(S) &= \chi_9(5)+\chi_9(10) = 1+1 = 2, \\
			\chi_9(T) &= \chi_9(3)+\chi_9(6)+\chi_9(9)+\chi_9(12) = \omega^{12} + \omega^9 + \omega^6 + \omega^3 = -1.
		\end{align*}

For $k=10$, we have that $\lambda_{\G}^\pm(\chi_{10}) = \chi_{10}(S) \pm \chi_{10}(T) = -1 \pm 4$ since
		\begin{align*}
			\chi_{10}(S) &= \chi_{10}(5)+\chi_{10}(10) = \omega^5 + \omega^{10} = -1, \\
			\chi_{10}(T) &= \chi_{10}(3)+\chi_{10}(6)+\chi_{10}(9)+\chi_{10}(12) = 1+1+1+1 = 4.
		\end{align*}
 
For $k=11$, we have that $\lambda_{\G}^\pm(\chi_{11}) = \chi_{11}(S) \pm \chi_{11}(T) = -1 \pm (-1)$ since
		\begin{align*}
			\chi_{11}(S) &= \chi_{11}(5)+\chi_{11}(10) = \omega^{10} + \omega^5 = -1, \\
			\chi_{11}(T) &= \chi_{11}(3)+\chi_{11}(6)+\chi_{11}(9)+\chi_{11}(12) = \omega^3 + \omega^6 + \omega^9 + \omega^{12} = -1.
		\end{align*}
 
For $k=12$, we have that $\lambda_{\G}^\pm(\chi_{12}) = \chi_{12}(S) \pm \chi_{12}(T) = 2 \pm (-1)$ since
		\begin{align*}
			\chi_{12}(S) &= \chi_{12}(5)+\chi_{12}(10) = 1+1 = 2, \\
			\chi_{12}(T) &= \chi_{12}(3)+\chi_{12}(6)+\chi_{12}(9)+\chi_{12}(12) = \omega^6 + \omega^{12} + \omega^3 + \omega^9 = -1.
		\end{align*}
 
For $k=13$, we have that $\lambda_{\G}^\pm(\chi_{13}) = \chi_{13}(S) \pm \chi_{13}(T) = -1 \pm (-1)$ since
		\begin{align*}
			\chi_{13}(S) &= \chi_{13}(5)+\chi_{13}(10) = \omega^5 + \omega^{10} = -1, \\
			\chi_{13}(T) &= \chi_{13}(3)+\chi_{13}(6)+\chi_{13}(9)+\chi_{13}(12) = \omega^9 + \omega^3 + \omega^{12} + \omega^6 = -1.
		\end{align*}
 
For $k=14$, we have that $\lambda_{\G}^\pm(\chi_{14}) = \chi_{14}(S) \pm \chi_{14}(T) = -1 \pm (-1)$ since
		\begin{align*}
			\chi_{14}(S) &= \chi_{14}(5)+\chi_{14}(10) = \omega^{10} + \omega^5 = -1, \\
			\chi_{14}(T) &= \chi_{14}(3)+\chi_{14}(6)+\chi_{14}(9)+\chi_{14}(12) = \omega^{12} + \omega^9 + \omega^6 + \omega^3 = -1.
		\end{align*}

Then, the spectrum of the mirror di-Cayley graph $\G_{di} = DX(G;S,S,T)$ is given by
	$$Spec(\G_{di})=\{[6]^1,[3]^{6},[1]^{4},[0]^{8},[-2]^{9},[-5]^{2}\}, $$
while the spectrum of the mirror di-Cayley sum graph $\G_{di}^+ = DX^+(\Z_{15};S,S,T)$ takes the form
	$$ Spec(\G_{di}^+) = \{[6]^1,[5]^1,[3]^3,[2]^4,[1]^2,[0]^8,[-1]^2,[-2]^5,[-3]^3,[-5]^1\}. $$
In this way all bi/di-Cayley (sum) graphs are integral.
\hfill $\diamond$
\end{exam}

\begin{exam}[\textit{Mirror, non-abelian}] \label{exam: 6.3}
Consider the dicyclic group  
	$$\mathbb{Q}_{12} = \langle a, b \mid a^6 = 1, b^2 = a^3, bab^{-1} = a^{-1} \rangle$$
of order $12$. Identifying $a^i$ with $(i,0)$ and $a^ib$ with $(i,1)$ for $i \in \{0, 1, 2, 3, 4, 5\}$, we write
	$a=(1,0)$ and $b=(0,1)$.
Let 
$$ S=\{(1,0),(2,0),(4,0),(5,0)\} 
 \qquad \text{and} \qquad  
T=\{(0,1),(2,1),(4,1)\}$$ 
be two subsets of $\mathbb {Q}_{12}$, and consider the associated mirror bi/di-Cayley (sum) graphs $\G_{bi}^* = BX^*(\mathbb{Q}_{12}; S, S, T)$ and $\G_{di}^* = DX^*(\mathbb{Q}_{12}; S, S, T)$.

The graphical representations of these graphs are as follows:

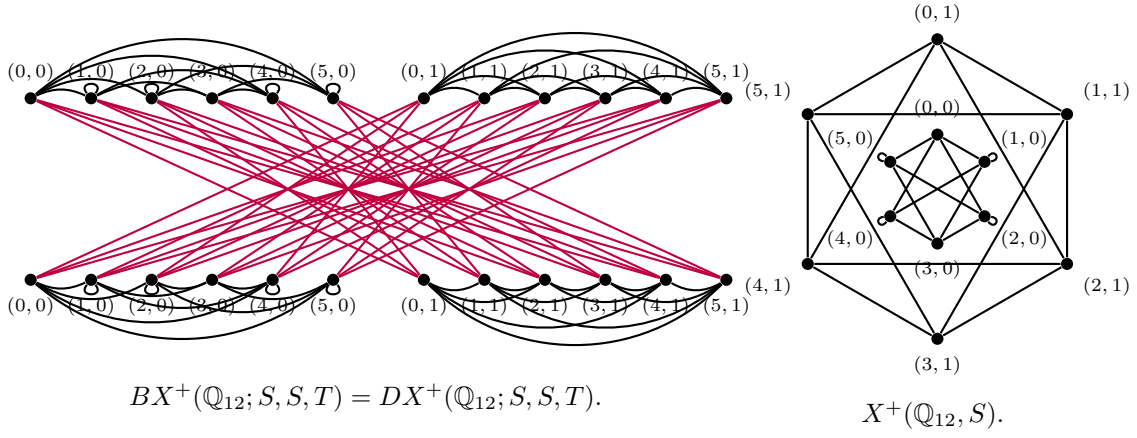
\begin{figure}[H]
	\centering
	\tikzset{
		cayley node/.style={circle, fill=black, inner sep=0pt, minimum size=1.6mm},
		edge S/.style={thick, color=black},
		edge T bi/.style={thick, color=purple},
		font label/.style={font=\tiny}
	}
	
	\begin{subfigure}[c]{0.65\textwidth}
		\centering
		\begin{tikzpicture}[scale=1, every node/.append style={transform shape}]
			\foreach \i in {0,...,5} {
				\node[cayley node, label={[font label]below:{$(\i,0)$}}] (A\i) at (\i*0.8, 0) {};
				\node[cayley node, label={[font label]below:{$(\i,1)$}}] (Ab\i) at (5.2 + \i*0.8, 0) {};
				\node[cayley node, label={[font label]above:{$(\i,0)$}}] (B\i) at (\i*0.8, 2.4) {};
				\node[cayley node, label={[font label]above:{$(\i,1)$}}] (Bb\i) at (5.2 + \i*0.8, 2.4) {};
			}
			
			\foreach \i in {1,2,4,5} {
				\draw[edge S] (A\i) to[out=230, in=310, looseness=5] (A\i);
			}
			\draw[edge S] (A0) to[bend right=25] (A1); \draw[edge S] (A0) to[bend right=30] (A2);
			\draw[edge S] (A0) to[bend right=35] (A4); \draw[edge S] (A0) to[bend right=40] (A5);
			\draw[edge S] (A1) to[bend right=25] (A3); \draw[edge S] (A1) to[bend right=30] (A4);
			\draw[edge S] (A2) to[bend right=25] (A3); \draw[edge S] (A2) to[bend right=35] (A5);
			\draw[edge S] (A3) to[bend right=25] (A4); \draw[edge S] (A3) to[bend right=30] (A5);
			
			\foreach \i/\j in {0/1, 1/2, 2/3, 3/4, 4/5} { \draw[edge S] (Ab\i) to[bend right=30] (Ab\j); }
			\foreach \i/\j in {0/2, 1/3, 2/4, 3/5} { \draw[edge S] (Ab\i) to[bend right=35] (Ab\j); }
			\draw[edge S] (Ab0) to[bend right=40] (Ab4); \draw[edge S] (Ab1) to[bend right=40] (Ab5);
			\draw[edge S] (Ab0) to[bend right=45] (Ab5);
			
			\foreach \i in {1,2,4,5} {
				\draw[edge S] (B\i) to[out=50, in=130, looseness=5] (B\i);
			}
			\draw[edge S] (B0) to[bend left=25] (B1); \draw[edge S] (B0) to[bend left=30] (B2);
			\draw[edge S] (B0) to[bend left=35] (B4); \draw[edge S] (B0) to[bend left=40] (B5);
			\draw[edge S] (B1) to[bend left=25] (B3); \draw[edge S] (B1) to[bend left=30] (B4);
			\draw[edge S] (B2) to[bend left=25] (B3); \draw[edge S] (B2) to[bend left=35] (B5);
			\draw[edge S] (B3) to[bend left=25] (B4); \draw[edge S] (B3) to[bend left=30] (B5);
			
			\foreach \i/\j in {0/1, 1/2, 2/3, 3/4, 4/5} { \draw[edge S] (Bb\i) to[bend left=30] (Bb\j); }
			\foreach \i/\j in {0/2, 1/3, 2/4, 3/5} { \draw[edge S] (Bb\i) to[bend left=35] (Bb\j); }
			\draw[edge S] (Bb0) to[bend left=40] (Bb4); \draw[edge S] (Bb1) to[bend left=40] (Bb5);
			\draw[edge S] (Bb0) to[bend left=45] (Bb5);
			
			\foreach \i in {0,...,5} {
				\foreach \k in {0,2,4} {
					\pgfmathtruncatemacro{\target}{mod(\k - \i + 12, 6)}
					\draw[edge T bi] (A\i) -- (Bb\target);
				}
			}
			\foreach \i in {0,...,5} {
				\foreach \k in {0,2,4} {
					\pgfmathtruncatemacro{\target}{mod(\i - \k + 12, 6)}
					\draw[edge T bi] (Ab\i) -- (B\target);
				}
			}
		\end{tikzpicture}
		\caption*{$BX^+(\mathbb{Q}_{12};S,S,T) = DX^+(\mathbb{Q}_{12};S,S,T)$.}
	\end{subfigure}%
	\hfill
	\begin{subfigure}[c]{0.35\textwidth}
		\centering
		\begin{tikzpicture}[scale=.9]
			\foreach \i in {0,...,5} {
				\pgfmathtruncatemacro{\angle}{90 - \i * 60}
				\node[cayley node, label={[font label]{\angle}:{$(\i,0)$}}] (In\i) at ({\angle}:0.8) {};
				\node[cayley node, label={[font label]{\angle}:{$(\i,1)$}}] (Out\i) at ({\angle}:2.2) {};
			}
			
			\foreach \i in {1,2,4,5} {
				\pgfmathtruncatemacro{\angle}{90 - \i * 60}
				\draw[edge S] (In\i) to[out=\angle-25, in=\angle+25, looseness=5] (In\i); 
			}
			
			\draw[edge S] (In0) -- (In1); \draw[edge S] (In0) -- (In2);
			\draw[edge S] (In0) -- (In4); \draw[edge S] (In0) -- (In5);
			\draw[edge S] (In1) -- (In3); \draw[edge S] (In1) -- (In4);
			\draw[edge S] (In2) -- (In3); \draw[edge S] (In2) -- (In5);
			\draw[edge S] (In3) -- (In4); \draw[edge S] (In3) -- (In5);
			
			\foreach \i/\j in {0/1, 1/2, 2/3, 3/4, 4/5, 5/0} {
				\draw[edge S] (Out\i) -- (Out\j);
			}
			\foreach \i/\j in {0/2, 1/3, 2/4, 3/5, 4/0, 5/1} {
				\draw[edge S] (Out\i) -- (Out\j);
			}
		\end{tikzpicture}
		\caption*{$X^+(\mathbb{Q}_{12},S)$.}
	\end{subfigure}
\end{figure}

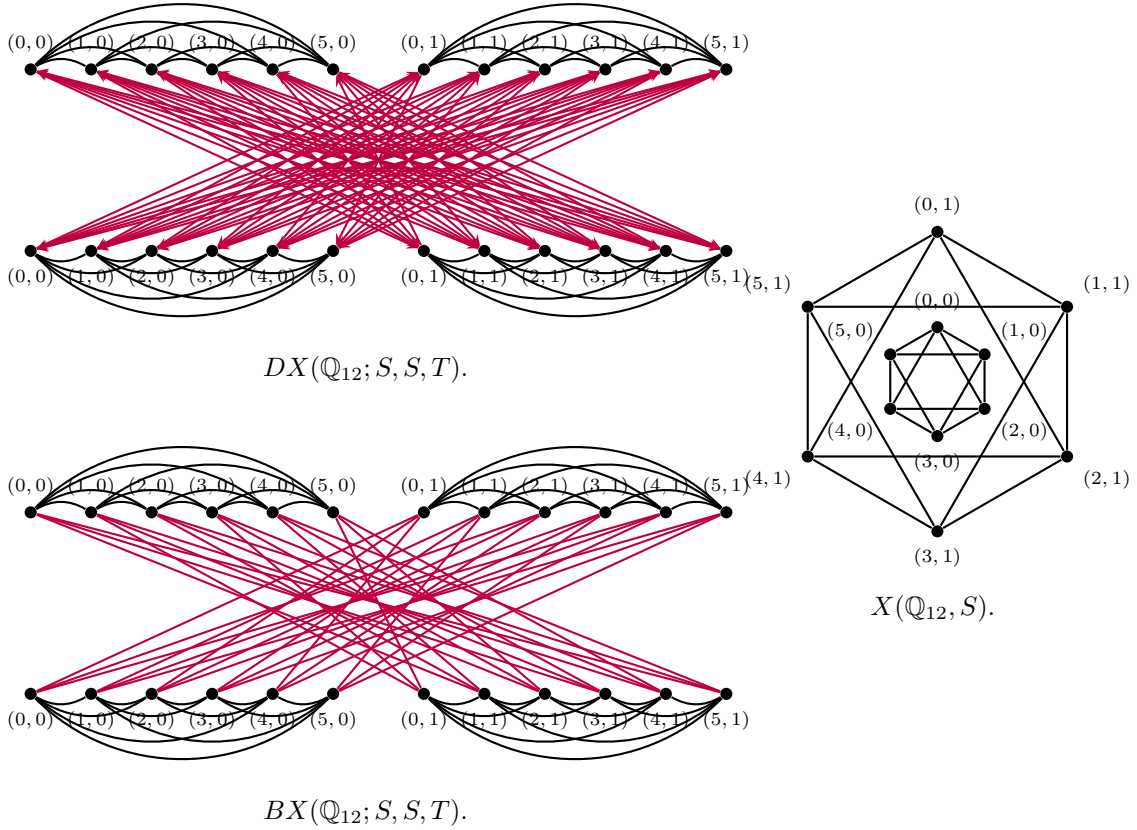
\begin{figure}[H]
	\centering
	\tikzset{
		cayley node/.style={circle, fill=black, inner sep=0pt, minimum size=1.6mm},
		edge S/.style={thick, color=black},
		edge T/.style={thick, color=purple, ->, >=stealth},
		edge T bi/.style={thick, color=purple},
		font label/.style={font=\tiny}
	}
	
	\begin{subfigure}[c]{0.65 \textwidth}
		\centering
		\begin{tikzpicture}[scale=1, every node/.append style={transform shape}]
			\foreach \i in {0,...,5} {
				\node[cayley node, label={[font label]below:{$(\i,0)$}}] (P\i) at (\i*0.8, 0) {};
				\node[cayley node, label={[font label]below:{$(\i,1)$}}] (Pb\i) at (5.2 + \i*0.8, 0) {};
				\node[cayley node, label={[font label]above:{$(\i,0)$}}] (M\i) at (\i*0.8, 2.4) {};
				\node[cayley node, label={[font label]above:{$(\i,1)$}}] (Mb\i) at (5.2 + \i*0.8, 2.4) {};
			}
			
			\foreach \i/\j in {0/1, 1/2, 2/3, 3/4, 4/5} {
				\draw[edge S] (P\i) to[bend right=30] (P\j); \draw[edge S] (Pb\i) to[bend right=30] (Pb\j);
			}
			\foreach \i/\j in {0/2, 1/3, 2/4, 3/5} {
				\draw[edge S] (P\i) to[bend right=35] (P\j); \draw[edge S] (Pb\i) to[bend right=35] (Pb\j);
			}
			\draw[edge S] (P0) to[bend right=40] (P4); \draw[edge S] (Pb0) to[bend right=40] (Pb4);
			\draw[edge S] (P1) to[bend right=40] (P5); \draw[edge S] (Pb1) to[bend right=40] (Pb5);
			\draw[edge S] (P0) to[bend right=45] (P5); \draw[edge S] (Pb0) to[bend right=45] (Pb5);
			
			\foreach \i/\j in {0/1, 1/2, 2/3, 3/4, 4/5} {
				\draw[edge S] (M\i) to[bend left=30] (M\j); \draw[edge S] (Mb\i) to[bend left=30] (Mb\j);
			}
			\foreach \i/\j in {0/2, 1/3, 2/4, 3/5} {
				\draw[edge S] (M\i) to[bend left=35] (M\j); \draw[edge S] (Mb\i) to[bend left=35] (Mb\j);
			}
			\draw[edge S] (M0) to[bend left=40] (M4); \draw[edge S] (Mb0) to[bend left=40] (Mb4);
			\draw[edge S] (M1) to[bend left=40] (M5); \draw[edge S] (Mb1) to[bend left=40] (Mb5);
			\draw[edge S] (M0) to[bend left=45] (M5); \draw[edge S] (Mb0) to[bend left=45] (Mb5);
			
			\foreach \i in {0,...,5} {
				\foreach \k in {0,2,4} {
					\pgfmathtruncatemacro{\target}{mod(\i + \k, 6)}
					\draw[edge T] (P\i) -- (Mb\target);
					\draw[edge T] (M\i) -- (Pb\target);
				}
				\foreach \k in {0,2,4} {
					\pgfmathtruncatemacro{\target}{mod(\i - \k + 3 + 12, 6)}
					\draw[edge T] (Pb\i) -- (M\target);
					\draw[edge T] (Mb\i) -- (P\target);
				}
			}
		\end{tikzpicture}
		\caption*{$DX(\mathbb{Q}_{12};S,S,T)$.}
		
		\vspace{0.5cm}
		
		\begin{tikzpicture}[scale=1, every node/.append style={transform shape}]
			\foreach \i in {0,...,5} {
				\node[cayley node, label={[font label]below:{$(\i,0)$}}] (P\i) at (\i*0.8, 0) {};
				\node[cayley node, label={[font label]below:{$(\i,1)$}}] (Pb\i) at (5.2 + \i*0.8, 0) {};
				\node[cayley node, label={[font label]above:{$(\i,0)$}}] (M\i) at (\i*0.8, 2.4) {};
				\node[cayley node, label={[font label]above:{$(\i,1)$}}] (Mb\i) at (5.2 + \i*0.8, 2.4) {};
			}
			
			\foreach \i/\j in {0/1, 1/2, 2/3, 3/4, 4/5} {
				\draw[edge S] (P\i) to[bend right=30] (P\j); \draw[edge S] (Pb\i) to[bend right=30] (Pb\j);
			}
			\foreach \i/\j in {0/2, 1/3, 2/4, 3/5} {
				\draw[edge S] (P\i) to[bend right=35] (P\j); \draw[edge S] (Pb\i) to[bend right=35] (Pb\j);
			}
			\draw[edge S] (P0) to[bend right=40] (P4); \draw[edge S] (Pb0) to[bend right=40] (Pb4);
			\draw[edge S] (P1) to[bend right=40] (P5); \draw[edge S] (Pb1) to[bend right=40] (Pb5);
			\draw[edge S] (P0) to[bend right=45] (P5); \draw[edge S] (Pb0) to[bend right=45] (Pb5);
			
			\foreach \i/\j in {0/1, 1/2, 2/3, 3/4, 4/5} {
				\draw[edge S] (M\i) to[bend left=30] (M\j); \draw[edge S] (Mb\i) to[bend left=30] (Mb\j);
			}
			\foreach \i/\j in {0/2, 1/3, 2/4, 3/5} {
				\draw[edge S] (M\i) to[bend left=35] (M\j); \draw[edge S] (Mb\i) to[bend left=35] (Mb\j);
			}
			\draw[edge S] (M0) to[bend left=40] (M4); \draw[edge S] (Mb0) to[bend left=40] (Mb4);
			\draw[edge S] (M1) to[bend left=40] (M5); \draw[edge S] (Mb1) to[bend left=40] (Mb5);
			\draw[edge S] (M0) to[bend left=45] (M5); \draw[edge S] (Mb0) to[bend left=45] (Mb5);
			
			\foreach \i in {0,...,5} {
				\foreach \k in {0,2,4} {
					\pgfmathtruncatemacro{\target}{mod(\i + \k, 6)}
					\draw[edge T bi] (P\i) -- (Mb\target);
				}
				\foreach \k in {0,2,4} {
					\pgfmathtruncatemacro{\target}{mod(\i - \k + 3 + 12, 6)}
					\draw[edge T bi] (Pb\i) -- (M\target);
				}
			}
		\end{tikzpicture}
		\caption*{$BX(\mathbb{Q}_{12};S,S,T)$.}
	\end{subfigure}%
	\hfill
	\begin{subfigure}[c]{0.35\textwidth}
		\centering
		\begin{tikzpicture}[scale=.9]
			\foreach \i in {0,...,5} {
				\pgfmathtruncatemacro{\angle}{90 - \i * 60}
				\node[cayley node, label={[font label]{\angle}:{$(\i,0)$}}] (In\i) at ({\angle}:0.8) {};
				\node[cayley node, label={[font label]{\angle}:{$(\i,1)$}}] (Out\i) at ({\angle}:2.2) {};
			}
			
			\foreach \i/\j in {0/1, 1/2, 2/3, 3/4, 4/5, 5/0} {
				\draw[edge S] (In\i) -- (In\j);
				\draw[edge S] (Out\i) -- (Out\j);
			}
			
			\foreach \i/\j in {0/2, 1/3, 2/4, 3/5, 4/0, 5/1} {
				\draw[edge S] (In\i) -- (In\j);
				\draw[edge S] (Out\i) -- (Out\j);
			}
		\end{tikzpicture}
		\caption*{$X(\mathbb{Q}_{12},S)$.}
	\end{subfigure}
\end{figure}

Notice that the covers are disconnected, $X(\Q_{12},S)$ being two isomorphic copies of the same graph, while $X^+(\Q_{12},S)$ has two non-isomorphic components.

We now present the irreducible representations of $\mathbb{Q}_{12}$. The group has four 1-dimensional representations ($\rho_1, \rho_2, \rho_3, \rho_4$) and two 2-dimensional representations ($\rho_5, \rho_6$), which are defined in the generators $a$ and $b$, where $\omega = e^{\frac{\pi i}{3}}$, as follows:
	
	\begin{table}[H]
		\centering
		\label{tab:rep_q12}
		\begin{tabular}{ccc}
			\hline
			Representation & $\rho_k(a)$ & $\rho_k(b)$ \\ \hline
			$\rho_1$ & $1$ & $1$ \\[1mm]
			$\rho_2$ & $1$ & $-1$ \\[1mm]
			$\rho_3$ & $-1$ & $i$ \\[1mm]
			$\rho_4$ & $-1$ & $-i$ \\[1mm]
			$\rho_5$ & $\big( \begin{smallmatrix} \omega & 0 \\ 0 & \omega^5 \end{smallmatrix} \big)$ & $\big( \begin{smallmatrix} 0 & 1 \\ -1 & 0 \end{smallmatrix} \big)$ \\[1mm]
			$\rho_6$ & $(\begin{smallmatrix} \omega^2 & 0 \\ 0 & \omega^4 \end{smallmatrix})$ & $(\begin{smallmatrix} 0 & 1 \\ 1 & 0 \end{smallmatrix})$ \\ \hline
		\end{tabular}
		\caption{Irreducible representations of the dicyclic group $\mathbb{Q}_{12}$.}
	\end{table}			
Using this character table and matrix representations, we compute $\rho_k(S) = \sum_{s \in S} \rho_k(s)$ and $\rho_k(T) = \sum_{t \in T} \rho_k(t)$ for each $k \in \{1, \ldots, 6\}$:
			
$\bullet$ For 1-dimensional representations, the local eigenvalues are given by $\rho_k(S) \pm \rho_k(T)$ for $DX$, and $\rho_k(S) \pm |\rho_k(T)|$ for $BX$:

For $k=1$: $\rho_1(S) = 4$ and $\rho_1(T) = 3$. 
		Thus, $\lambda^{\pm}(\rho_1) = 4 \pm 3 \in \{7, 1\}$.
					
For $k=2$: since $\rho_2(S) = 4$ and $\rho_2(T) = -3$, we have that $\lambda^{\pm}(\rho_2) = 4 \pm (-3) \in \{7, 1\}$.
					
For $k=3$: we have that $\rho_3(S) = \rho_3(a) + \rho_3(a^2) + \rho_3(a^4) + \rho_3(a^5) = -1 + 1 + 1 - 1 = 0$, and also 
	$\rho_3(T) = \rho_3(b) + \rho_3(a^2b) + \rho_3(a^4b) = i + i + i = 3i$. 
		Thus, $\lambda_{\G_{di}}^{\pm}(\rho_3) = 0 \pm 3i = \pm 3i$, and $\lambda_{\G_{bi}}^{\pm}(\rho_3) = 0 \pm |3i| = \pm 3$.
					
For $k=4$: $\rho_4(S) = 0$ and $\rho_4(T) = -3i$. 
	Thus, $\lambda_{\G_{di}}^{\pm}(\rho_4) = \pm 3i$, and $\lambda_{\G_{bi}}^{\pm}(\rho_4) = \pm 3$.
				
$\bullet$ For 2-dimensional representations ($d_k = 2$), the local spectrum is determined by the $4 \times 4$ block matrix
	$$ M_k = \begin{pmatrix} \rho_k(S) & \rho_k(T) \\ \rho_k(T)^* & \rho_k(S) \end{pmatrix},$$
where each eigenvalue of $M_k$ has multiplicity $d_k = 2$:

For $k=5$ we have: 
		\begin{align*}
			\rho_5(S) & = \big( \begin{smallmatrix} \omega + \omega^2 + \omega^4 + \omega^5 & 0 \\ 0 & \omega^5 + \omega^4 + \omega^2 + \omega \end{smallmatrix}\big)  = \big( \begin{smallmatrix} 0 & 0 \\ 0 & 0 \end{smallmatrix}\big), \\
			\rho_5(T) & = (I + \rho_5(a)^2 + \rho_5(a)^4) \rho_5(b) = \big( \begin{smallmatrix} 1 + \omega^2 + \omega^4 & 0 \\ 0 & 1 + \omega^4 + \omega^2 \end{smallmatrix}\big) \big( \begin{smallmatrix} 0 & 1 \\ -1 & 0 \end{smallmatrix} \big) = \big( \begin{smallmatrix} 0 & 0 \\ 0 & 0 \end{smallmatrix} \big).
		\end{align*}
Consequently, $M_5 = \mathbf{0}_{4 \times 4}$, contributing the eigenvalue $0$ with multiplicity $d_5= 2 \times 4 = 8$.
					

For $k=6$ we have
\begin{align*}
	& \rho_6(S) = \big( \begin{smallmatrix} \omega^2 + \omega^4 + \omega^2 + \omega^4 & 0 \\ 0 & \omega^4 + \omega^2 + \omega^4 + \omega^2 \end{smallmatrix} \big) = \big( \begin{smallmatrix} -2 & 0 \\ 0 & -2 \end{smallmatrix} \big), \\
	& \rho_6(T) = (I + \rho_6(a)^2 + \rho_6(a)^4)\rho_6(b) = \big( \begin{smallmatrix} 1 + \omega^4 + \omega^2 & 0 \\ 0 & 1 + \omega^2 + \omega^4 \end{smallmatrix} \big) \big( \begin{smallmatrix} 0 & 1 \\ 1 & 0 \end{smallmatrix} \big) = \big( \begin{smallmatrix} 0 & 0 \\ 0 & 0 \end{smallmatrix} \big). 	
\end{align*}
Thus, the $4 \times 4$ block matrix $M_6 = \big( \begin{smallmatrix} -2I & \mathbf{0} \\ \mathbf{0} & -2I \end{smallmatrix} \big)$ yields the eigenvalue $-2$ with multiplicity $4$. Since $d_6 = 2$, this contributes the eigenvalue $-2$ with a total multiplicity of $8$ to the spectrum of the graph.
			
%
%

Putting together all this information, we have that the spectrum of the mirror di-Cayley graph $\G_{di} = DX(\mathbb{Q}_{12}; S, S, T)$ is given by
$$ \operatorname{Spec}(\G_{di}) = \left\{ [7]^2, [1]^2, [0]^8, [-2]^8, [3i]^2, [-3i]^2 \right\}. $$
That is, $\G_{di}$ is a non-abelian mirror di-Cayley graph with spectrum $\operatorname{Spec}(\G_{di}) \subset \mathbb{Z}[i]$.

In contrast, for the mirror bi-Cayley graph $\G_{bi} = BX(\mathbb{Q}_{12}; S, S, T)$, evaluating the corresponding expressions yields the spectrum
$$ \operatorname{Spec}(\G_{bi}) = \left\{ [7]^2, [3]^2, [1]^2, [0]^8, [-2]^8, [-3]^2 \right\}. $$

Finally, we compute the spectrum of the mirror di-Cayley sum graph $\G_{di}^+ = DX^+(\mathbb{Q}_{12}; S, S, T)$ correctly accounting for the inversion operator on the representations, and we obtain
$$ \operatorname{Spec}(\G_{di}^+) = \left\{ [7]^2, [3]^2, [2]^2, [1]^2, [0]^8, [-2]^6, [-3]^2 \right\}. $$
We check that the sum of the eigenvalues is 8 which equals the number of loops.
So, $\G_{di}$ is Gaussian integral while $\G_{bi}$ and $\G_{di}^+$ are integral graphs.
\hfill $\diamond$
\end{exam}

To finish this work, we now consider three non-mirror bi/di-Cayley (sum) graphs.

\begin{exam}[\textit{Non-mirror, abelian}] \label{exam: 6.4}
Let $G = \mathbb{Z}_4 \times \mathbb{Z}_2$ and the subsets
$$ S_\ell = \{(0,1), (1,1), (3,0)\}, \quad S_r = \{(0,1), (1,0), (3,1)\}, \quad S_m = \{(0,1), (1,0), (1,1), (2,0)\}, $$
and consider the associated bi/di-Cayley (sum) graphs $\G_{bi}^*=BX^*(\Z_4\times\Z_2;S_\ell,S_r,S_m)$ and $\G_{di}^*=DX^*(\Z_4\times\Z_2;S_\ell,S_r,S_m)$.

The graphical representations of these graphs are as follows:
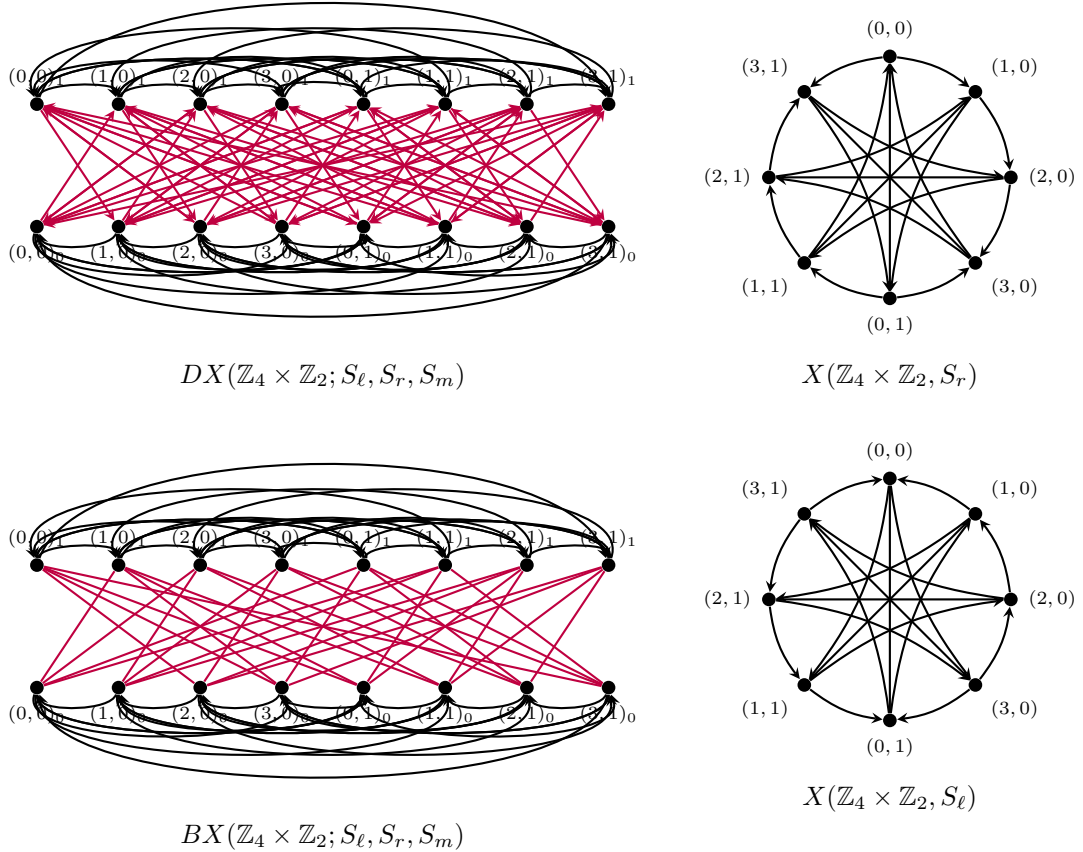
\tikzset{
	cnode/.style={circle, fill=black, inner sep=0pt, minimum size=1.8mm},
	eS/.style={thick, black, ->, >=stealth},
	eS-bi/.style={thick, black},
	eT/.style={thick, purple, ->, >=stealth},
	eT-bi/.style={thick, purple},
	flbl/.style={font=\tiny}
}
\begin{figure}[H]
	\centering
	\begin{subfigure}[c]{0.65\textwidth}
		\centering
	\begin{tikzpicture}[scale=.9]
		\def\xs{1.2}\def\ys{1.8}
		\foreach \x in {0,...,3} \foreach \y in {0,1} {
			\pgfmathtruncatemacro{\i}{\x + 4*\y}
			\node[cnode, label={[flbl]above:{$(\x,\y)_1$}}] (B\x\y) at (\i*\xs, \ys) {};
			\node[cnode, label={[flbl]below:{$(\x,\y)_0$}}] (A\x\y) at (\i*\xs, 0) {};
		}
		\foreach \x in {0,...,3} \foreach \y in {0,1} {
			\pgfmathtruncatemacro{\ny}{1-\y}
			
			\draw[eS-bi] (A\x\y) to[out=-90, in=-90, looseness=0.4] (A\x\ny);
			\pgfmathtruncatemacro{\vx}{mod(\x+1,4)}
			\draw[eS] (A\x\y) to[out=-100, in=-80, looseness=0.5] (A\vx\ny);
			\pgfmathtruncatemacro{\vx}{mod(\x+3,4)}
			\draw[eS] (A\x\y) to[out=-110, in=-70, looseness=0.6] (A\vx\y);
			
			\draw[eS-bi] (B\x\y) to[out=90, in=90, looseness=0.4] (B\x\ny);
			\pgfmathtruncatemacro{\vx}{mod(\x+1,4)}
			\draw[eS] (B\x\y) to[out=80, in=100, looseness=0.5] (B\vx\y);
			\pgfmathtruncatemacro{\vx}{mod(\x+3,4)}
			\draw[eS] (B\x\y) to[out=70, in=110, looseness=0.6] (B\vx\ny);
			
			\foreach \sx/\sy in {0/1, 1/0, 1/1, 2/0} {
				\pgfmathtruncatemacro{\vx}{mod(\x+\sx,4)}\pgfmathtruncatemacro{\vy}{mod(\y+\sy,2)}
				\draw[eT] (A\x\y) -- (B\vx\vy);
				\draw[eT] (B\x\y) -- (A\vx\vy);
			}
		}
	\end{tikzpicture}
		\caption*{$DX(\Z_4\times\Z_2; S_\ell, S_r, S_m)$}
		\vspace{0.5cm}
		
		\begin{tikzpicture}[scale=.9]
			\def\xs{1.2}\def\ys{1.8}
			\foreach \x in {0,...,3} \foreach \y in {0,1} {
				\pgfmathtruncatemacro{\i}{\x + 4*\y}
				\node[cnode, label={[flbl]above:{$(\x,\y)_1$}}] (B\x\y) at (\i*\xs, \ys) {};
				\node[cnode, label={[flbl]below:{$(\x,\y)_0$}}] (A\x\y) at (\i*\xs, 0) {};
			}
			\foreach \x in {0,...,3} \foreach \y in {0,1} {
				\pgfmathtruncatemacro{\ny}{1-\y}
				
				\draw[eS-bi] (A\x\y) to[out=-90, in=-90, looseness=0.4] (A\x\ny);
				\pgfmathtruncatemacro{\vx}{mod(\x+1,4)}
				\draw[eS] (A\x\y) to[out=-100, in=-80, looseness=0.5] (A\vx\ny);
				\pgfmathtruncatemacro{\vx}{mod(\x+3,4)}
				\draw[eS] (A\x\y) to[out=-110, in=-70, looseness=0.6] (A\vx\y);
				
				\draw[eS-bi] (B\x\y) to[out=90, in=90, looseness=0.4] (B\x\ny);
				\pgfmathtruncatemacro{\vx}{mod(\x+1,4)}
				\draw[eS] (B\x\y) to[out=80, in=100, looseness=0.5] (B\vx\y);
				\pgfmathtruncatemacro{\vx}{mod(\x+3,4)}
				\draw[eS] (B\x\y) to[out=70, in=110, looseness=0.6] (B\vx\ny);
				
				\foreach \sx/\sy in {0/1, 1/0, 1/1, 2/0} {
					\pgfmathtruncatemacro{\vx}{mod(\x+\sx,4)}\pgfmathtruncatemacro{\vy}{mod(\y+\sy,2)}
					\draw[eT-bi] (A\x\y) -- (B\vx\vy);
				}
			}
		\end{tikzpicture}
		\caption*{$BX(\Z_4\times\Z_2; S_\ell, S_r, S_m)$}
	\end{subfigure}%
	\hfill
	\begin{subfigure}[c]{0.35\textwidth}
		\centering
				\begin{tikzpicture}[scale=1]
			\foreach \x in {0,...,3} \foreach \y in {0,1} {
				\pgfmathtruncatemacro{\k}{\x + 4*\y}
				\node[cnode, label={[flbl]{90-\k*45}:{$(\x,\y)$}}] (M\x\y) at ({90-\k*45}: 1.6cm) {};
			}
			\foreach \x in {0,...,3} \foreach \y in {0,1} {
				\pgfmathtruncatemacro{\ny}{1-\y}
				\ifnum\y=0 \draw[eS-bi] (M\x\y) -- (M\x\ny); \fi
				\pgfmathtruncatemacro{\vx}{mod(\x+1,4)}\draw[eS] (M\x\y) to[bend left=15] (M\vx\y);
				\pgfmathtruncatemacro{\vx}{mod(\x+3,4)}\draw[eS] (M\x\y) to[bend right=15] (M\vx\ny);
			}
		\end{tikzpicture}
		\caption*{$X(\Z_4\times\Z_2, S_r)$}
		
		\vspace{0.5cm}

		\begin{tikzpicture}[scale=1]
	\foreach \x in {0,...,3} \foreach \y in {0,1} {
		\pgfmathtruncatemacro{\k}{\x + 4*\y}
		\node[cnode, label={[flbl]{90-\k*45}:{$(\x,\y)$}}] (N\x\y) at ({90-\k*45}: 1.6cm) {};
	}
	\foreach \x in {0,...,3} \foreach \y in {0,1} {
		\pgfmathtruncatemacro{\ny}{1-\y}
		\ifnum\y=0 \draw[eS-bi] (N\x\y) -- (N\x\ny); \fi
		\pgfmathtruncatemacro{\vx}{mod(\x+1,4)}\draw[eS] (N\x\y) to[bend left=15] (N\vx\ny);
		\pgfmathtruncatemacro{\vx}{mod(\x+3,4)}\draw[eS] (N\x\y) to[bend right=15] (N\vx\y);
		}
		\end{tikzpicture}
		\caption*{$X(\Z_4\times\Z_2, S_\ell)$}
	\end{subfigure}
\end{figure}

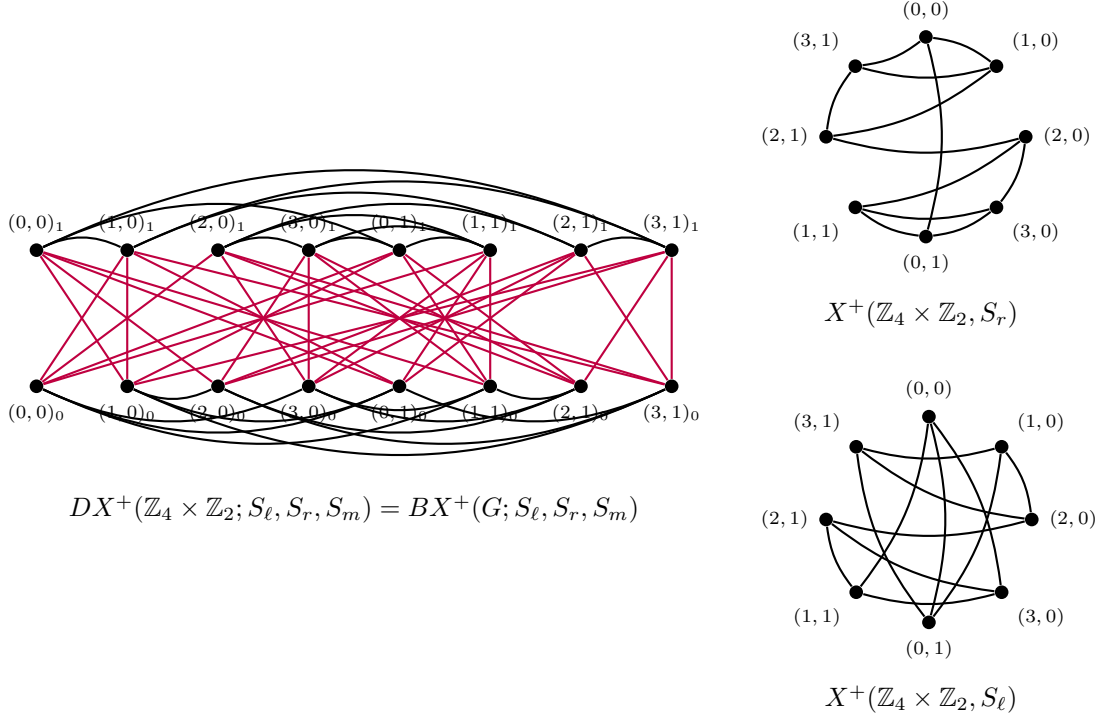
\begin{figure}[H]
	\centering
	\begin{subfigure}[c]{0.7\textwidth}
		\centering
		\begin{tikzpicture}[scale=1]
			\def\xs{1.2}\def\ys{1.8}
			\foreach \x in {0,...,3} \foreach \y in {0,1} {
				\pgfmathtruncatemacro{\i}{\x + 4*\y}
				\node[cnode, label={[flbl]above:{$(\x,\y)_1$}}] (B\x\y) at (\i*\xs, \ys) {};
				\node[cnode, label={[flbl]below:{$(\x,\y)_0$}}] (A\x\y) at (\i*\xs, 0) {};
			}
			\foreach \x in {0,...,3} \foreach \y in {0,1} {
				\pgfmathtruncatemacro{\u}{\x + 4*\y}
				\foreach \sx/\sy in {0/1, 1/1, 3/0} {
					\pgfmathtruncatemacro{\vx}{mod(\sx-\x+4,4)}\pgfmathtruncatemacro{\vy}{mod(\sy-\y+2,2)}
					\pgfmathtruncatemacro{\v}{\vx + 4*\vy}
					\ifnum\u<\v \draw[eS-bi] (A\x\y) to[bend right=25] (A\vx\vy); \fi
					\ifnum\u=\v \draw[eS-bi] (A\x\y) to[out=-120,in=-60,looseness=6] (A\x\y); \fi
				}
				\foreach \sx/\sy in {0/1, 1/0, 3/1} {
					\pgfmathtruncatemacro{\vx}{mod(\sx-\x+4,4)}\pgfmathtruncatemacro{\vy}{mod(\sy-\y+2,2)}
					\pgfmathtruncatemacro{\v}{\vx + 4*\vy}
					\ifnum\u<\v \draw[eS-bi] (B\x\y) to[bend left=25] (B\vx\vy); \fi
					\ifnum\u=\v \draw[eS-bi] (B\x\y) to[out=60,in=120,looseness=6] (B\x\y); \fi
				}
				\foreach \sx/\sy in {0/1, 1/0, 1/1, 2/0} {
					\pgfmathtruncatemacro{\vx}{mod(\sx-\x+4,4)}\pgfmathtruncatemacro{\vy}{mod(\sy-\y+2,2)}
					\draw[eT-bi] (A\x\y) -- (B\vx\vy);
				}
			}
		\end{tikzpicture}
		\caption*{$DX^+(\Z_4\times\Z_2; S_\ell, S_r, S_m) = BX^+(G; S_\ell, S_r, S_m)$}
		\vspace{0.5cm}
		
	\end{subfigure}%
	\hfill
	\begin{subfigure}[c]{0.3\textwidth}
		\centering
		\begin{tikzpicture}[scale=0.825]
			\foreach \x in {0,...,3} \foreach \y in {0,1} {
				\pgfmathtruncatemacro{\k}{\x + 4*\y}
				\node[cnode, label={[flbl]{90-\k*45}:{$(\x,\y)$}}] (M\x\y) at ({90-\k*45}: 1.6cm) {};
			}
			\foreach \x in {0,...,3} \foreach \y in {0,1} {
				\pgfmathtruncatemacro{\u}{\x + 4*\y}
				\foreach \sx/\sy in {0/1, 1/0, 3/1} {
					\pgfmathtruncatemacro{\vx}{mod(\sx-\x+4,4)}\pgfmathtruncatemacro{\vy}{mod(\sy-\y+2,2)}
					\pgfmathtruncatemacro{\v}{\vx + 4*\vy}
					\ifnum\u<\v \draw[eS-bi] (M\x\y) to[bend left=15] (M\vx\vy); \fi
					\ifnum\u=\v \draw[eS-bi] (M\x\y) to[out=90-\u*45-30,in=90-\u*45+30,looseness=6] (M\x\y); \fi
				}
			}
		\end{tikzpicture}
		\caption*{$X^+(\Z_4\times\Z_2, S_r)$}
		
		\vspace{0.5cm}

		\begin{tikzpicture}[scale=0.85]
	\foreach \x in {0,...,3} \foreach \y in {0,1} {
		\pgfmathtruncatemacro{\k}{\x + 4*\y}
		\node[cnode, label={[flbl]{90-\k*45}:{$(\x,\y)$}}] (N\x\y) at ({90-\k*45}: 1.6cm) {};
	}
	\foreach \x in {0,...,3} \foreach \y in {0,1} {
		\pgfmathtruncatemacro{\u}{\x + 4*\y}
		\foreach \sx/\sy in {0/1, 1/1, 3/0} {
			\pgfmathtruncatemacro{\vx}{mod(\sx-\x+4,4)}\pgfmathtruncatemacro{\vy}{mod(\sy-\y+2,2)}
			\pgfmathtruncatemacro{\v}{\vx + 4*\vy}
			\ifnum\u<\v \draw[eS-bi] (N\x\y) to[bend left=15] (N\vx\vy); \fi
			\ifnum\u=\v \draw[eS-bi] (N\x\y) to[out=90-\u*45-30,in=90-\u*45+30,looseness=6] (N\x\y); \fi
			}
		}
	\end{tikzpicture}
		\caption*{$X^+(\Z_4\times\Z_2, S_\ell)$}
	\end{subfigure}
\end{figure}

The irreducible characters of the group $G$ are given by
$$\chi_{(u,v)}(x,y) = i^{ux}(-1)^{vy}, \qquad u\in\{0,1,2,3\} \quad \text{and} \quad v\in\{0,1\}.$$

We have the following character table of the group.

\begin{table}[H]
	\centering
	\label{tab:caracteres_z4_z2}
	{\small 
		\renewcommand{\arraystretch}{1.2} 
	\begin{tabular}{c|cccccccc}
		\hline\hline
		$(x,y)$ & $(0,0)$ & $(1,0)$ & $(2,0)$ & $(3,0)$ & $(0,1)$ & $(1,1)$ & $(2,1)$ & $(3,1)$ \\ \hline
		$\chi_{(0,0)}$ & 1 & 1 & 1 & 1 & 1 & 1 & 1 & 1 \\
		$\chi_{(1,0)}$ & 1 & $i$ & $-1$ & $-i$ & 1 & $i$ & $-1$ & $-i$ \\
		$\chi_{(2,0)}$ & 1 & $-1$ & 1 & $-1$ & 1 & $-1$ & 1 & $-1$ \\
		$\chi_{(3,0)}$ & 1 & $-i$ & $-1$ & $i$ & 1 & $-i$ & $-1$ & $i$ \\ \hline
		$\chi_{(0,1)}$ & 1 & 1 & 1 & 1 & $-1$ & $-1$ & $-1$ & $-1$ \\
		$\chi_{(1,1)}$ & 1 & $i$ & $-1$ & $-i$ & $-1$ & $-i$ & 1 & $i$ \\
		$\chi_{(2,1)}$ & 1 & $-1$ & 1 & $-1$ & $-1$ & 1 & $-1$ & 1 \\
		$\chi_{(3,1)}$ & 1 & $-i$ & $-1$ & $i$ & $-1$ & $i$ & 1 & $-i$ \\ \hline\hline
	\end{tabular}}
	\caption{Character table of the group $G = \mathbb{Z}_4 \times \mathbb{Z}_2$.}
\end{table}

Now, we compute $\chi_{(u,v)}(S_\ell)$, $\chi_{(u,v)}(S_r)$ and $\chi_{(u,v)}(S_m)$ for $u=0,1,2,3$ and $v=0,1$.

For $u=0$ and $v=0$, we have that
	\begin{align*}
		\chi_{(0,0)}(S_\ell) &= \chi_{(0,0)}(0,1)+\chi_{(0,0)}(1,1)+\chi_{(0,0)}(3,0) = 1+1+1 = 3,\\
		\chi_{(0,0)}(S_r) &= \chi_{(0,0)}(0,1)+\chi_{(0,0)}(1,0)+\chi_{(0,0)}(3,1) = 1+1+1 = 3,\\
		\chi_{(0,0)}(S_m) &= \chi_{(0,0)}(0,1)+\chi_{(0,0)}(1,0)+\chi_{(0,0)}(1,1)+\chi_{(0,0)}(2,0) = 1+1+1+1 =4.
	\end{align*}	
	Then, $\lambda_{DX}^{\pm}(\chi_{0,0}) = \frac{3+3}{2} \pm \sqrt{\left(\frac{3-3}{2}\right)^2 + 4^2} = 3 \pm \sqrt{0 + 16} = 3 \pm 4$. 
	
For $u=1$ and $v=0$, we have that
	\begin{align*}
		\chi_{(1,0)}(S_\ell) &= \chi_{(1,0)}(0,1)+\chi_{(1,0)}(1,1)+\chi_{(1,0)}(3,0) = 1+i-i = 1,\\
		\chi_{(1,0)}(S_r) &= \chi_{(1,0)}(0,1)+\chi_{(1,0)}(1,0)+\chi_{(1,0)}(3,1) = 1+i-i = 1,\\
		\chi_{(1,0)}(S_m) &= \chi_{(1,0)}(0,1)+\chi_{(1,0)}(1,0)+\chi_{(1,0)}(1,1)+\chi_{(1,0)}(2,0) = 1+i+i-1 =2i.
	\end{align*}	
	Then, $\lambda_{DX}^{\pm}(\chi_{0,0}) = \frac{1+1}{2} \pm \sqrt{0 + (2i)^2} = 1 \pm \sqrt{-4} = 1 \pm 2i$.
	
For $u=2$ and $v=0$, we have that
	\begin{align*}
		\chi_{(2,0)}(S_\ell) &= \chi_{(2,0)}(0,1)+\chi_{(2,0)}(1,1)+\chi_{(2,0)}(3,0) = 1-1-1 = -1,\\
		\chi_{(2,0)}(S_r) &= \chi_{(2,0)}(0,1)+\chi_{(2,0)}(1,0)+\chi_{(2,0)}(3,1) = 1-1-1 = -1,\\
		\chi_{(2,0)}(S_m) &= \chi_{(2,0)}(0,1)+\chi_{(2,0)}(1,0)+\chi_{(2,0)}(1,1)+\chi_{(2,0)}(2,0) = 1-1-1+1 =0.
	\end{align*}	
	Then, $\lambda_{DX}^{\pm}(\chi_{0,0}) = \frac{-1-1}{2} \pm \sqrt{0 + 0^2} = -1$, with multiplicity $2$.
	
For $u=3$ and $v=0$, we have that
	\begin{align*}
		\chi_{(3,0)}(S_\ell) &= \chi_{(3,0)}(0,1)+\chi_{(3,0)}(1,1)+\chi_{(3,0)}(3,0) = 1-i+i = 1,\\
		\chi_{(3,0)}(S_r) &= \chi_{(3,0)}(0,1)+\chi_{(3,0)}(1,0)+\chi_{(3,0)}(3,1) = 1-i+i = 1,\\
		\chi_{(3,0)}(S_m) &= \chi_{(3,0)}(0,1)+\chi_{(3,0)}(1,0)+\chi_{(3,0)}(1,1)+\chi_{(3,0)}(2,0) = 1-i-i-1 =-2i.
	\end{align*}	
	Then, $\lambda_{DX}^{\pm}(\chi_{0,0}) = \frac{1+1}{2} \pm \sqrt{0 + (-2i)^2} = 1 \pm \sqrt{-4} = 1 \pm 2i$.
	
For $u=0$ and $v=1$, we have that
	\begin{align*}
		\chi_{(0,1)}(S_\ell) &= \chi_{(0,1)}(0,1)+\chi_{(0,1)}(1,1)+\chi_{(0,1)}(3,0) = -1-1+1 = -1,\\
		\chi_{(0,1)}(S_r) &= \chi_{(0,1)}(0,1)+\chi_{(0,1)}(1,0)+\chi_{(0,1)}(3,1) = -1+1-1 = -1,\\
		\chi_{(0,1)}(S_m) &= \chi_{(0,1)}(0,1)+\chi_{(0,1)}(1,0)+\chi_{(0,1)}(1,1)+\chi_{(0,1)}(2,0) = -1+1-1+1 =0.
	\end{align*}	
Then, $\lambda_{DX}^{\pm}(\chi_{0,0}) = \frac{-1-1}{2} \pm \sqrt{0 + 0^2} = -1$ with multiplicity $2$.
	
For $u=1$ and $v=1$, we have that
	\begin{align*}
		\chi_{(1,1)}(S_\ell) &= \chi_{(1,1)}(0,1)+\chi_{(1,1)}(1,1)+\chi_{(1,1)}(3,0) = -1-i-i = -1-2i,\\
		\chi_{(1,1)}(S_r) &= \chi_{(1,1)}(0,1)+\chi_{(1,1)}(1,0)+\chi_{(1,1)}(3,1) = -1+i+i = -1+2i,\\
		\chi_{(1,1)}(S_m) &= \chi_{(1,1)}(0,1)+\chi_{(1,1)}(1,0)+\chi_{(1,1)}(1,1)+\chi_{(1,1)}(2,0) = -1+i-i-1 =-2.
	\end{align*}	
	Then, $\lambda_{DX}^{\pm}(\chi_{0,0}) = \frac{(-1-2i)+(-1+2i)}{2} \pm \sqrt{(-2i)^2+(-2)^2} = -1$ with multiplicity $2$.
	
For $u=2$ and $v=1$, we have that
	\begin{align*}
		\chi_{(2,1)}(S_\ell) &= \chi_{(2,1)}(0,1)+\chi_{(2,1)}(1,1)+\chi_{(2,1)}(3,0) = -1+1-2 = -1,\\
		\chi_{(2,1)}(S_r) &= \chi_{(2,1)}(0,1)+\chi_{(2,1)}(1,0)+\chi_{(2,1)}(3,1) = -1-1+1 = -1,\\
		\chi_{(2,1)}(S_m) &= \chi_{(2,1)}(0,1)+\chi_{(2,1)}(1,0)+\chi_{(2,1)}(1,1)+\chi_{(2,1)}(2,0) = -1-1+1+1=0.
	\end{align*}	
	Then, $\lambda_{DX}^{\pm}(\chi_{0,0}) = \frac{-1-1}{2} \pm \sqrt{0 + 0^2} = -1$ with multiplicity $2$.
	
For $u=3$ and $v=1$, we have that
	\begin{align*}
		\chi_{(3,1)}(S_\ell) &= \chi_{(3,1)}(0,1)+\chi_{(3,1)}(1,1)+\chi_{(3,1)}(3,0) = -1+i+i = -1+2i,\\
		\chi_{(3,1)}(S_r) &= \chi_{(3,1)}(0,1)+\chi_{(3,1)}(1,0)+\chi_{(3,1)}(3,1) = -1-i-i = -1-2i,\\
		\chi_{(3,1)}(S_m) &= \chi_{(3,1)}(0,1)+\chi_{(3,1)}(1,0)+\chi_{(3,1)}(1,1)+\chi_{(3,1)}(2,0) = -1-i+i-1=-2.
	\end{align*}	
	Then, $\lambda_{DX}^{\pm}(\chi_{0,0}) = \frac{(-1+2i)+(-1-2i)}{2} \pm \sqrt{(2i)^2 + (-2)^2} = -1$, with multiplicity $2$.

Hence, the spectrum of the di-Cayley graph $DX(\Z_4\times \Z_2;S_\ell,S_r,S_m)$ is given by
$$Spec(\G_{di}) = \{[7]^1, [-1]^{11}, [1+2i]^2, [1-2i]^2\}.$$

Moreover, the spectrum of the bi-Cayley graph $DX(\Z_4\times \Z_2;S_\ell,S_r,S_m)$ is given by
$$Spec(\G_{bi}) = \{[7]^1,[3]^2, [-1]^{13}\}.$$

Now, we compute the spectrum of the di-Cayley sum graph $DX^+(\Z_4\times\Z_2;S_\ell,S_r,S_m)$ and the bi-Cayley sum graph $BX^+(\Z_4\times\Z_2;S_\ell,S_r,S_m)$, and we obtain
$$ Spec(\G_{di}^+) = Spec(\G_{bi}^+) = \{[7]^1,[3]^1,[1]^1,[-1]^8,[-3]^1,[2 \pm \sqrt{5}]^1, 
[-2 \pm \sqrt{5}]^1 
\}. $$
So, $\G_{di}$ is Gaussian integral, $\G_{bi}$ is integral and $\G_{di}^+=\G_{bi}^+$ has real non integral spectrum.
\hfill $\diamond$
\end{exam}

\begin{exam}[\textit{Non-mirror, abelian}] \label{exam: 6.5}
Consider the cyclic group $G=\mathbb{Z}_{12}$ and the subsets 
	$$ S_\ell = \{0,1\}, \qquad S_r = \{0,5\} \qquad \text{and} \qquad S_m = \{1,5\}. $$
	
The associated di-Cayley and bi-Cayley (sum) graphs are given by $\G_{di}^* = DX^*(G; S_\ell, S_r, S_m)$ and $\G_{bi}^* = BX^*(G; S_\ell, S_r, S_m)$. Since $G$ is abelian, the sum graphs $\G_{di}^+$ and $\G_{bi}^+$ are undirected. 
	
The graphical representations of these graphs and the corresponding Cayley graphs $X^*(G,S_\ell)$ and $X^*(G,S_r)$ are as follows:
	
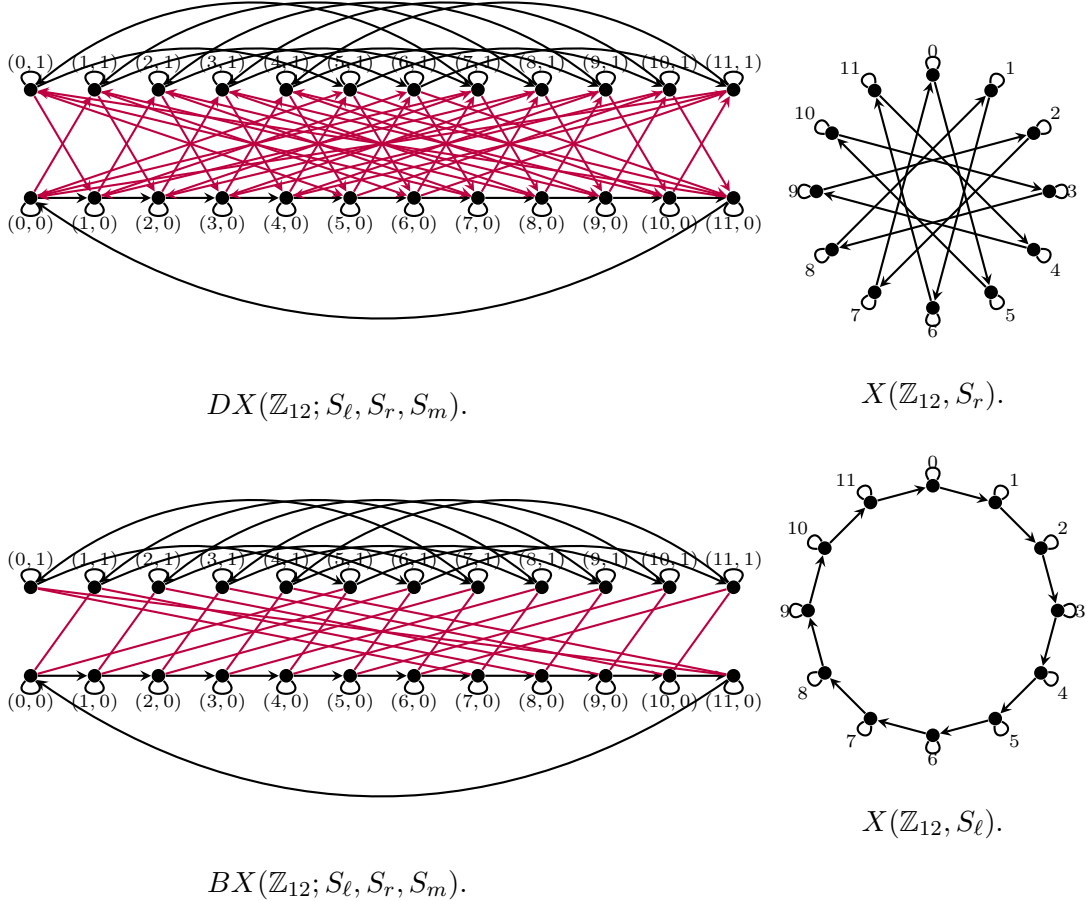
\begin{figure}[H]
	\centering
	\tikzset{
		cayley node/.style={circle, fill=black, inner sep=0pt, minimum size=1.8mm},
		edge S/.style={thick, color=black, ->, >=stealth},
		edge T/.style={thick, color=purple, ->, >=stealth},
		edge T bi/.style={thick, color=purple},
		font label/.style={font=\tiny}
	}
	
	\begin{minipage}[c]{0.6\textwidth}
		\centering
		\begin{tikzpicture}[scale=0.65, every node/.append] 
			\foreach \i in {0,...,11} {
				\node[cayley node, label={[font label]above:{$(\i,1)$}}] (M\i) at (\i*1.3, 2.2) {};
				\node[cayley node, label={[font label]below:{$(\i,0)$}}] (P\i) at (\i*1.3, 0) {};
			}
			
			\foreach \x in {0,...,11} {
				\draw[edge S, -] (P\x) to[out=225, in=315, looseness=6] (P\x);
				\pgfmathtruncatemacro{\nextone}{int(mod(\x + 1, 12))}
				\ifnum\nextone<\x
				\draw[edge S] (P\x) to[bend left=35] (P\nextone);
				\else
				\draw[edge S] (P\x) -- (P\nextone);
				\fi
			}
			
			\foreach \x in {0,...,11} {
				\draw[edge S, -] (M\x) to[out=45, in=135, looseness=6] (M\x);
				\pgfmathtruncatemacro{\nextfive}{int(mod(\x + 5, 12))}
				\ifnum\nextfive<\x
				\draw[edge S] (M\x) to[bend right=40] (M\nextfive);
				\else
				\draw[edge S] (M\x) to[bend left=25] (M\nextfive);
				\fi
			}
			
			\foreach \x in {0,...,11} {
				\pgfmathtruncatemacro{\mone}{int(mod(\x + 1, 12))}
				\pgfmathtruncatemacro{\mfive}{int(mod(\x + 5, 12))}
				
				\draw[edge T] (P\x) -- (M\mone);
				\draw[edge T] (P\x) -- (M\mfive);
				
				\draw[edge T] (M\x) -- (P\mone);
				\draw[edge T] (M\x) -- (P\mfive);
			}
		\end{tikzpicture}
		\caption*{$DX(\mathbb{Z}_{12};S_\ell,S_r,S_m)$.}
		
		\vspace{0.6cm}
		
		\begin{tikzpicture}[scale=0.65, every node/.append] 
			\foreach \i in {0,...,11} {
				\node[cayley node, label={[font label]above:{$(\i,1)$}}] (M\i) at (\i*1.3, 1.8) {};
				\node[cayley node, label={[font label]below:{$(\i,0)$}}] (P\i) at (\i*1.3, 0) {};
			}
			
			\foreach \x in {0,...,11} {
				\draw[edge S, -] (P\x) to[out=225, in=315, looseness=6] (P\x);
				\pgfmathtruncatemacro{\nextone}{int(mod(\x + 1, 12))}
				\ifnum\nextone<\x
				\draw[edge S] (P\x) to[bend left=35] (P\nextone);
				\else
				\draw[edge S] (P\x) -- (P\nextone);
				\fi
			}
			
			\foreach \x in {0,...,11} {
				\draw[edge S, -] (M\x) to[out=45, in=135, looseness=6] (M\x);
				\pgfmathtruncatemacro{\nextfive}{int(mod(\x + 5, 12))}
				\ifnum\nextfive<\x
				\draw[edge S] (M\x) to[bend right=40] (M\nextfive);
				\else
				\draw[edge S] (M\x) to[bend left=25] (M\nextfive);
				\fi
			}
			
			\foreach \x in {0,...,11} {
				\pgfmathtruncatemacro{\mone}{int(mod(\x + 1, 12))}
				\pgfmathtruncatemacro{\mfive}{int(mod(\x + 5, 12))}
				
				\draw[edge T bi] (P\x) -- (M\mone);
				\draw[edge T bi] (P\x) -- (M\mfive);
			}
		\end{tikzpicture}
		\caption*{$BX(\mathbb{Z}_{12};S_\ell,S_r,S_m)$.}
	\end{minipage}%
	\hfill
	\begin{minipage}[c]{0.35\textwidth}
		\centering
		
		\begin{tikzpicture}[scale=1.4]
			\foreach \i in {0,...,11} {
				\pgfmathtruncatemacro{\angle}{90 - \i * 30}
				\node[cayley node, label={[font label]{\angle}:{$\i$}}] (w\i) at ({\angle}:1.1) {};
			}
			\foreach \i in {0,...,11} {
				\pgfmathtruncatemacro{\ang}{90 - \i * 30}
				\draw[edge S, -] (w\i) to[out=\ang-35, in=\ang+35, looseness=6] (w\i);
				\pgfmathtruncatemacro{\nextfive}{mod(\i + 5, 12)}
				\draw[edge S] (w\i) -- (w\nextfive);
			}
		\end{tikzpicture}
		\caption*{$X(\mathbb{Z}_{12},S_r)$.}
		
		\vspace{0.4cm}
		
		\begin{tikzpicture}[scale=1.5]
			\foreach \i in {0,...,11} {
				\pgfmathtruncatemacro{\angle}{90 - \i * 30}
				\node[cayley node, label={[font label]{\angle}:{$\i$}}] (v\i) at ({\angle}:1.1) {};
			}
			\foreach \i in {0,...,11} {
				\pgfmathtruncatemacro{\ang}{90 - \i * 30}
				\draw[edge S, -] (v\i) to[out=\ang-35, in=\ang+35, looseness=6] (v\i);
				\pgfmathtruncatemacro{\next}{mod(\i + 1, 12)}
				\draw[edge S] (v\i) -- (v\next);
			}
		\end{tikzpicture}
		\caption*{$X(\mathbb{Z}_{12},S_\ell)$.}
				
	\end{minipage}
\end{figure}

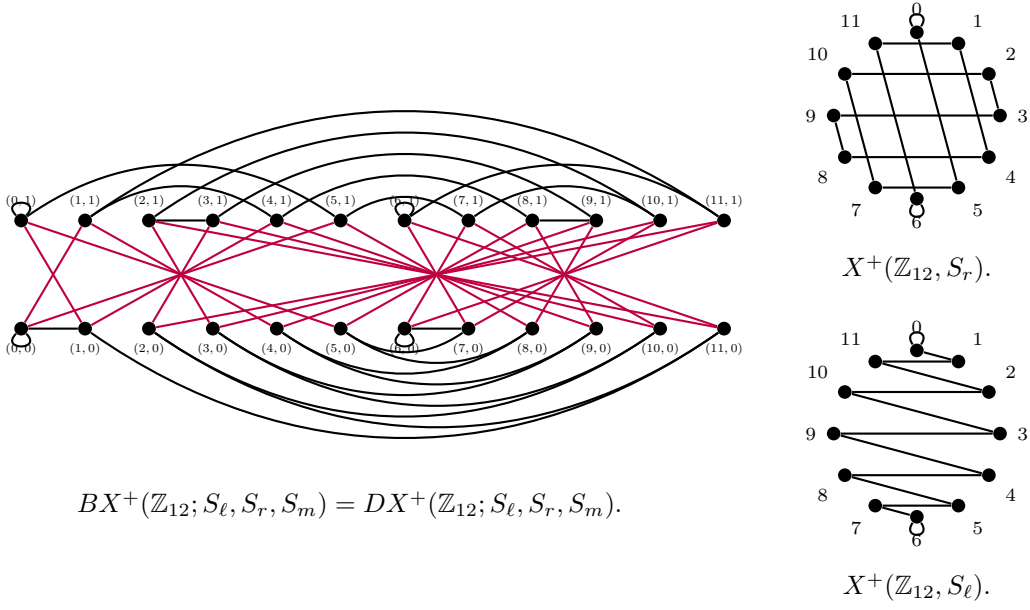
\vspace{-1cm}

\begin{figure}[H]
	\centering
	\tikzset{
		cayley node/.style={circle, fill=black, inner sep=0pt, minimum size=1.8mm},
		edge S/.style={thick, color=black},
		edge T bi/.style={thick, color=purple},
		font label/.style={font=\tiny}
	}
	
	\begin{subfigure}[c]{0.62\textwidth}
		\centering
		\begin{tikzpicture}[scale=0.65, every node/.append style={transform shape}, cayley node/.append style={minimum size=2.8mm}]
			\foreach \u in {0,...,11} {
				\node[cayley node, label={[font label]above:{$(\u,1)$}}] (B\u) at (\u*1.3, 2.2) {};
				\node[cayley node, label={[font label]below:{$(\u,0)$}}] (A\u) at (\u*1.3, 0) {};
			}
			
			\foreach \u in {0,...,11} {
				\foreach \s in {0,1} {
					\pgfmathtruncatemacro{\v}{mod(\s - \u + 12, 12)}
					\ifnum\u=\v
					\draw[edge S] (A\u) to[out=225, in=315, looseness=6] (A\u);
					\else
					\ifnum\u<\v
					\pgfmathtruncatemacro{\diff}{\v - \u}
					\ifnum\diff=1 
					\draw[edge S] (A\u) -- (A\v); 
					\else
					\draw[edge S] (A\u) to[bend right=35] (A\v); 
					\fi
					\fi
					\fi
				}
			}
			
			\foreach \u in {0,...,11} {
				\foreach \s in {0,5} {
					\pgfmathtruncatemacro{\v}{mod(\s - \u + 12, 12)}
					\ifnum\u=\v
					\draw[edge S] (B\u) to[out=45, in=135, looseness=6] (B\u);
					\else
					\ifnum\u<\v
					\pgfmathtruncatemacro{\diff}{\v - \u}
					\ifnum\diff=1 
					\draw[edge S] (B\u) -- (B\v); 
					\else
					\draw[edge S] (B\u) to[bend left=35] (B\v); 
					\fi
					\fi
					\fi
				}
			}
			
			\foreach \u in {0,...,11} {
				\foreach \t in {1,5} {
					\pgfmathtruncatemacro{\v}{mod(\t - \u + 12, 12)}
					\draw[edge T bi] (A\u) -- (B\v);
				}
			}
		\end{tikzpicture}
		\caption*{$BX^+(\mathbb{Z}_{12};S_\ell,S_r,S_m) = DX^+(\mathbb{Z}_{12};S_\ell,S_r,S_m)$.}
	\end{subfigure}%
	\hfill
	\begin{subfigure}[c]{0.38\textwidth}
		\centering
		
		\begin{tikzpicture}[scale=1]
			\foreach \i in {0,...,11} {
				\pgfmathtruncatemacro{\angle}{90 - \i * 30}
				\node[cayley node, label={[font label]{\angle}:{$\i$}}] (w\i) at ({\angle}:1.1) {};
			}
			
			\foreach \u in {0,...,11} {
				\foreach \s in {0,5} {
					\pgfmathtruncatemacro{\v}{mod(\s - \u + 12, 12)}
					\ifnum\u=\v
					\pgfmathtruncatemacro{\ang}{90 - \u * 30}
					\draw[edge S] (w\u) to[out=\ang-35, in=\ang+35, looseness=6] (w\u);
					\else
					\ifnum\u<\v
					\draw[edge S] (w\u) -- (w\v);
					\fi
					\fi
				}
			}
		\end{tikzpicture}
		\caption*{$X^+(\mathbb{Z}_{12},S_r)$.}
		
		\vspace{0.3cm}
	
		\begin{tikzpicture}[scale=1]
			\foreach \i in {0,...,11} {
				\pgfmathtruncatemacro{\angle}{90 - \i * 30}
				\node[cayley node, label={[font label]{\angle}:{$\i$}}] (v\i) at ({\angle}:1.1) {};
			}
			
			\foreach \u in {0,...,11} {
				\foreach \s in {0,1} {
					\pgfmathtruncatemacro{\v}{mod(\s - \u + 12, 12)}
					\ifnum\u=\v
					\pgfmathtruncatemacro{\ang}{90 - \u * 30}
					\draw[edge S] (v\u) to[out=\ang-35, in=\ang+35, looseness=6] (v\u);
					\else
					\ifnum\u<\v
					\draw[edge S] (v\u) -- (v\v);
					\fi
					\fi
				}
			}
		\end{tikzpicture}
		\caption*{$X^+(\mathbb{Z}_{12},S_\ell)$.}
				
	\end{subfigure}
\end{figure}

Notice that the cover Cayley graphs are isomorphic to directed cycles with loops while the cover Cayley sum graphs are isomorphic to paths with loops, more precisely we have 
	$$ X(\Z_{12}, S_\ell) \simeq X(\Z_{12}, S_r) \simeq  \mathring{\vec{C}}_{12} \qquad \text{and} \qquad X^+(\Z_{12}, S_\ell) \simeq X^+(\Z_{12}, S_r) \simeq\mathring P_{12}.$$
	
The irreducible characters of $\mathbb{Z}_{12}$ are given by 
	$$\chi_k(j) = \omega^{k \cdot j} = e^{\frac{2\pi i}{12} k \cdot j}, \qquad 0 \le k, j \le 11,$$ 
where $\omega = e^{\frac{\pi i}{6}}$. For each character $\chi_k$. We have the following character table of $\Z_{12}$:
	
	\begin{table}[H]
		\centering
		{\small
		\label{tab:caracteres_z12}
		\begin{tabular}{ccccccccccccc}
			\hline
			 & $0$ & $1$ & $2$ & $3$ & $4$ & $5$ & $6$ & $7$ & $8$ & $9$ & $10$ & $11$ \\ 
			\hline
			$\chi_0$ & $1$ & $1$ & $1$ & $1$ & $1$ & $1$ & $1$ & $1$ & $1$ & $1$ & $1$ & $1$ \\
			$\chi_1$ & $1$ & $\omega$ & $\omega^2$ & $i$ & $\omega^4$ & $\omega^5$ & $-1$ & $-\omega$ & $-\omega^2$ & $-i$ & $-\omega^4$ & $-\omega^5$ \\
			$\chi_2$ & $1$ & $\omega^2$ & $\omega^4$ & $-1$ & $-\omega^2$ & $-\omega^4$ & $1$ & $\omega^2$ & $\omega^4$ & $-1$ & $-\omega^2$ & $-\omega^4$ \\
			$\chi_3$ & $1$ & $i$ & $-1$ & $-i$ & $1$ & $i$ & $-1$ & $-i$ & $1$ & $i$ & $-1$ & $-i$ \\
			$\chi_4$ & $1$ & $\omega^4$ & $-\omega^2$ & $1$ & $\omega^4$ & $-\omega^2$ & $1$ & $\omega^4$ & $-\omega^2$ & $1$ & $\omega^4$ & $-\omega^2$ \\
			$\chi_5$ & $1$ & $\omega^5$ & $-\omega^4$ & $i$ & $-\omega^2$ & $\omega$ & $-1$ & $-\omega^5$ & $\omega^4$ & $-i$ & $\omega^2$ & $-\omega$ \\
			$\chi_6$ & $1$ & $-1$ & $1$ & $-1$ & $1$ & $-1$ & $1$ & $-1$ & $1$ & $-1$ & $1$ & $-1$ \\
			$\chi_7$ & $1$ & $-\omega$ & $\omega^2$ & $-i$ & $\omega^4$ & $-\omega^5$ & $-1$ & $\omega$ & $-\omega^2$ & $i$ & $-\omega^4$ & $\omega^5$ \\
			$\chi_8$ & $1$ & $-\omega^2$ & $\omega^4$ & $-1$ & $-\omega^2$ & $\omega^4$ & $1$ & $-\omega^2$ & $\omega^4$ & $-1$ & $-\omega^2$ & $\omega^4$ \\
			$\chi_9$ & $1$ & $-i$ & $-1$ & $i$ & $1$ & $-i$ & $-1$ & $i$ & $1$ & $-i$ & $-1$ & $i$ \\
			$\chi_{10}$ & $1$ & $-\omega^4$ & $-\omega^2$ & $1$ & $-\omega^4$ & $-\omega^2$ & $1$ & $-\omega^4$ & $-\omega^2$ & $1$ & $-\omega^4$ & $-\omega^2$ \\
			$\chi_{11}$ & $1$ & $-\omega^5$ & $-\omega^4$ & $-i$ & $-\omega^2$ & $-\omega$ & $-1$ & $\omega^5$ & $\omega^4$ & $i$ & $\omega^2$ & $\omega$ \\
			\hline
		\end{tabular}}
		\caption{Character table of $\mathbb{Z}_{12}$.}
	\end{table}

Now, we compute the character $\chi_j(S_\ell)$, $\chi_j(S_r)$ and $\chi_j(S_m)$ for $j=0,1,\dots,11$.
	
For $k = 0$, we have that $\chi_0(S_\ell) = 2$, $\chi_0(S_r) = 2$, $\chi_0(S_m) = 2$ and thus $\lambda \in \{4, 0\}$.
		
For $k = 1, 11$, we have that $\chi_1(S_\ell) = 1 + \frac{\sqrt{3}}{2} + \frac{1}{2}i$, $\chi_1(S_r) = 1 - \frac{\sqrt{3}}{2} + \frac{1}{2}i$, $\chi_1(S_m) = i$. 		
		This yields $\lambda \in \{1+i, 1\}$ (for $k=1$) and $\lambda \in \{1-i, 1\}$ (for $k=11$).
		
For $k = 2, 10$, we have that $\chi_2(S_\ell) = \frac{3}{2} + \frac{\sqrt{3}}{2}i$, $\chi_2(S_r) = \frac{3}{2} - \frac{\sqrt{3}}{2}i$, $\chi_2(S_m) = 1$. 
		This yields $\lambda \in \{2, 1\}$.
		
For $k = 3, 9$, we have that $\chi_3(S_\ell) = 1 + i$, $\chi_3(S_r) = 1 + i$, $\chi_3(S_m) = 2i$; and hence $\lambda \in \{1+3i, 1-i\}$ (for $k=3$) and $\lambda \in \{1-3i, 1+i\}$ (for $k=9$).
		
For $k = 4, 8$, we have that $\chi_4(S_\ell) = \frac{1}{2} + \frac{\sqrt{3}}{2}i$, $\chi_4(S_r) = \frac{1}{2} - \frac{\sqrt{3}}{2}i$, $\chi_4(S_m) = -1$; thus $\lambda \in \{1, 0\}$.
		
For $k = 5, 7$, we have that $\chi_5(S_\ell) = 1 - \frac{\sqrt{3}}{2} + \frac{1}{2}i$, $\chi_5(S_r) = 1 + \frac{\sqrt{3}}{2} + \frac{1}{2}i$, $\chi_5(S_m) = i$. 
This yields $\lambda \in \{1+i, 1\}$ (for $k=5$) and $\lambda \in \{1-i, 1\}$ (for $k=7$).
		
For $k = 6$, we have that $\chi_6(S_\ell) = 0$, $\chi_6(S_r) = 0$, $\chi_6(S_m) = -2$ and hence $\lambda \in \{2, -2\}$.
	
Collecting all 24 eigenvalues, we obtain the spectrum of $\G_{di}$:
	$$ Spec(\G_{di}) = \left\{ [4]^1, [2]^3, [1]^8, [0]^3, [-2]^1, [1 \pm 3i]^1, [1 \pm i]^3 \right\}. $$
	
In contrast, for the bi-Cayley graph $\G_{bi} = BX(\mathbb{Z}_{12}; S_\ell, S_r, S_m)$, the eigenvalues are given by
	$$ Spec(\G_{bi}) = \left\{ [4]^1, [2]^3, [1]^4, [0]^3, [-2]^1, [3\pm i]^1, [-1\pm i]^1, [1+\tfrac{\sqrt{7}}{2}\pm \tfrac{i}{2}]^2, [1-\tfrac{\sqrt{7}}{2}\pm \tfrac{i}{2}]^2 \right\}. $$
We check that their sum is 24 in coincidence with the number of loops.

Finally, we compute the spectrum of the di-Cayley sum graph $\G_{di}^+ = DX^+(\mathbb{Z}_{12}; S_\ell, S_r, S_m)$. The spectrum is given by
	$$ Spec(\G_{di}^{+}) = \left\{ [4]^1, [\pm 2]^2, 
[0]^7, [\pm \sqrt{6}]^2, 
[1 \pm \sqrt{3}]^1, [-1 \pm \sqrt{3}]^1, [2 \pm \sqrt{2}]^1, [-2 \pm \sqrt{2}]^1 \right\},$$
whose sum is 4 which is the number of loops.
\hfill $\diamond$
\end{exam}


\begin{exam}[\textit{Non-mirror, non-abelian}] \label{exam: 6.6}
Consider the dihedral group of order $8$,
	$$\mathbb{D}_4=\langle r,s \mid r^4=1,\ s^2=1,\ srs=r^{-1}\rangle.$$
Identifying $r$ with $(1,0)$ and $s$ with $(0,1)$, we consider the subsets

$$	S_\ell = \{(1,0),(3,0)\}, \qquad S_r = \{(0,0),(2,0)\}, \qquad S_m = \{(0,1),(2,1),(1,1),(3,1)\}. $$

Here, since all the characters are real (see the table below) and the connection sets are normal and symmetric, the four graphs $DX^*(\mathbb{D}_4;S_\ell,S_r,S_m)$,  $BX^*(\mathbb{D}_4;S_\ell, S_r, S_m)$, $DX^+(\mathbb{D}_4;S_\ell,S_r,S_m)$ and $BX^+(\mathbb{D}_4;S_\ell,S_r,S_m)$ coincide. 

By Theorem \ref{thm: Spec DX}, the eigenvalues of the di-Cayley graph $DX(\mathbb{D}_4;S_\ell,S_r,S_m)$ are given by expression \eqref{eq: autovalores DX} 
where $\chi$ runs over the irreducible characters of $\mathbb{D}_4$. These characters are given in the following table:
	$$
	\begin{array}{c|ccccc}
		& \{(0,0)\}
		& \{(2,0)\}
		& \{(1,0),(3,0)\}
		& \{(0,1),(2,1)\}
		& \{(1,1),(3,1)\}
		\\
		\hline
		\chi_0
		& 1 & 1 & 1 & 1 & 1
		\\
		\chi_1
		& 1 & 1 & 1 & -1 & -1
		\\
		\chi_2
		& 1 & 1 & -1 & 1 & -1
		\\
		\chi_3
		& 1 & 1 & -1 & -1 & 1
		\\
		\chi_4
		& 2 & -2 & 0 & 0 & 0
	\end{array}
	$$
	
Now, we compute $\chi_i(S_\ell)$, $\chi_i(S_r)$ and $\chi_i(S_m)$ for $i=0,1,2,3,4$.
For $i=0$, we have that
		\begin{align*}
			\chi_0(S_\ell) &= \chi_0(1,0)+\chi_0(3,0) = 1+1 = 2,\\
			\chi_0(S_r) &= \chi_0(0,0)+\chi_0(2,0)= 1+1 = 2,\\
			\chi_0(S_m) &= \chi_0(0,1)+\chi_0(2,1)+\chi_0(1,1)+\chi_0(3,1) = 1+1+1+1 = 4.
		\end{align*} 
Then, $\lambda_{DX}^{\pm}(\chi_0) = \frac{2+2}{2} \pm \sqrt{\frac{2-2}{2}^2+4^2} = 2 \pm 4.$
For $i=1$, we have that
		\begin{align*}
			\chi_1(S_\ell) &= \chi_1(1,0)+\chi_1(3,0) = 1+1 = 2,\\
			\chi_1(S_r) &= \chi_1(0,0)+\chi_1(2,0)= 1+1 = 2,\\
			\chi_1(S_m) &= \chi_1(0,1)+\chi_1(2,1)+\chi_1(1,1)+\chi_1(3,1) = -1-1-1-1 = -4.
		\end{align*} 
Then, 
$\lambda_{DX}^{\pm}(\chi_1) = \frac{2+2}{2} \pm \sqrt{\frac{2-2}{2}^2+(-4)^2} = 2 \pm 4.$
For $i=2$, we have that
		\begin{align*}
			\chi_2(S_\ell) &= \chi_2(1,0)+\chi_2(3,0) = -1-1 = -2,\\
			\chi_2(S_r) &= \chi_2(0,0)+\chi_2(2,0)= 1+1 = 2,\\
			\chi_2(S_m) &= \chi_2(0,1)+\chi_2(2,1)+\chi_2(1,1)+\chi_2(3,1) = 1+1-1-1 = 0.
		\end{align*} 
Then, 
$\lambda_{DX}^{\pm}(\chi_2) = \frac{-2+2}{2} \pm \sqrt{\frac{-2-2}{2}^2+(0)^2} = \pm 2.$
For $i=3$, we have that
		\begin{align*}
			\chi_3(S_\ell) &= \chi_3(1,0)+\chi_3(3,0) = -1-1 = -2,\\
			\chi_3(S_r) &= \chi_3(0,0)+\chi_3(2,0)= 1+1 = 2,\\
			\chi_3(S_m) &= \chi_3(0,1)+\chi_3(2,1)+\chi_3(1,1)+\chi_3(3,1) = -1-1+1+1 = 0.
		\end{align*} 
Then, 
$\lambda_{DX}^{\pm}(\chi_3) = \frac{-2+2}{2} \pm \sqrt{\frac{-2-2}{2}^2+(0)^2} = \pm 2.$
For $i=4$, we have that
		\begin{align*}
			\chi_4(S_\ell) &= \chi_4(1,0)+\chi_4(3,0) = 0+0 = 0,\\
			\chi_4(S_r) &= \chi_4(0,0)+\chi_4(2,0)= 2-2 = 0,\\
			\chi_4(S_m) &= \chi_4(0,1)+\chi_4(2,1)+\chi_4(1,1)+\chi_4(3,1) = 0+0+0+0 = 0.
		\end{align*} 
Then, 
$\lambda_{DX}^{\pm}(\chi_4) = \frac{0+0}{2} \pm \sqrt{\frac{0-0}{2}^2+(0)^2} = 0$, with multiplicity $8$ (note that the irreducible character $\chi_4$ has dimension $2$).
	
Hence, the spectrum of the di-Cayley graph $DX(\mathbb{D}_4;S_\ell,S_r,S_m)$ is
	$$Spec(DX(\mathbb{D}_4;S_\ell,S_r,S_m))=\{[6]^2,[2]^2,[0]^8,[-2]^4\}.$$
Hence, the four graphs coincide and have integral spectrum.

The geometric representation of these graphs is shown below.
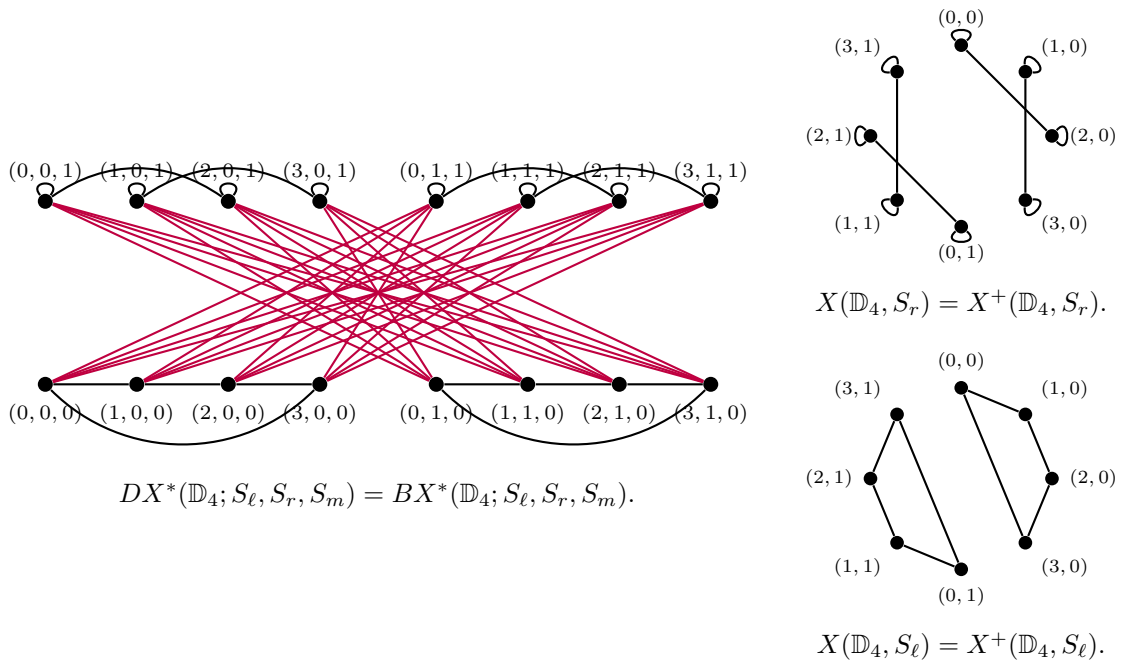
\begin{figure}[H]
	\centering
	\tikzset{
		cayley node/.style={circle, fill=black, inner sep=0pt, minimum size=1.8mm},
		edge S/.style={thick, color=black},
		edge T bi/.style={thick, color=purple},
		font label/.style={font=\tiny}
	}
	
	\begin{subfigure}[c]{0.67\textwidth}
		\centering
		\begin{tikzpicture}[scale=1.1, every node/.append style={transform shape}]
			\foreach \i in {0,...,3} {
				\node[cayley node, label={[font label]above:{$(\i,0,1)$}}] (M\i) at (\i*1.1, 2.2) {};
				\node[cayley node, label={[font label]below:{$(\i,0,0)$}}] (P\i) at (\i*1.1, 0) {};
			}
			\foreach \i in {0,...,3} {
				\pgfmathtruncatemacro{\idx}{\i + 4}
				\node[cayley node, label={[font label]above:{$(\i,1,1)$}}] (M\idx) at (4.7 + \i*1.1, 2.2) {};
				\node[cayley node, label={[font label]below:{$(\i,1,0)$}}] (P\idx) at (4.7 + \i*1.1, 0) {};
			}
			
			\draw[edge S] (P0) -- (P1);
			\draw[edge S] (P1) -- (P2);
			\draw[edge S] (P2) -- (P3);
			\draw[edge S] (P3) to[bend left=45] (P0);
			
			\draw[edge S] (P5) -- (P4);
			\draw[edge S] (P6) -- (P5);
			\draw[edge S] (P7) -- (P6);
			\draw[edge S] (P4) to[bend right=45] (P7);
			
			\foreach \x in {0,...,7} {
				\draw[edge S] (M\x) to[out=50, in=130, looseness=5] (M\x);
			}
			\draw[edge S] (M0) to[bend left=35] (M2); 
			\draw[edge S] (M1) to[bend left=35] (M3);
			\draw[edge S] (M4) to[bend left=35] (M6); 
			\draw[edge S] (M5) to[bend left=35] (M7);
			
			\foreach \x in {0,...,3} {
				\foreach \y in {4,...,7} {
					\draw[edge T bi] (P\x) -- (M\y);
					\draw[edge T bi] (P\y) -- (M\x);
				}
			}
		\end{tikzpicture}
		\caption*{$DX^*(\mathbb{D}_4;S_\ell,S_r,S_m) = BX^*(\mathbb{D}_4;S_\ell,S_r,S_m)$.}
\end{subfigure}%
\hfill
\begin{subfigure}[c]{0.3\textwidth}
	\centering
	\begin{tikzpicture}[scale=1.0]
		\foreach \i in {0,...,3} {
			\pgfmathtruncatemacro{\angle}{90 - \i * 45}
			\node[cayley node, label={[font label]{\angle}:{$(\i,0)$}}] (P\i) at ({\angle}:1.2) {};
		}
		\foreach \i in {0,...,3} {
			\pgfmathtruncatemacro{\angle}{270 - \i * 45}
			\node[cayley node, label={[font label]{\angle}:{$(\i,1)$}}] (M\i) at ({\angle}:1.2) {};
		}
		
		\foreach \i in {0,...,3} {
			\pgfmathtruncatemacro{\angleA}{90 - \i * 45}
			\draw[edge S] (P\i) to[out=\angleA-50, in=\angleA+50, looseness=5] (P\i);
			\pgfmathtruncatemacro{\angleB}{270 - \i * 45}
			\draw[edge S] (M\i) to[out=\angleB-50, in=\angleB+50, looseness=5] (M\i);
		}
		\draw[edge S] (P0) -- (P2); \draw[edge S] (P1) -- (P3);
		\draw[edge S] (M0) -- (M2); \draw[edge S] (M1) -- (M3);
	\end{tikzpicture}
	\caption*{$X(\mathbb{D}_4,S_r) = X^+(\mathbb{D}_4,S_r)$.}
	
	\vspace{0.25cm}
	
	\begin{tikzpicture}[scale=1.0]
		\foreach \i in {0,...,3} {
			\pgfmathtruncatemacro{\angle}{90 - \i * 45}
			\node[cayley node, label={[font label]{\angle}:{$(\i,0)$}}] (P\i) at ({\angle}:1.2) {};
		}
		\foreach \i in {0,...,3} {
			\pgfmathtruncatemacro{\angle}{270 - \i * 45}
			\node[cayley node, label={[font label]{\angle}:{$(\i,1)$}}] (M\i) at ({\angle}:1.2) {};
		}
		
		\foreach \i in {0,...,3} {
			\pgfmathtruncatemacro{\next}{mod(\i + 1, 4)}
			\draw[edge S] (P\i) -- (P\next);
			\draw[edge S] (M\i) -- (M\next);
		}
	\end{tikzpicture}
	\caption*{$X(\mathbb{D}_4,S_\ell) = X^+(\mathbb{D}_4,S_\ell)$.}
\end{subfigure}
\end{figure}

This graph is connected with loops. However, the Cayley (sum) covers are disconnected. In fact, $X(\mathbb{D}_4,S_\ell)=X^+(\mathbb{D}_4,S_\ell)$ is the union of two $C_4$'s and $X(\mathbb{D}_4,S_r)=X^+(\mathbb{D}_4,S_r)$ is the union of four $\mathring{P}_2$'s.
\hfill $\diamond$ 
\end{exam}

We conclude this work with a summary of the examples in this section.
\begin{table}[H]
	\centering
	\small
	\renewcommand{\arraystretch}{1.2}
	\begin{tabular}{cccccc}
		\hline
		\textbf{Example} & \textbf{Group $G$} & \textbf{MDCG} & \textbf{Covers} & \textbf{Directedness} & \textbf{Spectrum} \\ \hline
		\ref{exam: 6.1} & $\mathbb{Z}_6$ & yes & connected & mixed (undir. covers) & $\mathbb{Z}[\sqrt{3}], \mathbb{Z}[\sqrt{3}i]$ \\
		\ref{exam: 6.2} & $\mathbb{Z}_{15}$ & yes & disconnected & undirected & $\mathbb{Z}$ \\
		\ref{exam: 6.3} & $\mathbb{Q}_{12}$ & yes & disconnected & mixed (undir. covers) & $\mathbb{Z}, \mathbb{Z}[i]$ \\
		\ref{exam: 6.4} & $\mathbb{Z}_4 \times \mathbb{Z}_2$ & no & connected & mixed / undirected sums & $\mathbb{Z}, \mathbb{Z}[i], \mathbb{Z}[\sqrt{5}]$ \\
		\ref{exam: 6.5} & $\mathbb{Z}_{12}$ & no & connected & mixed (undir. covers) & $\mathbb{Z}[i], \mathbb{Z}[\sqrt{2}, \sqrt{3}], \mathbb{Z}[\frac{\sqrt 7}2,i]$,  \\
		\ref{exam: 6.6} & $\mathbb{D}_4$ & no & disconnected & mixed / undirected sums & $\mathbb{Z}$ \\ \hline
	\end{tabular}
	\caption{Summary of explicit bi/di-Cayley (sum) graphs.}
	\label{tab: summary examples}
\end{table}
In the table, the expression `mixed with undirected covers' has to be understood that the edges between covers are all directed. 
In the last column we just mention the rings where the different spectra lie.

\newpage 

\begin{thebibliography}{XXX}
	\bibitem{AT}
	\textsc{M.\@ Arezoomand, B.\@ Taeri}.
	\textit{On the automorphism group of bi-Cayley graphs over finite groups}. 
	Des.\@ Codes Cryptogr.\@ \textbf{70:3}, 251--260 (2013).
	
	\bibitem{Babai}
	\textsc{L.\@ Babai}.
	\textit{Spectra of Cayley graphs}. 
	J.\@ Comb.\@ Theory, Ser.\@ B \textbf{27}, 180--189 (1979).
			
	\bibitem{Bis}
	\textsc{A.\@ Biswas, J.P.\@ Saha}.
	\textit{Expansion in Cayley graphs, Cayley sum graphs and their twists}, 
	arXiv:2103.05935, (2021).
	
	\bibitem{ChP} 
	\textsc{P.M.\@ Chiapparoli, R.A.\@ Podestá}.
	\textit{Isospectral Cayley graphs with even and odd spectrum}. 	\newline
	arXiv:2601.05510, Jan 2026.

%
	
	\bibitem{DVGM}
	\textsc{M.\@ DeVos, L.\@ Goddyn, B.\@ Mohar, R.\@ \v{S}{\'a}mal}.
	\textit{Cayley sum graphs and eigenvalues of $(3,6)$-fullerenes}. 
	J.\@ Comb.\@ Theory, Ser.\@ B {\bf 99:2}, 358--369 (2009).
	
		
	\bibitem{GL1}
	\textsc{H.\@ Gao, Y.\@ Luo}. 
	\textit{The isomorphism problem for bi-Cayley graphs}. 
	Journal of Combinatorial Optimization \textbf{19:2}, 150--161 (2010). 
	
	\bibitem{GL}
	\textsc{X.\@ Gao, Y.\@ Luo}.
	\textit{The spectrum of semi-Cayley graphs over abelian groups}.
	Linear Algebra Appl.\@ \textbf{432:11}, 2974--2983 (2010).
	

	\bibitem{KK}
	\textsc{I.\@ Kovács, B.\@ Kuzman}.
	\textit{The classification of bi-Cayley graphs over abelian groups}. 
	Eur.\@ J.\@ Comb.\@ \textbf{46}, 122--136 (2015).

%
%

	\bibitem{KMS}
	\textsc{I.\@ Kovács, M.\@ Muzychuk, G.\@ Somlai}. 
	\textit{The isomorphism problem of bi-Cayley graphs over cyclic groups}. 
	Combinatorica \textbf{33:4}, 435--460 (2013).
	
	\bibitem{LF}
	\textsc{H.L.\@ Liu, Y.Q.\@ Feng}. 
	\textit{On the isomorphism of bi-Cayley graphs}. 
	Communications in Algebra \textbf{39:7}, 2452--2467 (2011).
	
	\bibitem{McW-M}
	\textsc{F.J.\@ MacWilliams, H.B.\@ Mann}.
	\textit{On the p-rank of the design matrix of a difference set}. 
	Inf.\@ Control \textbf{12}, 474--488 (1968).
	
	\bibitem{Marusic} 
	\textsc{D.\@ Marušič}.
	\textit{On vertex-transitive graphs of order $p^n$}. 
	Discrete Math.\@ \textbf{67:3}, 313--318 (1987).

	\bibitem{Marusic2} 
	\textsc{D.\@ Marušič}.
	\textit{Strongly regular bicirculants and tricirculants}. 
	Ars Comb.\@ \textbf{25C}, 11--15 (1988).
	
	\bibitem{PV} 
	\textsc{R.A.\@ Podestá, D.E.\@ Videla}.
	\textit{Integral equienergetic non-isospectral unitary Cayley graphs}. 
	Linear Algebra and its Applications \textbf{612}, 42--74 (2021). 
	
%
%
%
%

	\bibitem{RJ}
	\textsc{M.J.\@ Resmini, D.\@ Jungnickel}. 
	\textit{On semi-Cayley graphs}. 
	Journal of Algebraic Combinatorics \textbf{2:4}, 395--405 (1993).

	\bibitem{ZF}
	\textsc{J.X.\@ Zhou, Y.Q.\@ Feng}.
	\textit{The automorphisms of bi-Cayley graphs}. 
	J.\@ Comb.\@ Theory, Ser.\@ B \textbf{116}, 504--532 (2016).	

	\bibitem{Z}
	\textsc{P-H.\@ Zieschang}.
	\textit{Cayley graphs of finite groups}.
	J.\@ Algebra \textbf{118:2}, 447--454 (1988).
	
\end{thebibliography}
\end{document}